\documentclass[10pt,reqno]{amsart}
\usepackage{latexsym}
\usepackage{leftidx}
\usepackage{amsmath,amsthm, color, pdfpages}
\usepackage{enumerate}
\usepackage{amssymb}
\usepackage{upgreek}
\usepackage[margin=1in,dvips]{geometry}
\usepackage{tikz}

\def\f12{\frac 1 2}
\def\B{\mathcal{B}}
\def\a{\alpha}
\def\b{\beta}
\def\ga{\gamma}
\def\Ga{\Gamma}

\def\ep{\epsilon}

\def\La{\Lambda}
\def\si{\sigma}

\def\om{\omega}
\def\Om{\Omega}

\def\Lb{\underline{L}}

\def\ab{\underline{\a}}

\def\pa{\partial}
\def\les{\lesssim}
\def\cL{{\mathcal L}}
\def\cM{\mathcal{M}} 
\def\lg{\mathfrak{g}} 
\def\cT{\mathcal{T}} 
\def\cG{\mathbf{G}}

\def\f12{\frac 1 2}
\def\div{\text{div}}

\newcommand{\vol}{\textnormal{vol}}
\newcommand{\lap}{\mbox{$\Delta \mkern-13mu /$\,}}
\newcommand{\D}{\mbox{$D \mkern-13mu /$\,}}
\newcommand{\J}{\mbox{$J \mkern-13mu /$\,}}
\newcommand{\nabb}{\mbox{$\nabla \mkern-13mu /$\,}}

\newtheorem{Thm}{Theorem}[section]

\newtheorem{Prop}{Proposition}[section]

\newtheorem{Lem}{Lemma}[section]
\newtheorem{cor}{Corollary}[section]
\newtheorem{Remark}{Remark}[section]
\theoremstyle{definition}

\begin{document}

\title{On the global dynamics of Yang-Mills-Higgs equations}
\date{}

\author[Donyi Wei]{Dongyi Wei}
\address{School of Mathematical Sciences, Peking University, Beijing, China\\ Beijing, China}
\email{jnwdyi@pku.edu.cn}

\author[Shiwu Yang]{Shiwu Yang}
\address{Beijing International Center for Mathematical Research, Peking University\\ Beijing, China}
\email{shiwuyang@math.pku.edu.cn}

\author{Pin Yu}
\address{Department of Mathematics and Yau Mathematical Sciences Center, Tsinghua University\\ Beijing, China}
\email{yupin@mail.tsinghua.edu.cn}

\maketitle

\begin{abstract}
We study solutions to the Yang-Mills-Higgs equations on the maximal Cauchy development of the data given on a ball of radius $R$ in $\mathbb{R}^3$. The energy of the data could be infinite and  the solution grows at most inverse polynomially in $R-t$ as $t\rightarrow R$. As applications, we derive pointwise decay estimates for Yang-Mills-Higgs fields in the future of a hyperboloid or in the Minkowski space  $\mathbb{R}^{1+3}$ for data bounded in the weighted energy space with weights $|x|^{1+\epsilon}$. Moreover, for the abelian case of Maxwell-Klein-Gordon system, we extend the small data result of Lindblad and Sterbenz \cite{LindbladMKG} to general large data (under same assumptions but without any smallness). The proof is gauge independent and it is based on the framework of Eardley and Moncrief \cite{Moncrief1, Moncrief2} together with the geometric Kirchhoff-Sobolev parametrix constructed by Klainerman and Rodnianski \cite{Kl:paramatrix}. The new ingredient is a class of weighted energy estimates through backward light cones adapted to the initial data.
\end{abstract}

\section{Introduction}

\subsection{Yang-Mills-Higgs equations}
Let $\beta: \cG\rightarrow \mathbf{U}(V)$ be a finite dimensional unitary representation of a compact Lie group $\cG$.  We use $\langle \cdot ,\cdot \rangle_{_V}$ to denote the Hermitian inner product on $V$.  The Lie algebra $\mathfrak{g}$ of $\cG$ is also endowed with a bi-invariant inner product $\langle \cdot , \cdot \rangle_{\lg}$. We use $\mathcal{M}$ to denote a region in the Minkowski spacetime $\mathbb{R}^{1+3}$ with flat metric $(m_{\mu\nu})=\textnormal{diag}(-1, 1, 1, 1)$ in the Cartesian coordinates $(t, x^1, x^2, x^3)$. Let $\mathcal{P}$ be a principal $\cG$-bundle over $\cM$ and $\mathcal{P}\times_\beta V$ be the vector bundle associated to the representation $\beta$. A connection on $\mathcal{P}$ can be realized by a $\lg$-valued $1$-form $A$ on $\cM$. Given a section $\varphi$ of $\mathcal{P}\times_\beta V$, the covariant derivative associated to $A$ acts on $\varphi$ via
\[
D_\mu \varphi=\pa_\mu\varphi+A_{\mu}\cdot \varphi,
\]
where $A_{\mu}\cdot \varphi$ is the induced Lie algebra action of $\beta$.\footnote{We have $A_{\mu} \cdot \phi=\beta_*(A_{\mu})(\phi)$, where $\beta_*:\lg\rightarrow \mathfrak{u}(V)$ is the tangent map of $\beta$.} Similarly for $\lg$-valued function $\phi$, the covariant derivative is defined by
\[
D_\mu \phi=\pa_\mu \phi+[A_\mu, \phi].
\]
The Yang-Mills field $F$, i.e., the curvature of the connection $A$, is the $\lg$-valued $2$-form given by
\begin{equation*}
  F_{\mu\nu}=[D_{\mu},D_{\nu}]=\pa_{\mu}A_{\nu}-\pa_{\nu}A_{\mu}+[A_{\mu}, A_{\nu}].
\end{equation*}
The covariant derivative $D_{\mu}$ is compatible with $\langle \cdot ,\cdot \rangle_{\lg}$ and $\langle \cdot ,\cdot \rangle_{_V}$ in the following sense:
\[
\pa_\mu \langle\varphi_1, \varphi_2\rangle=\langle D_\mu \varphi_1, \varphi_2\rangle+\langle \varphi_1, D_\mu\varphi_2\rangle,
\]
where $\varphi_1$ and $\varphi_2$ are $\lg$-valued sections (such as connection $1$-forms or Yang-Mills fields) or sections of $\mathcal{P}\times_\beta V$.  We will drop the subscript $\lg$ or $V$ for $\langle \cdot ,\cdot \rangle{\lg}$ or  $\langle \cdot ,\cdot \rangle_{_V}$ since there will be no ambiguity in the context.

The Yang-Mills-Higgs theory plays a crucial role in modern physics. The Yang-Mills-Higgs (YMH) equation is a system of equations for the connection field $A$ and the $V$-valued scalar field $\phi$:
\begin{equation}
 \label{eq:YMH}\tag{YMH}
\begin{cases}
D^\nu F_{\nu\mu}=J[\phi]_{\mu}=\langle D_{\mu}\phi, \mathcal{T}\phi\rangle_{_V}+\langle \mathcal{T}\phi, D_{\mu}\phi\rangle_{_V},\\
D^\mu D_\mu\phi=\Box_A\phi=0,
\end{cases}
\end{equation}
where $\cT$ is the linear operator $\cT: V\rightarrow \lg \otimes V $ so that $\langle a\cdot v, w\rangle_{_V}=\langle a, \langle \cT v, w\rangle_{_V}\rangle_{\lg}$ for all $a\in \lg$, $v, w\in V$.
The system (YMH) can be derived as the Euler-Lagrange equations for the action
\[
L[A, \phi]=\int_{\mathbb{R}^{1+3}} \frac{1}{4} \langle F_{\mu\nu}, F^{\mu\nu}\rangle_{\lg}+ \langle D_{\mu}\phi, D^{\mu}\phi \rangle_{_V}  dxdt .
\]
The Yang-Mills-Higgs system is gauge invariant. A change of the gauge can be characterized by a smooth section $b:\cM\rightarrow G$ of $\mathcal{P}$. If we regard the group $G$ as a subgroup of $\mathbf{U}(V)$ (via the representation $\beta$, hence as a matrix group), then the gauge transformation of $b$ sends the connection field $A$, the Yang-Mills field $F$ and the scalar field $\phi$ to
\begin{align*}
  A \mapsto b^{-1}A b+b^{-1}db,\quad  F\mapsto b^{-1}F b,\quad \phi\mapsto \beta(b)\phi.
\end{align*}
 The gauge invariance of (YMH) means that if $(A, \phi)$ solves \eqref{eq:YMH}, then $(b^{-1}A b+b^{-1}db, \beta(b)\phi)$ is also a solution for all $b$.


The Cauchy problem of (YMH) has been studied extensively. In the seminal works \cite{Moncrief1,Moncrief2}, Eardley-Moncrief showed that sufficiently regular initial data lead to the existence of globally in time solutions to \eqref{eq:YMH}. Generalizations and extensions of this classical result are made primarily on relaxing the regularity (or equivalently differentiability) of the initial data.  Goganov-Kapitanski\u{\i} in  \cite{Kapitanskii85:YMH} proved that the solution still exists globally with the physically interested boundary condition $\lim\limits_{|x|\rightarrow\infty}|\phi|=1$. By introducing the celebrated bilinear estimates for null forms, Klainerman-Machedon in \cite{YMkl} constructed global solutions with finite energy data, see different proofs in \cite{sungjinYM}, \cite{Tesfahun15:YM:3D}, \cite{Sigmund16:YM:3D}. This remarkable global existence result could be viewed as a local well-posedness result with rough initial data due to the conservation of energy. However the conservation of energy is not necessary for proving such a global existence result.  Chru\'sciel-Shatah in \cite{Shatah:YM:generalManifold} (also see \cite{Sari16:YM:glob}) extended Eardley-Moncrief's result to any globally hyperbolic four-dimensional Lorentzian manifold.


The equation \eqref{eq:YMH} has a scaling symmetry
\begin{align*}
  A \mapsto \lambda A(\lambda t, \lambda x),\quad \phi \mapsto \lambda \phi(\lambda t, \lambda x),
\end{align*}
which makes it energy subcritical in three space dimension and energy critical in four space dimension.
From the PDE point of view, it is believed that the system is locally well-posedness for data smoother than the critical regularity, which has been confirmed for the abelian case of Maxwell-Klein-Gordon equations (see for example \cite{KriegerMKG4}, \cite{MKGmachedon}, \cite{OhMKG4},  \cite{MKGigor}) or in higher dimensions \cite{Klainerman:YM:4D:H1}. For the general non-abelian case of \eqref{eq:YMH}, the threshold conjecture for the energy critical case in four space dimension has recently been solved by Oh-Tataru in \cite{Oh21:YM:4D}. Local solutions with data below energy space in three space dimension have been constructed in  \cite{Tesfahun15:YM:belowE}, \cite{Sigmund16:YM:3D},  \cite{tao03:YM:temporal:belowE},
\cite{Pecher21:YM:3D:lowr},
\cite{Pecher19:YM:3D:infenergy}. However the regularity required in these works is still not sharp.

Another central topic of the field is to understand the long time dynamics of the solutions to \eqref{eq:YMH}. In the early works
\cite{Glassey79:YMH:energydecay}, \cite{Strauss79:YM:decay} of Glassey-Strauss, they obtained local energy decay for data bounded in the conformal energy class (weighted energy norm with weights $|x|^{2}$).
Sharp pointwise asymptotic decay estimates have been shown for spherically symmetric solutions in
\cite{Pedro:YM:sph:large}, \cite{Sari18:YM:Schwarz}, \cite{YM:spher:Pedro}. These results rely on the conservation of conformal energy, which is a direct consequence of the conformal invariance structure of \eqref{eq:YMH} in $\mathbb{R}^{1+3}$. The conformal mapping method is the geometric way to exploit this conformal invariance. It is widely used to study the long time dynamics of solutions to linear and nonlinear fields
\cite{Christ81:YM}, \cite{ChDNull}. Combining with the classical global existence result of Eardley-Moncrief, it has been shown that solution of \eqref{eq:YMH} in $\mathbb{R}^{1+3}$ asymptotically behaves like linear wave for sufficiently smooth and rapid decaying data \cite{Bruhat82:YM:survey}, \cite{Bruhat83:YM:univcosmos}, \cite{Bruhat82:YM:global:EinsteinU},  \cite{TauG19:YM:decay}. However, the decay assumptions on the initial data, such as the weighted energy with weights $|x|^2$, exclude the existence of charges of the fields, hence exclude the most interesting physical solutions such as solutions with dipole or quadrapole moments
\footnote{
If we formally expand the Yang-Mills field $F$ near the spatial infinity
\begin{align*}
F=F_2+F_3+\ldots+F_k+\ldots,\quad F_k=O(|x|^{-k}),
\end{align*}
then $F_2$, $F_3$ and $F_4$ vanish at any fixed time.
}.

The charges characterize the total strength of the field. It is of particular importance to study the long time dynamics of the fields with nonvanishing charges, which as far as we know are only understood for some special fields (for example \cite{Glassey81:YM:specialsl}, \cite{Glassey83:YM:scattering:special}) or for the abelian case of massless Maxwell-Klein-Gordon (MKG) system. The nonvanishing charges imply that the Yang-Mills fields (or the Maxwell fields) behave like $r^{-2}$ at spatial infinity, which indicates that  the connection field $A$ decays as $r^{-1}$ at any fixed time. This decay rate is critical (as the charge is scaling invariant) and causes substantial difficulty for the asymptotic analysis for the solutions even on the linear level. See \cite{LindbladMKG} for more detailed discussions. However, for the abelian case of MKG, Shu in \cite{Shu} observed that the charge may only affect the asymptotic behavior of the solutions in the exterior region, i.e., the region $\{t+R\leq |x|\}\subset \mathbb{R}^{1+3}$ outside a forward light cone. In other words, the charge is a local property at the spatial infinity and will not be transported to the time infinity along the future null infinity. This has been confirmed by Lindblad-Sterbenz in \cite{LindbladMKG} (also see  \cite{Lydia:MKG:small}, \cite{yangMKG}, \cite{Lindblad19:MKG}, \cite{Kauffman18:MKG}) for the small data case.

A key step toward the long time dynamics of MKG system for general large data was contributed by Yang in \cite{yangILEMKG}, where quantitative energy flux decay has been shown for the solution in the presence of arbitrarily large charge. Based on this treatment for large charge, Yang-Yu in \cite{YangYu:MKG:smooth} proposed a framework to study the long time behaviors of charged scalar fields (as solutions to the MKG system). In view of finite speed of propagation, the full problem can be divided into two steps: the first step is to study the solution in the exterior region, where the charge is large but the chargeless part can be made small (initial data are sufficiently localized). In particular, perturbation argument yields the pointwise decay estimates for the solution. For the second step inside the lightcone where the data are large but without the long range effect of nonzero charges, the asymptotic decay follows by performing a conformal transformation together with Eardley-Moncrief's classical result \cite{Moncrief1,Moncrief2}. However, directly applying Eardley-Moncrief's result requires fast decay of the solution along the hyperboloid used to do the conformal transformation. Consequently the chargeless part of the initial data were assumed to belong to the weighted energy space with weights $|x|^{6+\epsilon}$, which implies that the dipole and quadrapole moments must vanish.

Regarding  the regularity on the initial data used to obtain long time dynamics of charged scalar fields, there is a gap between the small data result in \cite{LindbladMKG} and the large data case in \cite{YangYu:MKG:smooth}. Both works required that the chargeless part of the initial data are bounded in some weighted energy space. However the former one used the weights $|x|^{1+\epsilon}$ while the latter work used the weights $|x|^{6+\epsilon}$.  These weights characterize the decay rate of the initial data at spatial infinity. The regularity required for the small data case is almost sharp in the sense that the weighted energy space with weights $|x|$ is scaling invariant. The main aim of this work is to fill this gap by extending the classical small data result of Lindblad-Sterbenz to general large data, that is, under the same regularity assumptions on the initial data but without any smallness assumption, we derive pointwise decay estimates for solutions of massless MKG system.

As discussed above, the fast decay assumptions on the initial data in \cite{YangYu:MKG:smooth} are necessary in order to apply Eardley-Moncrief's result after conformal transformation for the solutions in the interior region. Slight modification of the argument in \cite{YangYu:MKG:smooth} can be easily adapted to data in weighted energy space with weights $|x|^{1+\epsilon }$ in the exterior region. The best one can expect is that the solution lies in the same weighted energy space on the hyperboloid  which is chosen to be the initial hypersurface for the solutions in the interior region. After conformal compactification, the associated energy could be infinite. It is then equivalent to quantitatively control the growth of the solutions on the maximal Cauchy development of a ball for a class of singular data which may blow up on the boundary. In other words, for the general case of YMH equations, we show that the solution grows at most inverse polynomially for a class of infinite energy data on the initial ball.

As applications, we obtain pointwise decay estimates for YMH fields with data bounded in some weighted energy space on a hyperboloid or on the constant time slice. The data allow the presence of dipole and quadrapole moments. The charges for general Yang-Mills fields are much more complicated (see discussions for example in \cite{Yang52:YM}, \cite{Chrusciel87:charge:YM}). Shu's observation is no longer valid and the nonzero charges may have a long range effect for the asymptotic behavior of the solutions  in the whole space instead of a neighborhood of spatial infinity for the abelian case of MKG. However in our future work we will address the global dynamics of a class of Yang-Mills fields with non vanishing charges while the results in this work will serve as a crucial step. Meanwhile for the abelian case of MKG, combined with the result in \cite{YangYu:MKG:smooth}, we are  able to fill the aforementioned gap between \cite{LindbladMKG} and \cite{YangYu:MKG:smooth}.


\subsection{Main results}
In order to state our main results, let's first introduce some necessary notation. We use the standard spherical coordinates $(t, r,\om)$ on the Minkowski space where $\om$ is a local coordinate system on the unit two sphere. We also use the null coordinates $(u, v,\om)$ where
\[u=\frac{t-r}{2}, \quad v=\frac{t+r}{2}.\]
We introduce the null frame $\{L, \Lb, e_1, e_2\}$ with
\[
L=\pa_v=\pa_t+\pa_r,\quad \Lb=\pa_u=\pa_t-\pa_r
\]
and $\{e_1, e_2\}$ is an orthonormal basis tangent to the sphere of radius $r$.
We use $\D$ to denote the projection of the covariant derivative associated to the connection field $A$ to the sphere with radius $r$. For any $\lg$-valued 2-form $F$,  we can use the null frame to decompose $F$ as follows:
\begin{equation}\label{def: null decomposition of F}
\a_i=F_{Le_i}=F(L,e_i),\ \ \underline{\a}_i=F_{\Lb e_i}=F(\Lb,e_i),\ \ \rho=\f12 F_{\Lb L}=\f12 F(\Lb,L), \ \ \si=F_{e_1 e_2}=F(e_1,e_2),
\end{equation}
where $i=1$ or $2$.

The first result of the paper concerns the growth of the solution with infinite energy. For any fixed constant $R>0$, let $\B_{R}$ be the following part of the initial Cauchy hypersurface
\[
\B_R:=\{(t, x)|t=0, \quad |x|< R\}
\]
and let $\mathcal{J}^+(\B_R)$ be the future maximal Cauchy development (or the future domain of dependence)
\[
\mathcal{J}^{+}(\B_{R}):=\{(t, x)|t+|x|< R, \quad t\geq 0\}.
\]
We consider the Cauchy problem to the  \eqref{eq:YMH} system with smooth initial data $(\bar A, \bar E, \phi_0, \phi_1)$ given on $\B_{R}$, where $\bar A$ is a $\lg$-valued $1$-form, $\bar E$ is a $\lg$-valued vector  field  and $(\phi_0, \phi_1)$ are $V$-valued functions. The initial data set has to satisfy the following constraint equation:
\begin{equation*}
  \div \bar E+[\bar A, \bar E]= J[\phi]_0\big|_{t=0} =\langle \phi_1, \cT \phi_0\rangle +\langle \cT \phi_0, \phi_1\rangle.
\end{equation*}
The fields $\bar A$, $\phi_0$ and $\phi_1$ are the restrictions of $A$, $\phi$ and $D_0\phi$ to $\B_R$. The Yang-Mills field $F$ can be decomposed into the electric field $E$ and the magnetic field $H$ through $F_{0i}(0, x)=E_i(x)$ and $F_{ij}(0, x)=\epsilon_{ijk} H_k(x)$, where $\ep_{ijk}$ is the Euclidean volume element on $\B_{R}$ in the Cartesian  coordinates $(x^1, x^2, x^3)$. The field $\bar E$ is the restriction of $E$ to $\B_R$ and it contains the information of time derivative of $A$.
Moreover, the magnetic field $H$ restricted on $\B_{R}$, noted as $\bar H$, can be determined by $\bar A$ as follows
\[
\bar H= d\bar A+\f12 [\bar A\wedge \bar A] \ \ \Leftrightarrow \ \ \bar H_i=\epsilon_{ijk}\left(\pa_j \bar A_k+\f12 [\bar A_j, \bar A_k]\right),
\]
where $d$ is the exterior differential on $\B_{R}$.

For a constant $\gamma\in (0, 1)$, we define the following gauge invariant weighted energy
\begin{equation}\label{def:E_0 gamma R}
\mathcal{E}_{0, \gamma}^{R}=\int_{\B_R}(R-|x|)^{\gamma}\left(|\a|^2+|D_L\phi|^2\right)+|\rho|^2+|\sigma|^2+|\ab|^2+|D_{\Lb}\phi|^2+|\D\phi_0|^2+|\phi_0|^2 dx.
\end{equation}
We also introduce higher order weighted energies  with $k=1$ and $2$:
\begin{equation}\label{def:E_2 gamma R}
\mathcal{E}_{k, \ga}^{R}=\sum\limits_{l_1+l_2\leq k}\sum\limits_{i, j}\int_{\B_{R}}(|D^{l_1} D_{\Om_{ij}}^{l_2}F|^2+
|D^{l_1+1} D_{\Om_{ij}}^{l_2}\phi|^2)(R-|x|)^{\ga+2l_1}dx+\mathcal{E}_{0, \ga}^{R},
\end{equation}
where the rotation vector fields $\Om_{ij}$ are defined as $\Omega_{ij}=x_i\pa_j-x_j\pa_i$ for $i,j\in \{1,2,3\}$ and $i\neq j$.
The operator $D$ is the full covariant derivatives in Minkowski space (including $D_t$). The norm $|\cdot |$ is defined as $|f|=\sqrt{\langle f, f\rangle}$ for $\lg$-valued or $V$-valued field $f$. In particular, the above norms are gauge invariant. We remark that the norms are uniquely determined by the data $(\bar A, \bar E, \phi_0, \phi_1)$ together with the equation \eqref{eq:YMH} , i.e., the missing time derivatives can be recovered from the the equation.

\medskip

The first main theorem of the paper is as follows.
\begin{Thm}
\label{thm:EM}
For the Cauchy problem to the Yang-Mills-Higgs equations \eqref{eq:YMH} with initial data $(\bar A, \bar E, \phi_0,\\ \phi_1)$ given  on $\B_R$,  we assume that the weighted energy $\mathcal{E}_{2, \ga}^R$ is finite for some $0<\gamma<1$. For all positive constant $\ep$ with $0<\ep<\frac{1-\ga}{2}$, there exists a constant $C$  depending only on $\mathcal{E}_{0, \ga}^R$, $R$, $\ga$ and $\ep$, so that the solution $(\phi, F)$ satisfies the following estimates on $\mathcal{J}^{+}(\B_{R})$:
\begin{align*}
|\phi(t, x)| &\leq C  \sqrt{\mathcal{E}_{1, \ga}^R}(R-t)^{-\frac{1+\ga}{2}},\\
|D\phi(t, x)|+|F(t, x) | &\leq C  \sqrt{\mathcal{E}_{2, \ga}^R}(R-t-|x|)^{-\frac{1+\ga}{2}-\ep}(R-t)^{-1+\ep}.
\end{align*}
\end{Thm}

This theorem together with the conformal invariance of  \eqref{eq:YMH} (see for example  \cite{Christodoulou:Book:GR1}) leads to the time decay estimates for the Yang-Mills-Higgs field in the future of a hyperboloid.
 Let $R>1$ be a constant (maybe different from the one in the previous  theorem). We define the hyperboloid $\mathcal{H}$ in Minkowski space
\begin{equation*}
\mathcal{H}:=\left\{(t, x)|(t^*)^2-|x|^2=R^{-1} t^* >0\right\}, \quad t^*=t+ R +1. 
\end{equation*}
The future of $\mathcal{H}$ is given by
\begin{align*}
\mathbf{D}:=\left\{(t, x)|(t^*)^2-|x|^2\geq R^{-1} t^*>0\right\}.
\end{align*}
The second main theorem of the paper provides precise asymptotic decay estimates for solutions of \eqref{eq:YMH} in the region $\mathbf{D}$ where the data of the solution is bounded in some weighted energy space on the hyperboloid $\mathcal{H}$.
 Let $\Ga$ be the following set of commutator vector fields
\[
\Ga=\big\{\pa_{t}, \Om_{ij}=x_{i}\pa_{j}-x_j\pa_i, S=t\pa_t+r\pa_r ~\big|~
i,j =1,2,3, i\neq j\big\}.
\]
We regard the hyperboloid $\mathcal{H}$ as the graph of the function
\[t^*=(2R)^{-1}+\sqrt{|x|^2+(2R)^{-2}}.\]
Therefore, we can  parameterize $\mathcal{H}$ by $x\in \mathbb{R}^3$.
For some constant $1<\ga_1<2$, we define the following gauge invariant weighted Sobolev norm
\begin{equation*}
\begin{split}
\mathcal{E}_{k, \ga_1}^{\mathcal{H}}=&\sum\limits_{l\leq k}\sum\limits_{Z\in \Ga}\int_{\mathcal{H}_+}r^{\ga_1}\left(|\a(\mathcal{L}_Z^lF)|^2+r^{-2}|D_L D_Z^l(r\phi)|^2+
r^{-2}|\phi|^2\right)dx +E[D_Z^l\phi, \mathcal{L}_Z^l F](\mathcal{H})\\&+\int_{\mathcal{H}_{-}}\sum_{l\leq k}(|{D}^l{F}|^2+|{D}^{l+1}\phi|^2+|\phi|^2)d{x},
\end{split}
\end{equation*}
where $\mathcal{H}_+=\mathcal{H}\cap\{t\geq 0\}$, $\mathcal{H}_-=\mathcal{H}\cap\{t< 0\}$ and $E[D_Z^l\phi, \mathcal{L}_Z^l F](\mathcal{H})$ is the energy flux of the  fields $D_Z^l \phi$ and  $\mathcal{L}_Z^l F$ through the hypersurface $\mathcal{H}$. For $V$-valued field $f$ and $\lg$-valued $2$-form $G$, the energy flux $E[f, G](\mathcal{H})$ is defined as
\begin{align*}
  E[f, G](\mathcal{H}) =  \int_{\mathcal{H}} & \frac{2R(t^*-r)-1}{2R t^*-1}(|\ab(G)|^2+|D_{\Lb} f|^2)+\frac{2R(t^*+r)-1}{2R t^*-1}(|D_L f|^2+|\a(G)|^2)\\
  &+|\rho(G)|^2+(|\si(G)|^2+|\D \phi|^2)
  dx.
\end{align*}
The $\alpha(G),\ab(G), \rho(G)$ and $\sigma(G)$ are the null components of $G$ defined exactly in the same way as in \eqref{def: null decomposition of F}.
The Lie derivative $\mathcal{L}_Z $ is given by $(\mathcal{L}_ZG)_{\mu\nu}=D_{Z}G_{\mu\nu}+\partial_{\mu}Z^{\delta}G_{\delta\nu}+
\partial_{\nu}Z^{\delta}G_{\mu\delta}. $
\medskip

We have peeling decay estimates for Yang-Mills-Higgs fields in the future of $\mathcal{H}$.
\begin{Thm}
  \label{thm:YMH:hyperboloid}
Let $(F, \phi)$ be a global smooth solution to the Yang-Mills-Higgs equations \eqref{eq:YMH} in Minkowski space. We assume that $\mathcal{E}_{2, \ga_1}^{\mathcal{H}}$ is finite for some constant $1<\ga_1<2$. For all small constant $\ep$ with $0<\ep<10^{-2}(\ga_1-1)$, there exists a constant $C$ depending only on $\mathcal{E}_{0, \ga_1}^{\mathcal{H}}$, $\ep$, $\ga_1$ and $R$, so that $(F, \phi)$ satisfies the decay estimates in $ \mathbf{D} $:
  \begin{equation*}
    \begin{split}
   & |\phi|^2\leq C \mathcal{E}_{2, \ga_1}^{\mathcal{H}} u_+^{1-\ga_1}v_+^{-2},\quad |D_L\phi|^2 \leq C\mathcal{E}_{2, \ga_1}^{\mathcal{H}} u_+^{1-\ga_1}v_+^{-4}, \\
   & |\ab|^2+|D_{\Lb}\phi|^2\leq C \mathcal{E}_{2, \ga_1}^{\mathcal{H}} u_+^{-2\ep-4}v_+^{1-\ga_1+2\ep},\\
&u_+^2(|\D\phi|^2+|\rho|^2+|\si|^2)+v_+^2|\a|^2+|D_L(r\phi)|^2\leq C \mathcal{E}_{2, \ga_1}^{\mathcal{H}}  u_+^{-2\ep}v_+^{-1-\ga_1+2\ep},
    \end{split}
  \end{equation*}
where $u_+=1+\f12|t-|x||$ and $v_+=1+\f12|t+|x||$.
\end{Thm}
\begin{Remark}
The decay rates can be improved for larger $\ga_1$, i.e., for $\gamma_1\geq 2$. The case when $\ga_1>6$ for MKG equations has been studied in \cite{YangYu:MKG:smooth}. This can be generalized to the nonabelian case of Yang-Mills-Higgs equations in the parallel way: Because the data on the initial hypersurface $\mathcal{H}$ decay so fast that the associated solution after the conformal transformation is uniformly bounded in $H^2$, we can directly apply the classical result of Moncrief-Eardley \cite{Moncrief1} or the work of Klainerman-Machedon \cite{YMkl}. For the case when $2\leq \ga_1\leq 6$, the dependence on $\ga_1$ of the decay is also of independent  interest and we do not pursue this here.
\end{Remark}

\begin{Remark}
 The theorem can be formulated as a Cauchy problem with data given on the hyperboloid $\mathcal{H}$. The data consist of the $\lg$-valued $1$-forms $A\big|_{\mathcal{H}}$ and $F(N,\cdot )|_{\mathcal{H}}$ as well as the $V$-valued fields $\phi\big|_{\mathcal{H}}$ and $D_N \phi\big|_{\mathcal{H}}$, where $N$ is the future pointed unit normal of $\mathcal{H}$. The weighted energy norm $\mathcal{E}_{0, \ga_1}^{\mathcal{H}}$ is then uniquely determined by the data together with the  \eqref{eq:YMH} system.
\end{Remark}

\begin{Remark}
The decay rates of the components $|\ab|$, $|\rho|$, $|\si|$, $|D_{\Lb}\phi|$ and $|\D\phi|$ are not sharp. They are almost constant along outgoing null cones (where $u_+$ is a constant) when $\ga_1$ is approaching to $1$. However, based on these rough estimates, we will derive the sharp decay estimates in our future work.
\end{Remark}

\begin{Remark}
The uniform constant $C$ in the estimates grows exponentially in terms of the zeroth order weighted energy $\mathcal{E}_{0, \ga_1}^{\mathcal{H}}$ but linearly on higher order weighted energies.  This was conjectured in \cite{YangYu:MKG:smooth} and was also suggested by the celebrated bilinear estimates of Klainerman and Machedon in \cite{MKGkl}.
\end{Remark}

The third main theorem of the paper concerns the Cauchy problem to the MKG equations in $\mathbb{R}^{1+3}$ with admissible initial data $(\phi_0, \phi_1, E, H)$ on $\{t=0\}$. This can be viewed as a special case of YMH when $\mathbf{G}=\mathbf{U}(1)$, $\lg=\mathfrak{u}(1)=i\mathbb{R}$ and  $V=\mathbb{C}$. The  inner products on $V$ and $\mathfrak{u}(1)$ are given by $\langle a, b\rangle=a\bar b$ and $\langle a, ib\rangle =ab$ respectively. The map $\cT$ can be realized as $\cT(v)=iv$. The total charge is given by
\begin{equation*}
  q_0=\frac{1}{4\pi}\int_{\mathbb{R}^3}\Im(\phi_0\cdot\overline{\phi_1})dx
\end{equation*}
on each time slice and it is a conserved constant. We also define the weighted energy norm for the chargeless part of the initial data:
\begin{align*}
  \mathcal{E}_{l, \ga_1}^{M}=\sum\limits_{k\leq l}\int_{\mathbb{R}^3}&(1+r)^{\ga_1+2k}(|D^{k+1}\phi_0|^2+|D^k \phi_1|^2+|\nabla^k\widetilde{E} |^2+|\nabla^k H|^2)+|\phi_0|^2 dx
\end{align*}
with charge part removed from the electric field $\widetilde{E}=E-\mathbf{1}_{\{|x|\geq 1\}} q_0 r^{-3}x$.

\medskip

We have peeling decay estimates for MKG fields in the Minkowski space.

\begin{Thm} \label{thm:MKG:ga1:rough}
For the Cauchy problem to the MKG equations in Minkowski space with admissible initial data $(\phi_0, \phi_1, E, H)$ on the initial hypersurface $\{t=0\}$,  we assume that the weighted energy $\mathcal{E}_{2, \ga_1}^M$  is finite for some $\ga_1\in (1, 2)$. Then,  the solution is globally in time. Moreover, for all small constant $\ep$ with $0<\ep<10^{-2}(\ga_1-1)$, there exists a constant $C$ depending on $\mathcal{E}_{2, \ga_1}^M$, $\ga_1$ and $\ep$, so that the solution satisfies the following pointwise decay estimates:
  \begin{equation*}
    \begin{split}
   & |\phi|^2\leq C   u_+^{1-\ga_1}v_+^{-2},\quad |D_L\phi|^2 \leq C  u_+^{1-\ga_1}v_+^{-4}, \\
   & |\ab|^2+|D_{\Lb}\phi|^2\leq C  u_+^{-\ep-4}v_+^{1-\ga_1+\ep},\\
&u_+^2(|\D\phi|^2+|\rho-\mathbf{1}_{\{|x|\geq 1+t\}}q_0 r^{-2}|^2+|\si|^2)+v_+^2|\a|^2+|D_L(r\phi)|^2\leq C   u_+^{-\ep}v_+^{-1-\ga_1+\ep},
    \end{split}
  \end{equation*}
where $q_0$ is the total charge.
\end{Thm}
\begin{Remark}
The solutions possess stronger decay estimates in the exterior region $\{|x|\geq t+1\}$, we refer to \cite{LindbladMKG} and \cite{yangMKG} for more details. However, the decay estimates for the components $\ab$ and $D_{\Lb}\phi$ in the interior region $\{|x|\leq t+1\}$ are not sharp. Based on these rough estimates, we will derive the sharp decay estimates in our future work.
We emphasize here that we do not have restrictions on the size of the $\mathcal{E}_{2, \ga_1}^M$ or $q_0$. We completely removed the smallness assumption that was used in \cite{LindbladMKG}.
\end{Remark}

\begin{Remark}
  The constant $C$ still depends exponentially on $\mathcal{E}_{0, \ga_1}^M$ but polynomially on the higher order energies. The  reason is as follows: under the framework developed in \cite{YangYu:MKG:smooth} we need to choose a large constant $R$ to analyze the solution in the exterior region $\{|x|\geq t+R\}$. The choice of this constant $R$ relies on the size of the initial data $\mathcal{E}_{2, \ga_1}^M$. However, the new method in this paper also works in the exterior region so that the constant $C$ depends linearly on the higher order energy $\mathcal{E}_{2, \ga_1}^M$.
\end{Remark}

The proof for the theorem relies on the stability result of \cite{YangYu:MKG:smooth} in the exterior  region adapted to the data with finite $\mathcal{E}_{2, \ga_1}^M$. Once we have obtained the decay estimates for the solution in the exterior region, the asymptotic decay properties of the solution in the interior region follow immediately from Theorem \ref{thm:YMH:hyperboloid}. We refer the interested reader to \cite{YangYu:MKG:smooth}, \cite{LindbladMKG} and \cite{yangMKG}. For the sake of simplicity, we dot not provide the detailed analysis for the solution in the exterior region since this can be derived exactly in the same manner as in the aforementioned references.

The fourth main theorem of the paper concerns the non-abelian cases of pure Yang-Mills equations in Minkowski space. To study the asymptotic behaviors of the Yang-Mills fields, the above discussions together with Theorem \ref{thm:YMH:hyperboloid} show that it suffices to understand the solution in the exterior region. The main difficulty in this region is the possible long range effect of charges. However, unlike the abelian case for which the charge part of the field is characterized by the total electric charge, the charge part (gauge dependent) of a Yang-Mills field is much more complicated. To our best knowledge, there is no such characterization for the general non-abelian Yang-Mills fields in the literature. Nevertheless, using the vector field method as well as Theorem \ref{thm:YMH:hyperboloid}, we will provide asymptotic decay properties for Yang-Mills fields without charges.

For the Cauchy problem to the Yang-Mills equations with admissible data $(\bar A, \bar E)$ on the initial hypersurface $\{t=0\}$ (or equivalently switching off the scalar field $\phi$ in  \eqref{eq:YMH}), for any constant $\ga_1\in (1, 2)$, we define the following gauge invariant weighted norm:
\begin{align*}
  \mathcal{E}_{l, \ga_1}^{Y}=\sum\limits_{k\leq l}\int_{\mathbb{R}^3}(1+|x|)^{\ga_1+2k}(|\bar D^k  \bar E|^2+|\bar D^k \bar H|^2)dx,
\end{align*}
where $\bar H=d\bar A+\f12 [\bar A\wedge \bar A]$ and $\bar D$ is the covariant derivative associated to the $1$-form $\bar A$ on the initial hypersurface $\{t=0\}\times \mathbb{R}^3$. We have the following decay estimates for Yang-Mills fields.
\begin{Thm}\label{thm:YM:nocharge}
 For the Cauchy problem to the pure Yang-Mills equations with data $(\bar A, \bar E)$ on the initial hypersurface $\{t=0\}$ in Minkowski space $\mathbb{R}^{1+3}$, we assume that the weighted energy norm $\mathcal{E}_{2, \ga_1}^Y$ is finite for some constant $\ga_1\in (1, 2)$. Then the Yang-Mills field $F$ exists globally in time. Moreover,  for all small constant $\ep$ with $0<\ep<10^{-2}(\ga_1-1)$, there exists a constant $C$ depending on $\mathcal{E}_{2, \ga_1}^Y$, $\ga_1$ and $\ep$, so that the Yang-Mills field satisfies the following pointwise decay estimates
  \begin{equation*}
      |\ab|^2\leq C u_+^{-\ep-4}v_+^{1-\ga_1+\ep},\quad   |\a|^2\leq C    u_+^{-\ep}v_+^{-3-\ga_1+\ep},\quad   |\rho|^2+|\si|^2 \leq C    u_+^{-2-\ep}v_+^{-1-\ga_1+\ep}.
  \end{equation*}
\end{Thm}

The theorem extends the classical results of \cite{Bruhat82:YM:global:EinsteinU} 
 and \cite{Bruhat82:YM:survey} (also see recent work \cite{TauG19:YM:decay}). 
Compared to  Theorem \ref{thm:YM:nocharge}, these results required faster fall off of the initial data or quantitatively $\ga_1\geq 2$. The reason for assuming such strong decay is similar to the case in \cite{YangYu:MKG:smooth}. As in the previous discussion, the direct use of Eardley-Moncrief's result combined with the conformal method requires that the data decay sufficiently fast so that the associated field after the conformal transformation is sufficiently regular.

Similar to Theorem \ref{thm:MKG:ga1:rough}, the proof of Theorem \ref{thm:YM:nocharge} relies on an exterior stability result, for which the methods in \cite{YangYu:MKG:smooth}, \cite{LindbladMKG} and \cite{yangMKG} can be adapted. The decay estimates in the interior region then follow from Theorem \ref{thm:YMH:hyperboloid}.

\subsection{Strategy of the proof}

The early work \cite{shu2} of Shu proposed a general framework to study the long time dynamics of solutions to Maxwell-Klein-Gordon equations and Yang-Mills equations. The existence of nonzero charge has a long range effect on the dynamics of the solutions in the exterior region. The exterior region is referred to $\{t+R\leq |x|\}$  and the interior region is $\{t+R\geq |x|\}$, see the figure below. In view of the finite speed of propagation for hyperbolic equations, one can study the solutions in the exterior regions first. The main difficulty lies in the critical decay of Maxwell field due to the presence of nonzero large charge, which has been addressed in \cite{yangILEMKG}. By choosing a sufficiently large radius $R$, we can make the initial data in this region sufficiently small in energy norms. Such small data problem has been studied extensively in \cite{YangYu:MKG:smooth}, \cite{LindbladMKG} and \cite{Lydia:MKG:small}. The full problem is then reduced to the studying of the global dynamics of the solutions in the interior region or equivalently in the future of the hyperboloid $\mathcal{H}$. In particular, Theorem \ref{thm:MKG:ga1:rough} and \ref{thm:YM:nocharge} are consequences of Theorem \ref{thm:YMH:hyperboloid} combined with an exterior stability result in the scope of small data regime.
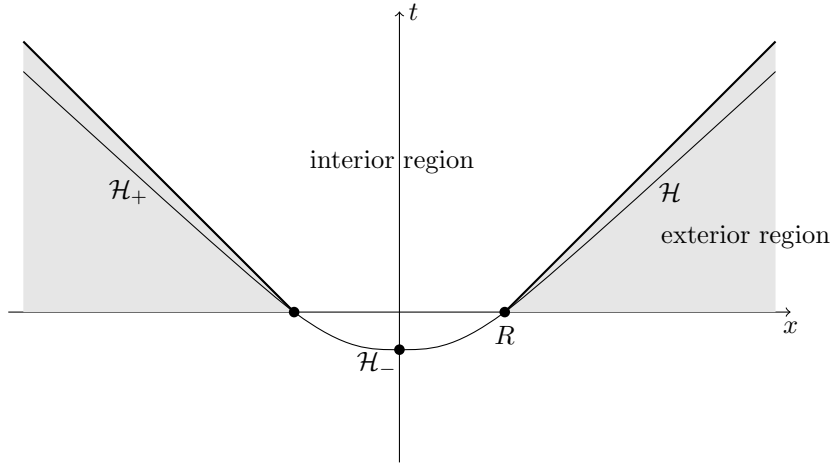
\begin{center}
\begin{tikzpicture}
  \draw[->] (-5.2,0) -- (5.2,0) coordinate[label = {below:$x$}] (xmax);
  \draw[->] (0,-2) -- (0,4) coordinate[label = {right:$t$}] (ymax);
  \fill[gray!20] (1.4, 0)--(5, 3.6)--(5, 0)--cycle;
  \fill[gray!20] (-1.4, 0)--(-5, 3.6)--(-5, 0)--cycle;
   \fill [black] (-1.4, 0) circle (2pt);
     \fill [black] (1.4, 0) circle (2pt);
     \draw [thick] (1.4,0 ) -- (5,3.6);
      \draw (-1.4,0 ) -- (-4,2.6);
  \draw plot[smooth] coordinates {(-5,3.2) (-1.4,0) (0, -0.5) (1.4, 0) (5, 3.2)};
  \node at (3.6, 1.6) {$\mathcal{H}$};
   \node at (-3.6, 1.6) {$\mathcal{H}_{+}$};
  \node at (-0.3, -0.69) {$\mathcal{H}_{-}$};
   \fill [black] (0, -0.5) circle (2pt);
       \node at (4.6, 1) {exterior region};
\draw [thick] (-1.4,0 ) -- (-5,3.6);
  \node at (-0.1, 2) {interior region};
  \node at (1.4, -0.3) {$R$};
\end{tikzpicture}
\end{center}
The future of the hyperboloid $\mathcal{H}$ can be conformally compactified to the truncated backward light cone $\mathcal{J}^{+}(\B_R)$. By the conformal invariance of the YMH equations, the solutions in the future of $\mathcal{H}$ can be transformed to the solutions of YMH in $\mathcal{J}^{+}(\B_R)$ and verse versa. In particular, Theorem \ref{thm:YMH:hyperboloid} is equivalent to Theorem \ref{thm:EM}. We can make an explicit conformal transformation to map the hyperboloid $\mathcal{H}$ to the spatial ball $\B_R$. Hence, Theorem \ref{thm:YMH:hyperboloid} follows by controlling the weighted energy norm $\mathcal{E}_{2, 2-\ga_1}^R $ on $\B_R$ in terms of the norm $\mathcal{E}_{2, \ga_1}^{\mathcal{H}}$ on the hyperboloid $\mathcal{H}$. The details are provided in section \ref{sec:conformal}.

The focus of the analysis in the paper is the proof for Theorem \ref{thm:EM}. This is a generalization of the classical global existence result of Eardley-Moncrief for a class of singular data in the sense that $\mathcal{E}_{2, \ga}^{R}$ is finite. In particular, in view of the weight $(R-|x|)^\gamma$ appearing in \eqref{def:E_0 gamma R}, the initial data could be singular on the boundary and the energy is allowed to be infinite. The theorem provides precise information on the growth of the solution in the maximal Cauchy development $\mathcal{J}^{+}(\B_R)$. Without loss of generality, we switch off the Higgs field $\phi$ and only discuss the pure Yang-Mills equations to illustrate the strategy.

We first review the ideas of Eardley and Moncrief from \cite{Moncrief2}. The global existence result follows by showing that the $L^\infty$ norm of the Yang-Mills field $F$ does not blow up in finite time. We make use of the covariant wave equations \eqref{eq:EQDphiF}, i.e.,
\[
\Box_A F_{\mu\nu}=2[F_{\mu\ga}, F_{\nu}^{\ \ga}],\quad \Box_A=D^\ga D_{\ga},
\]
and express $F_{\mu\nu}$ through the Kirchhoff formula. In \cite{Moncrief2}, they exploited the Cronstr\"{o}m gauge to express the connection field  $A$ in terms of the curvature $F$. Neglecting the lower order terms, the curvature components $F$ are roughly written as
\begin{equation}\label{eq: rep 0}
F_{\mu\nu}(q) \approx F_{\mu\nu}^{\rm lin}(q)+2\int_{\mathcal{N}^{-}(q)} [F_{\mu\ga}, F_{\nu}^{\ \ga}] \widetilde{r}d\widetilde{r}d\widetilde{\om}.
\end{equation}
For the given point $q=(t_0, x_0)$ in $\mathbb{R}^{1+3}$, $F_{\mu\nu}^{\rm lin}$ is the linear evolution depending only (linearly) on the initial data and $\mathcal{N}^{-}(q)$ is the backward light cone emanating from $q$ and truncated by $\{t=0\}$. We also used $(\widetilde{t}, \widetilde{r}, \widetilde{\om})$ as the polar coordinates centered at $q$ in the formula. For sufficiently regular initial data, the linear evolution $F_{\mu\nu}^{\rm lin}$ is uniformly bounded. To control the nonlinear term, the key observation is that $[F_{\mu\ga}, F_{\nu}^{\ \ga}]$ verifies the classical null condition:
\[
|[F_{\mu\ga}, F_{\nu}^{\ \ga}]|\lesssim |F^{\rm T}| |F|,
\]
where $F^{\rm T}$ is projection of $F$ to the hypersurface $\mathcal{N}^{-}(q)$. Hence, we have
\begin{align*}
\left|\int_{\mathcal{N}^{-}(q)} [F_{\mu\ga}, F_{\nu}^{\ \ga}] \widetilde{r}d\widetilde{r}d\widetilde{\om}\right|\lesssim\left(\int_{\mathcal{N}^{-}(q)} |F^T|^2 d\sigma \right)^{\f12} \left(\int_0^{t_0} \|F\|_{L_x^\infty} ds\right)^{\f12},
\end{align*}
in which $d\sigma$ is the induced submanifold measure on $\mathcal{N}^{-}(q)$. The first term is the standard energy  flux through $\mathcal{N}^{-}(q)$ and it is bounded by the energy estimates. Therefore, we obtain that
\begin{align*}
\|F(t_0, x)\|_{L_x^\infty}^2 \lesssim 1+ \int_0^{t_0}\|F(s, x)\|_{L^\infty_x}^2 ds.
\end{align*}
The global regularity of the Yang-Mills field then  follows from the Gronwall's inequality. To summarize, the two key points of the proof in \cite{Moncrief2} are as follows: a representation formula for the solution via a suitable gauge choice and a uniform energy estimate to bound the nonlinear term.

To adapt the framework to the current setting, firstly we derive a weighted energy estimate
\begin{equation*}
\begin{split}
&\int_{\mathcal{N}^{-}(q)} \left( (R-t-r)^\ga(|\widetilde{\rho}|^2+|\widetilde{\si}|^2) +(R-t-\min\{r, |x_0|\})^{\ga}
|\widetilde{\ab}|^2 \right)\widetilde{r}^2 d\widetilde{u}d\widetilde{\om}
\lesssim \mathcal{E}_{0, \ga}^{R}
\end{split}
\end{equation*}
by the multiplier vector field $X=(R-t-r)^{\gamma}L+(R-t+r)^{\gamma}\Lb$ with $0<\gamma<1$.
Here, $q\in \mathcal{J}^{+}(\B_R)$ and the tilde components are those associated to the coordinates or frames centered at the point $q$. We emphasize that the case $\gamma=0$ corresponds to the classical energy estimate.

Secondly, we represent the solution as a linear evolution plus a nonlinear term, similar to \eqref{eq: rep 0}. Instead of choosing a particular gauge, we use the gauge invariant geometric Kirchoff formula introduced in \cite{Kl:paramatrix}:
\begin{equation}\label{eq: rep 1}
\begin{split}
4\pi\langle\mathbf{J}_{q},  F_{\mu\nu})(q)\rangle=&\underbrace{\int_{\B_{R}}\langle\widetilde{r}^{-1}h\delta(\widetilde{v}), D_0 F_{\mu\nu}\rangle-\langle D_0(\widetilde{r}^{-1}h\delta(\widetilde{v})), F_{\mu\nu}\rangle dx}_{\mathbf{Int}_1} \\ &+\underbrace{\iint_{\mathcal{N}^{-}(q)}\langle\widetilde{\lap}_A h-[\widetilde{\rho}, h], F_{\mu\nu}\rangle\widetilde{r} d\widetilde{r}d\widetilde{\om}}_{\mathbf{Int}_2}-\underbrace{\iint_{\mathcal{N}^{-}(q)}\langle h,  \Box_A F_{\mu\nu}\rangle\widetilde{r}d\widetilde{r}d\widetilde{\om}}_{\mathbf{Int}_3}.
\end{split}
\end{equation}
The operator $\widetilde{\lap}_A$ is the covariant Laplacian associated to  the connection field $A$ on the $2$-sphere of radius $\widetilde{r}$. The test field $h$ represents the $\lg$-valued function so that
\[
D_{\widetilde{\Lb}} h=0,\quad h(q)=\mathbf{J}_{q} \  \textnormal{on} \ \mathcal{N}^{-}(q),
\]
where $\mathbf{J}_{q}$ is an arbitrary element in $\lg$ with $|\mathbf{J}_{q}|=1$.  This gauge invariant method has been used by Ghanem in \cite{Sari16:YM:glob} to give an alternative proof for the result of  Chru\'sciel and Shatah in \cite{Shatah:YM:generalManifold}.

There are three integrals on the right hand side of \eqref{eq: rep 1}. The first one is the linear evolution. It can be bounded by the vector field method, see details in Section \ref{sec:linear:th}.
More precisely, for all $\epsilon>0$, we have
\begin{align*}
|{\mathbf{Int}_1} |\les \sqrt{\mathcal{E}_{2, \ga}^R}(R-t_0-r_0)^{-\frac{1+\ga}{2}-\ep}(R-t_0)^{-1+\ep}.
\end{align*}
To control $\mathbf{Int}_3$ in \eqref{eq: rep 1}, as $|h|\leq 1$ on the cone $\mathcal{N}^{-}(q)$,  we can apply the classical null condition:
\begin{align*}
|\langle h,  \Box_A F_{\mu\nu}\rangle|\leq |F||F^{\rm T}|\les  |F|(|\widetilde{\ab}|+|\widetilde{\rho}|+|\widetilde{\si}|).
\end{align*}
Let
\[M_1(t)=\sup\limits_{|x|\leq R-t} v_*^{\frac{1+\ga}{2}+\ep}u_*^{1-\ep}|F| ,\]
where $v_*=R-t-r$ and $u_*=R-t+r$, we can bound $|F|$ in terms of $M_1$ and bound $F^{\rm T}$ by using the uniform weighted energy estimate.  We also remark that the estimates in Theorem \ref{thm:EM} is then reduced to  showing the uniform boundedness of $M_1(t)$. Since on the cone $\mathcal{N}^{-}(q)$, it holds that
\[
v_*=R-t-r\geq R-t_0-r_0,\quad u_*=R-t+r\geq R-t\geq R-t_0,
\]
we can estimate $\mathbf{Int}_3$ in \eqref{eq: rep 1} as follows:
\begin{equation*}
\begin{split}
|\mathbf{Int}_3 | &\les
 (R-t_0-r_0)^{-\frac{1+\ga}{2}-\ep} (R-t_0)^{\ep-1}  \Big(\int_{0}^{t_0}M_1^2(t_0-\widetilde{r})  \widetilde{r}^{-\ga}d\widetilde{r} \Big)^\f12.
\end{split}
\end{equation*}
It then remains to control $\mathbf{Int}_2$  in \eqref{eq: rep 1}. By commuting angular derivatives with the equation $D_{\widetilde{\Lb}}h=0$, we derive the following transport equations
\begin{align*}
D_{\widetilde{\Lb}}(\widetilde{r} D_{\widetilde{e}_j}h) =[\widetilde{r}\widetilde{\ab}_j ,h],\quad D_{\widetilde{\Lb}}(\widetilde{r}^2\widetilde{\lap}_A h- \widetilde{r}^2[\widetilde{\rho} , h])=2 \widetilde{r}^2[\widetilde{\ab}_i , D^{\widetilde{e}_i}h],
\end{align*}
where $(\widetilde{L}, \widetilde{\Lb}, \widetilde{e}_1, \widetilde{e}_2)$ is the null frame associated to the polar coordinates centered at $q$, see Section \ref{sec:proof of 5.1} for details.
 For $0\leq \widetilde{r}\leq t_0$,  we integrate over the $2$-sphere with radius $\widetilde{r}$ centered at $(t_0-\widetilde{r}, x_0)$. Using the uniform weighted energy estimate to bound $\widetilde{\ab}$, we first derive a $L^2$ estimate for $D_{\widetilde{e}_j}h$ on this $2$-sphere
\begin{align*}
\int_{|\widetilde{\om}|=1}|\widetilde{r} D_{\widetilde{e}_j}h|^2 d\widetilde{\om}&
 \les \int_{\mathcal{N}^{-}(q) } |s\widetilde{\ab}|^2(R-(t_0-s)-r_0)^{\ga} ds \cdot \int_{0}^{\widetilde{r}}(R-(t_0-s)-r_0)^{-\ga}ds d\widetilde{\om} \les \widetilde{r}^{1-\ga},
\end{align*}
which in turn leads to
\begin{align*}
\int_{|\widetilde{\om}|=1}\widetilde{r}^2|\widetilde{\lap}_A h-  [\widetilde{\rho},  h]|d\widetilde{\om}\les \int_{0}^{\widetilde{r}}\int_{|\widetilde{\om }|=1} |D^{\widetilde{e}_j} h||\widetilde{\ab}_j|  s^2dsd\widetilde{\om}\les \left(\int_{0}^{\widetilde{r}}\int_{|\widetilde{\om }|=1}  \sum_{j=1}^2|D^{\widetilde{e}_j} h|^2   s^{2-\ga }dsd\widetilde{\om}\right)^{\f12}\les \widetilde{r}^{1-\ga}.
\end{align*}
We notice that $t_0+r_0\leq R$, $0<\ga<1$. Therefore, we obtain that
\begin{align*}
\left|\iint_{\mathcal{N}^{-}(q)} \langle \widetilde{\lap}_A h-[\widetilde{\rho}, h], F_{\mu\nu} \rangle \widetilde{r} d\widetilde{r}d\widetilde{\om}\right|\les (R-t_0-r_0)^{-\frac{1+\ga}{2}-\ep} (R-t_0)^{\ep-1} \int_0^{t_0} \widetilde{r}^{-\ga} M_1(t_0-\widetilde{r})d\widetilde{r}.
\end{align*}
 In view of the definition for $M_1(t)$ together with the estimates for ${\mathbf{Int}_1}$ and ${\mathbf{Int}_3}$, multiplying both sides of the representation formula \eqref{eq: rep 1} by $ (R-t_0-r_0)^{\frac{1+\ga}{2}+\ep} (R-t_0)^{1-\ep}$,  we then derive that
\begin{align*}
M_1(t_0)\les \sqrt{\mathcal{E}_{2, \ga}^R}+\Big(\int_{0}^{t_0}M_1^2(t_0-\widetilde{r})  \widetilde{r}^{-\ga}d\widetilde{r} \Big)^\f12+  \int_0^{t_0} \widetilde{r}^{-\ga} M_1(t_0-\widetilde{r})d\widetilde{r}.
\end{align*}
The uniform boundedness of  $M_1$  then follows by a type of Gronwall inequality (see Lemma \ref{lem:Gronwall:V}).

\subsection{Outline of the paper}
In Section \ref{sec:notation}, we introduce all the quantities and notation and  we prove a type of Gronwall's inequality mentioned above. In Section \ref{sec:weightedEE}, we use vector field method to establish a uniform weighted energy estimate through backward light cones. As a consequence, we derive necessary bounds for the Higgs field, which could be viewed as a variant of Hardy's inequality. In Section \ref{sec:linear:th}, we control the linear evolution with infinite energy.  In Section \ref{sec:5}, based on the geometric representation formula together with the uniform weighted energy estimates, we obtain pointwise estimates for the solutions and prove Theorem \ref{thm:EM}. The last section is devoted to the proof for Theorem \ref{thm:YMH:hyperboloid} by the conformal method.

\bigskip

\noindent \textbf{Acknowledgments.} S. Yang is supported by the National Science Foundation of China  12171011,  12141102 and  the National Key R\&D Program of China 2021YFA1001700. P. Yu is  supported by the National Science Foundation of China 11825103, 12141102 and MOST-2020YFA0713003.

\section{Preliminaries and notations}
\label{sec:notation}

To facilitate the geometric computations, at any fixed point, we may choose the null frame $\{L, \Lb, e_1, e_2\}$ in such a way that
\begin{equation}
 \label{eq:nullderiv}
 \begin{split}
&\nabla_{e_i} L=r^{-1}e_i, \quad \nabla_{e_i}\Lb=-r^{-1}e_i, \quad \nabla_{e_1} e_2=\nabla_{e_2}e_1=0, \quad \nabla_{e_i}e_i=-r^{-1}\pa_r,
 \end{split}
\end{equation}
where $\nabla$ is the covariant derivatives of the Minkowski metric.

\def\J{\mathcal{J}^{-}}
\def\N{\mathcal{N}^{-}}
For any point $q=(t_0, x_0)\in \mathcal{J}^{+}(\B_R)$, denote $\N(q)$ as the past null cone with vertex $q$, i.e., \begin{align*}
   \mathcal{N}^{-}(q):=\big\{(t, x)\big| t_0-t=|x-x_0|,\ t\geq 0\big\}
 \end{align*}
 and $\mathcal{J}^{-}(q)$ to be the past of the point $q$, i.e., the region surrounded by $\mathcal{N}^{-}(q)$ and $\B_R$. For $r>0$ and $q=(t_0, x_0)\in\mathbb{R}^{1+3}$, denote $B_q(r)$ as the 3-dimensional spatial ball at time $t_0$ with radius $r$ centered at the point $q$, that is,
\begin{align*}
  B_q( r)=\big\{(t,x)\big| t=t_0, |x-x_0|\leq r\big\}.
\end{align*}
 The boundary of $B_q( r)$ is the $2$-sphere $S_q(r)$.

Unless it is specified, we will  fix a point $q=(t_0, x_0)\in \mathcal{J}^{+}(\B_{R})$ and we will also use the translated coordinate system $(\widetilde{t}, \widetilde{x})$ so that it is  centered at the point $q$. More specifically,
\[
\widetilde{t}=t-t_0,\quad \widetilde{x}=x-x_0,\quad \widetilde{r}=|\widetilde{x}|,\quad \widetilde{\om}=\frac{\widetilde{x}}{|\widetilde{x}|},\quad \widetilde{u}=\f12 (\widetilde{t}-\widetilde{r}),\quad \widetilde{v}=\f12(\widetilde{t}+\widetilde{r}).
\]
In particular, we have the associated null frame $\{\widetilde{L}, \widetilde{\Lb}, \widetilde{e}_1, \widetilde{e}_2\}$ which also satisfies the relation \eqref{eq:nullderiv}.

For a solution $(\phi, F)$ to the Yang-Mills-Higgs system \eqref{eq:YMH}, the associated energy momentum tensor is
\begin{align*}
  T[\phi, F]_{\mu\nu}=&\langle F_{\mu \ga}, F_{\nu}^{\ \gamma}\rangle-\frac{1}{4}m_{\mu\nu}\langle F_{\ga \delta}, F^{\ga \delta}\rangle+\langle D_{\mu}\phi, D_{\nu}\phi \rangle+\langle D_{\nu}\phi, D_{\mu}\phi\rangle- m_{\mu\nu}\langle D_\ga \phi, D^\ga\phi\rangle.
\end{align*}
We compute
\begin{align*}
  \pa^\mu T[\phi, F]_{\mu\nu}=&\langle D^\mu F_{\mu\ga}, F_{\nu}^{\ \ga}\rangle+\langle D^\mu\phi, F_{\mu\nu}\cdot \phi\rangle+\langle F_{\mu\nu}\cdot \phi, D^\mu\phi\rangle+\langle \Box_A\phi, D_\nu\phi\rangle+\langle D_\nu\phi, \Box_A\phi\rangle,
\end{align*}
where we have used the Bianchi identity $D_{[\mu}F_{\nu\ga]}=0$. In view of the definition of the map $\cT$, we derive that
\begin{align*}
  \langle D^\mu\phi, F_{\mu\nu}\cdot \phi\rangle+\langle F_{\mu\nu}\cdot \phi, D^\mu\phi\rangle=\left\langle F_{\mu\nu}, \langle \cT \phi, D^\mu\phi\rangle +\langle D^\mu\phi, \cT\phi \rangle \right\rangle=\langle F_{\mu\nu}, J[\phi]^\mu\rangle.
\end{align*}
For solutions of  \eqref{eq:YMH}, we then have
\begin{align*}
  \pa^\mu T[\phi, F]_{\mu\nu}= \langle J[\phi]_{\ga}, F_{\nu}^{\ \ga} \rangle+ \langle F_{\mu\nu}, J[\phi]^\mu \rangle=0.
\end{align*}
For any vector field $X$ and smooth function $\chi$, we define the current $J^{X, \chi}_\mu[\phi, F]$ as follows:
\begin{equation*}
J^{X, \chi}_\mu[\phi, F]=T[\phi,F]_{\mu\nu}X^\nu -
\f12\pa_{\mu}\chi \cdot|\phi|^2 + \f12 \chi\pa_{\mu}|\phi|^2. 
\end{equation*}
We have the following energy identity
\begin{align}
\label{eq:energy:id}
\iint_{\mathcal{J}^{-}(q)}\pa^\mu  J^{X,\chi}_\mu[\phi,F] d\vol =&\iint_{\mathcal{J}^{-}(q)} T[\phi, F]^{\mu\nu}\pi^X_{\mu\nu}+
\chi  \langle D_\mu\phi, D^\mu\phi\rangle -\f12\Box\chi\cdot|\phi|^2 d\vol,
\end{align}
where $\pi_{\mu\nu}^X=\f12 \mathcal{L}_X m_{\mu\nu}$ is the deformation tensor of $X$ and the operator $\Box$ is the wave operator with respect to Minkowski metric.

Throughout the paper, $\ep$  denotes a small positive constant such that
\[ 0<\ep<\min\{\frac{1-\ga}{2}, 10^{-3}(\ga_1-1)\}.\]
We use the convention that $A\les B$ means there is a constant $C$ depending only on $R$ (the radius of the initial ball $\B_R$ and the one used to define the hyperboloid $\mathcal{H}$), $\gamma$, $\gamma_1$, $\ep$ and the  representation $\b: \mathbf{G}\rightarrow U(V)$, so that $A\leq CB$.

\bigskip

We end this section with two lemmas.
For any fixed point $q=(t_0, x_0)\in \mathcal{J}^{+}(\B_R)$, we define the following function:
\[
W_q=\frac{2u_*^\ga v_*^\ga}{(1-\tau)u_*^\ga+(1+\tau)v_*^\ga},\]
where $\gamma \in (0,1)$ and
\[ v_*=R-t-r,\ \ u_*=R-t+r,\ \ \tau=\frac{x\cdot (x-x_0)}{|x||x-x_0|}.\]
Although the function $W_q$ is defined on the entire $\mathcal{J}^{+}(\B_R)$, only its value on the backward light cone $\mathcal{N}^{-}(q)$ is of interest.
We have a trivial lower bound: $$W_q\geq v_*^{\ga}=(R-t-r)^{\ga}. $$
 The first key lemma improves the above lower bound and it plays a crucial role in the sequel.
\begin{Lem}
\label{lem:Wq:lowerBD}
For any point $q=(t_0, x_0)\in \mathcal{J}^{+}(\B_R)$, let $r_q=\min\{r, |x_0|\}$. For all $(t, x)\in \mathcal{N}^{-}(q)$, we have
\begin{equation}\label{eq:Wq:lowerBD}
W_q(t, x)^{-1}\leq 3 (R-t-r_q)^{-\ga}.
\end{equation}
\end{Lem}
\begin{proof}
Let $\widetilde{x}$ be the coordinates centered at $x_0$. Therefore, $x=x_0+\widetilde{r}\widetilde{\om}$ and $\widetilde{x}=\widetilde{r}\widetilde{\om}$. Let $x_0=r_0\om_0$ and $\tau_0=\widetilde{\om}\cdot \om_0$. We have the following relations:
\[
r\tau=r_0\tau_0+\widetilde{r},\quad r=|x_0+\widetilde{r}\widetilde{\om}|.
\]
In view of the facts that $|\tau|\leq 1$, $0<\ga<1$ and
\begin{align*}
W_q^{-1}= \f12 (1+\tau)u_*^{-\ga}+\f12(1-\tau)v_*^{-\ga}\leq (R-t-r_q)^{-\ga}+\f12(1-\tau)(R-t-r)^{-\ga},
\end{align*}
to prove \eqref{eq:Wq:lowerBD}, it suffices to show that
\begin{equation}\label{eq:WqBD:0}
(1-\tau)(R-t-r_0)\leq 4( R-t-r).
\end{equation}
This inequality obviously holds if $r_0=0$ (which implies $\tau=1$) or $r\leq r_0$ or $\tau=1$ or $R-t-r\geq \frac{1}{2}(R-t-r_0)$. Hence, we may assume
\[
R-t\geq r\geq \frac{1}{2}(R-t)+\frac{1}{2}r_0,\quad r>r_0>0,\quad \tau<1.
\]
Let $s=\widetilde{r}+r_0\tau_0$, we compute $r^2$ in terms of $s$ as follows:
\begin{align*}
r^2=r_0^2-\widetilde{r}^2+2\widetilde{r}(\widetilde{r}+r_0\tau_0)=r_0^2-\widetilde{r}^2+2\widetilde{r}s.
\end{align*}
The above assumption in particular shows that $s\geq 0$. By $\widetilde{r}\leq R-t-r_0$, we derive\begin{align*}
r^2=(\widetilde{r}+r_0\tau_0)^2+(1-\tau_0^2)r_0^2\leq \max\{(R-t-r_0+\tau_0 r_0)^2+(1-\tau_0^2)r_0^2, r_0^2\}
\end{align*}
Therefore we have
\begin{align*}
R-t-r=\frac{(R-t)^2-r^2}{R-t+r}&\geq \min\left\{\frac{(R-t+r_0)(R-t-r_0)}{2(R-t)}, \frac{2(1-\tau_0)r_0(R-t-r_0)}{2(R-t)}\right\}\\
&\geq \frac{(1-\tau_0)r_0(R-t-r_0)}{2(R-t)}.
\end{align*}
We used the fact that $R-t\geq r\geq  r_0$. On the other hand, we have
\begin{align*}
1-\tau=1-\frac{s}{r}=\frac{r^2-s^2}{r(r+s)}=\frac{|x_0+\widetilde{r}\widetilde{\om}|^2-(\widetilde{r}+r_0\tau_0)^2}{r(r+s)}=\frac{(1-\tau_0^2) r_0^2}{r(r+s)}
\leq \frac{2(1-\tau_0^2)r_0^2}{(R-t)(r+s)}.
\end{align*}
Using the fact that $s\geq 0$, we proceed as follows
\begin{align*}
(1-\tau)(R-t-r_0)\leq \frac{2(1-\tau_0^2)r_0^2(R-t-r_0)}{(R-t)r}\leq \frac{2(1-\tau_0)r_0^2(R-t-r_0)}{(R-t)r_0}\leq 4(R-t-r).
\end{align*}
This yields \eqref{eq:WqBD:0} and completes the proof of the lemma.
 \end{proof}

The second lemma is a variant of the Gronwall's inequality.
\begin{Lem}
\label{lem:Gronwall:V}
Let $f$ be a positive continuous function on $[0, t_0]$. If there exist positive constants $A$, $B$ and $\gamma\in(0, 1)$ so that for all $t\in [0,t_0]$, we have
\begin{align*}
f(t)\leq A+B\int_0^t s^{-\gamma} f(t-s)ds,
\end{align*}
Then, there exists a constant $C$ depending only on $B$, $\gamma$ and $t_0$ so that
\begin{align*}
f(t)\leq C A.
\end{align*}
\end{Lem}
\begin{proof}
Let $\epsilon_0=\left(\frac{1-\ga}{2B}\right)^{\frac{1}{1-\ga}}$. For all integer $l \leq  [ \frac{t_0}{\epsilon_0} ]+1$, we prove by induction that, for all $t\leq l\epsilon_0$, we have
\begin{align*}
f(t)\leq 2^{l} A.
\end{align*}
It obviously holds for $l=0$. We assume the inequality is valid for $l\leq k-1$. Let
\begin{align*}
M_k=\sup_{t \in [(k-1)\epsilon_0, t_k=\min\{t_0, k\epsilon_0\}]} f(t).
\end{align*}
By virtue of the induction assumption, we conclude that (by choosing a sequence)
\begin{align*}
M_k &\leq A+\sum\limits_{l=1}^{k-1} B\int_{l\epsilon_0}^{(l+1)\epsilon_0} 2^{k-l} A s^{-\ga}d\ga +B\int_{0}^{\epsilon_0} M_k s^{-\ga}ds\\
&\leq A+\sum\limits_{l=1}^{k-1} 2^{k-l} A B \frac{(l+1)^{1-\ga}-l^{1-\ga}}{1-\ga}\epsilon_0^{1-\ga}+
B M_k \frac{\epsilon_0^{1-\ga}}{1-\ga}\\
&\leq A+\sum\limits_{l=1}^{k-1} 2^{k-l-1 } A     +
\frac{1}{2} M_k.
\end{align*}
In the last step, we used the fact that $(l+1)^{1-\ga}-l^{1-\gamma}\leq 1$ for $l\geq 0$ and $0\leq \ga\leq 1$. Therefore, for $k\leq 1+\frac{t_0}{\epsilon_0}$, we obtain that
\begin{align*}
M_k\leq 2^{k}A.
\end{align*}
This completes the proof of the lemma.
\end{proof}

\section{A uniform weighted energy estimate} \label{sec:weightedEE}

We begin to study the solution to \eqref{eq:YMH} on the truncated backward light cone $\mathcal{J}^{+}(\B_R)$  with smooth admissible initial data $(\bar A, \bar E, \phi_0, \phi_1)$. The data may blow up on the boundary and are not uniformly bounded in $H^2$. Our estimates will capture the precise dependence of the $L^\infty$ norm of the solution on the initial data. This is a generalization of the classical results of Eardley and Moncrief  in \cite{Moncrief1} and \cite{Moncrief2}.  The key is a uniform weighted energy estimate through backward light cones.


When $\ga>0$, since the data may not be bounded in the energy space on $\B_R$, we do not expect the energy flux through the backward light cone
$\mathcal{N}^{-}(q)$ to be uniformly bounded: It may even blow up approaching the point $(R, 0)$.
Nevertheless, we show the boundedness by considering the weighted energy flux.

\begin{Prop}
\label{prop:EF:cone:gamma}
There exists a constant $C$ depending only on $R$, so that for any point $q=(t_0, x_0)\in \mathcal{J}^{+}(\B_R)$, we have the following energy estimate
\begin{equation}
\label{eq:Eflux:Bt}
\int_{B_{(t, x_0)}( t_0-t)}|D\phi|^2+|F|^2 dx\leq C \mathcal{E}_{0, \ga}^{R} (R-t_0-r_0)^{-\ga}
\end{equation}
and the following  uniform weighted energy flux bound
\begin{equation}
\label{eq:comp:v:EF}
\begin{split}
&\int_{\mathcal{N}^{-}(q)}v_*^\ga(|\widetilde{\rho}|^2+|\widetilde{\si}|^2+r^{-2}|\widetilde{\D}(r\phi)|^2) +(R-t-r_q)^{\ga}
(|\widetilde{\ab}|^2+r^{-2}|D_{\widetilde{\Lb}}(r\phi)|^2) \quad \widetilde{r}^2 d\widetilde{u}d\widetilde{\om}\\
&+ \int_{B_{(t, x_0)}( t_0-t)} v_*^\ga (|F|^2+r^{-2}|D(r\phi)|^2)dx \leq C \mathcal{E}_{0, \ga}^{R}
\end{split}
\end{equation}
where $r_0=|x_0|$, $r=|x|$ and $r_q=\min\{r, r_0\}$. The tilde notations are quantities with respect to the coordinates $(\widetilde{t}, \widetilde{x})$ centered at the point $q=(t_0, x_0)$.
\end{Prop}

\begin{Remark}
Given a point $q=(t_0, x_0)$, the weight function $(R-t-r_0)^{\ga}$ depends only on $t$. This specific improvement in the weight will play a crucial role in obtaining the associated Hardy's inequality for the Higgs field $\phi$. The estimate \eqref{eq:Eflux:Bt} will be used to bound $\phi$ near the axis $\big\{r=0\big\}$ where the weighted energy flux bound \eqref{eq:comp:v:EF} losses control on $\phi$ for small $r$.
\end{Remark}

\begin{proof}

The estimate \eqref{eq:Eflux:Bt} is the classical energy estimate. We take $X=\pa_t$ and $\chi=0$ in the energy identity \eqref{eq:energy:id}. Since $\pa_t$ is a Killing vector field,  that is $\pi^{\pa_t}=0$, \eqref{eq:energy:id} yields
\begin{equation}\label{eq: eq 1 in energy estimates}
E[\phi, F]\left(B_{( t', x_0)}( t_0- t')\right)+E[\phi, F]\left(\mathcal{N}^{-}(q)\cap\{t\leq  t'\}\right)=E[\phi, F]\left(B_{(0, x_0)}( t_0)\right).
\end{equation}
Given a  hypersurface $\Sigma \subset \mathbb{R}^{1+3}$, $E[\phi, F](\Sigma)$ denotes the energy flux for the Higgs field $\phi$ and Yang-Mills field $F$ through $\Sigma$. In view of \eqref{def:E_0 gamma R} and the assumption on the initial data, we have
\[
E[\phi, F]\left(B_{(0, x_0)}( t_0)\right)\leq \mathcal{E}_{0, \ga}^R (R-t_0-r_0)^{-\ga}.
\]
Combined with \eqref{eq: eq 1 in energy estimates}, the energy bound \eqref{eq:Eflux:Bt} follows immediately.

To prove the weighted energy flux estimate \eqref{eq:comp:v:EF}, we choose the following multiplier vector field:
\[
X=(R-t-r)^{\gamma} L+(R-t+r)^\gamma \Lb=v_*^\ga L+u_*^\ga \Lb.
\]
We emphasize that $L=\pa_t+\pa_r$ and $\Lb=\pa_t-\pa_r$.  To compute the bulk integral on the right hand side of \eqref{eq:energy:id}, we use the null frame $(L, \Lb, e_1, e_2)$ and we compute
\[
\nabla_{L}X=-2\gamma v_*^{\gamma-1} L,\quad \nabla_{\Lb}X=-2\gamma u_*^{\gamma}\Lb,\quad \nabla_{e_i}X=r^{-1}(v_*^\gamma-u_*^{\gamma}) e_i.
\]
In particular,  the non-vanishing components of $\pi_{\mu\nu}^X$ are listed as follows:
\[
\pi^X_{L\Lb}=2\gamma \left(v_*^{\gamma-1}+u_*^{\gamma-1}\right),\quad \pi^X_{e_i e_i}=r^{-1}(v_*^{\ga}-u_*^{\gamma}).
\]
Therefore, we derive that
\begin{align*}
T[\phi, F]^{\mu\nu}\pi^X_{\mu\nu}
&=\left(\ga(v_*^{\gamma-1}+u_*^{\gamma-1})+r^{-1}(v_*^{\ga}-u_*^{\gamma})\right)(|\rho|^2+|\si|^2+2|\D\phi|^2)-2r^{-1}(v_*^{\ga}-u_*^{\gamma})\langle D^\mu\phi, D_\mu\phi\rangle.
\end{align*}
We choose the following function $\chi$ to guarantee the positivity of the bulk integral:
\[
\chi=2r^{-1}(v_*^{\ga}-u_*^{\gamma}).
\]
Since $\chi$ is  spherically symmetric with respect to the coordinates $(t, x)$, for $r>0$, we can easily derive
\begin{align*}
\Box \chi=-r^{-1}L\Lb (r\chi)=-2r^{-1}L\Lb(v_*^\ga-u_*^\ga)=0.
\end{align*}
We remark that  $\Box \chi$ may be singular at $r=0$. But since the growth is at most $r^{\gamma-3}$, it is locally integrable.
Therefore, $\Box \chi=0$ as locally integrable function. 
We then obtain that
\begin{align*}
T[\phi, F]^{\mu\nu}\pi^X_{\mu\nu}+
\chi  \langle D^\mu\phi, D_\mu\phi\rangle -\f12\Box\chi\cdot|\phi|^2 =\underbrace{\left(\ga(v_*^{\gamma-1}+u_*^{\gamma-1})+r^{-1}(v_*^{\ga}-u_*^{\gamma})\right)}_{\text{coefficient}}(|\rho|^2+|\si|^2+|\D\phi|^2).
\end{align*}
We show that the above coefficient is nonnegative. We notice that $v_*\geq 0$ and $u_*\geq 0$. Let $f(u)=u^{\ga}$ for $u>0$. As $0\leq \ga<1$, the function $f'(u)=\ga u^{\ga-1}$ is convex. Therefore,
\begin{align*}
\frac{f(u_*)-f(v_*)}{u_*-v_*}=\int_{0}^{1}f'(su_*+(1-s)v_*)ds \leq \int_{0}^{1}sf'(u_*)+(1-s)f'(v_*)ds=\frac{f'(u_*)+f'(v_*)}{2}.
\end{align*}
This inequality is equivalent to the positivity of the coefficient.  We then obtain the following energy estimate
\begin{equation}\label{eq:positive:ga:comp}
\begin{split}
\int_{\pa\mathcal{J}^{-}(q)}\pa^\mu J^{X, \chi}_\mu[\phi, F]d\vol\geq 0.
\end{split}
\end{equation}
This boundary integral consists of the contributions from the initial hypersurface $\B_R\cap \mathcal{J}^{-}(q)$ and from the backward light cone $\mathcal{N}^{-}(q)$. The contribution through $\B_R\cap \mathcal{J}^{-}(q)$ can be bounded by the initial data. Since the bulk integral is nonnegative, it yields the control of weighted energy flux through $\mathcal{N}^{-}(q)$. To find the explicit form of the weighted energy flux, we use the coordinates centered at the point $q=(t_0, x_0)$. In this new coordinate, the volume form reads as $
d\vol=d\widetilde{x}d\widetilde{t}=2\widetilde{r}^2 d\widetilde{v}d\widetilde{u}d\widetilde{\om}.
$
The submanifold measure of $\mathcal{N}^{-}(q)$ (characterized by $\{\widetilde{v}=0\}$) reads as $2\widetilde{r}^2d\widetilde{u}d\widetilde{\om}$. Therefore, we have
\begin{align*}
-i_{J^{X,\chi}[\phi, F]}d\vol=J_{\widetilde{\Lb}}^{X,\chi}[\phi, F]\widetilde{r}^2d\widetilde{u}d\widetilde{\om}.
\end{align*}
where the notation $i_Y \eta$ is the contraction of a form $\eta$ with a vector field $Y$.

We first consider the quadratic terms. Since
\begin{align*}
T[\phi, F]_{\widetilde{\Lb}\nu}X^\nu =T[\phi, F]_{\widetilde{\Lb}\widetilde{\Lb}}X^{\widetilde{\Lb}}+T[\phi, F]_{\widetilde{\Lb}\widetilde{L}}X^{\widetilde{L}}+T[\phi, F]_{\widetilde{\Lb}\widetilde{e}_i}X^{\widetilde{e}_i},
\end{align*}
we decompose the vector field $X$ in the null frame $\{\widetilde{L}, \widetilde{\Lb}, \widetilde{e}_1, \widetilde{e}_2\}$. First we have
\begin{align*}
\pa_r=\om \cdot \nabla=\om \cdot \widetilde{\nabla}=\om\cdot \widetilde{\om}\pa_{\widetilde{r}}+ \om\cdot (\widetilde{\nabla}- \widetilde{\om} \pa_{\widetilde{r}}).
\end{align*}
where $\om=\frac{x}{|x|}$ and $\nabla=(\pa_{x^1}, \pa_{x^2}, \pa_{x^3})$. Using the counterparts, we have
\begin{align*}
X
&=\f12 \left(u_*^\ga+v_*^\ga+(v_*^\ga-u_*^\ga)\om\cdot \widetilde{\om}\right)\widetilde{L}+\f12 \left(u_*^\ga+v_*^\ga-(v_*^\ga-u_*^\ga)\om\cdot \widetilde{\om}\right)\widetilde{\Lb}+(v_*^\ga-u_*^\ga)\om\cdot \widetilde{\nabb},
\end{align*}
where $\widetilde{\nabb}=\widetilde{\nabla}- \widetilde{\om} \pa_{\widetilde{r}}$. Let $\tau=\om\cdot \widetilde{\om}$. Thus,
\begin{align*}
2T[\phi, F]_{\widetilde{\Lb}\nu}X^\nu &=\left((1-\tau)v_*^\ga+(1+\tau)u_*^\ga\right)(2|D_{\widetilde{\Lb}}\phi|^2+|\widetilde{\ab}|^2) +\left((1+\tau)v_*^\ga+(1-\tau)u_*^\ga\right)(|\widetilde{\rho}|^2+|\widetilde{\si}|^2+2|\widetilde{\D}\phi|^2)\\
&+2 (v_*^\ga-u_*^\ga)\left(\langle \widetilde{\rho}, F(\widetilde{\Lb}, \ \om\cdot \widetilde{\nabb})\rangle+ \langle \widetilde{\si}, F(\Lb, \ (\om\cdot \widetilde{\nabb})^{\perp} )\rangle+\langle D_{\widetilde{\Lb}}\phi, \ \om\cdot \widetilde{\D}\phi \rangle+\langle \om\cdot \widetilde{\D}\phi, \ D_{\widetilde{\Lb}}\phi\rangle\right).
\end{align*}
The $\perp$ notation is defined as follows:  $Y^{\perp}=a \widetilde{e}_2-b \widetilde{e}_1$ for $Y=a \widetilde{e}_1+b\widetilde{e}_2$.
As $|\tau|\leq 1$, the first term on the right hand side is nonnegative. We notice that $$
\om\cdot \tilde{\nabb}={-\om\cdot (\tilde{\om}\times(\tilde{\om}\times \tilde{\nabla}))}=-(\om\times \tilde{\om})\cdot (\tilde{\om}\times \tilde{\nabla})$$
 and $\om\cdot \widetilde{\nabb}$ is perpendicular to $(\om\cdot \widetilde{\nabb})^{\perp}$ with the same length. By Cauchy-Schwarz inequality,  the Yang-Mills components are bounded as follows:
\begin{align*}
&|\langle \widetilde{\rho}, F(\widetilde{\Lb}, \ \om\cdot \widetilde{\nabb}) \rangle+\langle \widetilde{\si}, F(\Lb, \ (\om\cdot \widetilde{\nabb})^{\perp} )\rangle |\leq \sqrt{|\widetilde{\rho}|^2+|\widetilde{\si}|^2} |\widetilde{\ab}||\om\cdot \widetilde{\nabb}|= \sqrt{1-\tau^2}\sqrt{|\widetilde{\rho}|^2+|\widetilde{\si}|^2}|\widetilde{\ab}|,
\end{align*}
where we used $|\om\times \widetilde{\om}|=\sqrt{1-\tau^2}$. Since $u_*^\ga\geq v_*^\ga$, we can proceed as follows:
\begin{align*}
&\left((1-\tau)v_*^\ga+(1+\tau)u_*^\ga\right)|\widetilde{\ab}|^2+\left((1+\tau)v_*^\ga+(1-\tau)u_*^\ga\right)(|\widetilde{\rho}|^2+|\widetilde{\si}|^2)\\
&\quad +2 (v_*^\ga-u_*^\ga)(\langle \widetilde{\rho}, F(\widetilde{\Lb}, \ \om\cdot \widetilde{\nabb})\rangle+\langle\widetilde{\si}, F(\Lb, \ (\om\cdot \widetilde{\nabb})^{\perp} )\rangle )\\
&\geq \left((1-\tau)v_*^\ga+(1+\tau)u_*^\ga\right)|\widetilde{\ab}|^2+\left((1+\tau)v_*^\ga+(1-\tau)u_*^\ga\right)
(|\widetilde{\rho}|^2+|\widetilde{\si}|^2) -2 (u_*^\ga-v_*^\ga)\sqrt{1-\tau^2}\sqrt{ |\widetilde{\rho}|^2+|\widetilde{\si}|^2}|\widetilde{\ab}|\\
&\geq \frac{2u_*^\ga v_*^\ga}{(1-\tau)u_*^\ga+(1+\tau)v_*^\ga} |\widetilde{\ab}|^2+\frac{2u_*^\ga v_*^\ga}{(1+\tau)u_*^\ga+(1-\tau)v_*^\ga} (|\widetilde{\rho}|^2+|\widetilde{\si}|^2)\\
&\geq  W_q  |\widetilde{\ab}|^2+ v_*^\ga (|\widetilde{\rho}|^2+|\widetilde{\si}|^2).
\end{align*}
The last step relies on the fact that $u_*^\gamma\geq v_*^\gamma$.

The Higgs field $\phi$ satisfies a similar lower bound. However, it has to be estimated together with the lower order terms $-\widetilde{\Lb}\chi \cdot |\phi|^2$
and $\chi \widetilde{\Lb}|\phi|^2$. Since $\widetilde{\Lb}(r)=-\tau$ and $\widetilde{\nabb}(r)=\om-\widetilde{\om}\tau$,
we have
\begin{align*}
-&r^2 { \widetilde{\Lb}\chi}|\phi|^2+r^2\chi \widetilde{\Lb}|\phi|^2=2(r\chi)\langle D_{\widetilde{\Lb}}(r\phi)+\tau \phi, \phi\rangle-
(r\widetilde{\Lb}(r\chi)+\tau r\chi)|\phi|^2,\\
&r^2|D_{\widetilde{\Lb}}\phi|^2
=|D_{\widetilde{\Lb}}(r\phi)|^2+\tau^2|\phi|^2+\langle D_{\widetilde{\Lb}}(r\phi), \tau\phi \rangle+\langle\tau\phi, D_{\widetilde{\Lb}}(r\phi)\rangle,\\
& r^2|\widetilde{\D}\phi|^2
=|\widetilde{\D}(r\phi)|^2+(1-\tau^2)|\phi|^2-\langle(\om-\widetilde{\om}\tau)\widetilde{\D}(r\phi), \phi\rangle-\langle\phi, (\om-\widetilde{\om}\tau)\widetilde{\D}(r\phi)\rangle,\\
& r^2\langle D_{\widetilde{\Lb}}\phi, \om\cdot \widetilde{\D}\phi\rangle=\langle D_{\widetilde{\Lb}}(r\phi), \om \cdot \widetilde{\D}(r\phi)\rangle-\tau(1-\tau^2)|\phi|^2+\langle \phi, \tau\om\cdot \widetilde{\D}(r\phi)\rangle -\langle (1-\tau^2)D_{\widetilde{\Lb}}(r\phi),\phi\rangle.
\end{align*}
Next, we  write the Higgs field terms in $T[\phi, F]_{\widetilde{\Lb}\nu}X^\nu-\f12\widetilde{\Lb}\chi \cdot|\phi|^2+\f12\chi \widetilde{\Lb}|\phi|^2$ in terms of $D_{\widetilde{\Lb}}(r\phi)$ and $\widetilde{\D}(r\phi)$. The quadratic terms can be bounded below by $r^{-2}(W_q  |D_{\widetilde{\Lb}}(r\phi)|^2+ v_*^\ga |\widetilde{\D}(r\phi)|^2)$ for the same argument performed to estimates the quadratic terms of the Yang-Mills field. By the above computations, the lower order terms after multiplying the factor $r^2$ can be written as
\begin{align*}
&\left((1-\tau)v_*^\ga+(1+\tau)u_*^\ga\right)\left(\tau^2|\phi|^2+\langle D_{\widetilde{\Lb}}(r\phi),\tau\phi \rangle+ \langle \tau\phi, D_{\widetilde{\Lb}}(r\phi)\rangle \right)-\f12 r^2\widetilde{\Lb}\chi\cdot |\phi|^2+\f12 r^2\chi\widetilde{\Lb}|\phi|^2\\
&+\left((1+\tau)v_*^\ga+(1-\tau)u_*^\ga\right)\left((1-\tau^2)|\phi|^2-\langle (\om-\widetilde{\om}\tau)\widetilde{\D}(r\phi), \phi \rangle-\langle \phi, (\om-\widetilde{\om}\tau)\widetilde{\D}(r\phi) \rangle\right)\\
&+(v_*^\ga-u_*^\ga)\left(-2\tau(1-\tau^2)|\phi|^2+\langle \phi, \tau\om\cdot \widetilde{\D}(r\phi)-(1-\tau^2)D_{\widetilde{\Lb}}(r\phi)\rangle\right)\\
&+(v_*^\ga-u_*^\ga)\langle \tau\om\cdot \widetilde{\D}(r\phi)-(1-\tau^2)D_{\widetilde{\Lb}}(r\phi),\phi\rangle\\
=&\left(-\f12 r\widetilde{\Lb}(r\chi)+v_*^\ga+u_*^\ga\right)|\phi|^2
+(v_*^\ga+u_*^\ga)\left(\langle(\tau D_{\widetilde{\Lb}}-\om\cdot \widetilde{\D})(r\phi), \phi \rangle+\langle \phi, (\tau D_{\widetilde{\Lb}}
-\om\cdot \widetilde{\D})(r\phi)\rangle\right)\\
=&-r^2\widetilde{r}^{-1} \widetilde{\Om}_{ij}\left(r^{-3}(v_*^\ga+u_*^\ga) \om_j\widetilde{\om}_i |r\phi|^2\right)+\widetilde{r}^{-2}r^2\widetilde{\Lb}\left(r^{-1}\tau\widetilde{r}^2(v_*^\ga+u_*^\ga) |\phi|^2\right)\\
&+\left(-\f12 r\widetilde{\Lb}(r\chi)+v_*^\ga+u_*^\ga\right)|\phi|^2-\widetilde{r}^{-2}r^2\widetilde{\Lb}\left(r^{-3}\tau\widetilde{r}^2(v_*^\ga+u_*^\ga)\right) |r\phi|^2+r^2 \widetilde{r}^{-1}\widetilde{\Om}_{ij}\left(r^{-3}(v_*^\ga+u_*^\ga) \om_j\widetilde{\om}_i\right) |r\phi|^2.
\end{align*}
where $\widetilde{\Om}_{ij}=\widetilde{x}_i\widetilde{\pa}_{j}-\widetilde{x}_j\widetilde{\pa}_{{i}}=
\widetilde{r}(\widetilde{\om}_i\widetilde{\pa}_j-\widetilde{\om}_j\widetilde{\pa}_i)$ are tangential to the sphere with radius $|\widetilde{x}|$.
We show that the last line vanishes. In fact, since $\widetilde{\Lb}=\pa_t-\widetilde{\om}\cdot \nabla$ and
$\widetilde{r}^{-1}\widetilde{\Om}_{ij}=\widetilde{\om}_i\pa_j-\widetilde{\om}_j\pa_i$, we have
\begin{align*}
&\widetilde{r}^{-1}\widetilde{\Om}_{ij}(r^{-3}\om_j\widetilde{\om}_i)=-2r^{-4}(1-2\tau^2)-2\tau \widetilde{r}^{-1}r^{-3},\ \ \widetilde{r}^{-2}r^4\widetilde{\Lb}(r^{-3}\widetilde{r}^2\tau)=4\tau^2-1-2r\widetilde{r}^{-1}\tau.
\end{align*}
We now compute the coefficients of $|\phi|^2$ in the last line of the lower order terms:
\begin{align*}
&-\f12 r\widetilde{\Lb}(r\chi)+v_*^\ga+u_*^\ga-\widetilde{r}^{-2}r^4\widetilde{\Lb}(r^{-3}\tau\widetilde{r}^2(v_*^\ga+u_*^\ga)) +r^4 \widetilde{r}^{-1} \widetilde{\Om}_{ij}(r^{-3}(v_*^\ga+u_*^\ga) \om_j\widetilde{\om}_i) \\
=&-r(\pa_t-\widetilde{\om}\cdot \nabla)(v_*^\ga-u_*^\ga)-r\tau (\pa_t-\widetilde{\om}\cdot \nabla)(u_*^\ga+v_*^\ga)+r(\pa_r-\tau \widetilde{\om}\cdot \nabla)(u_*^\ga+v_*^\ga)\\
&+(u_*^\ga+v_*^\ga)\left(1-(4\tau^2-1-2r\widetilde{r}^{-1}\tau)-2(1-2\tau^2){-}2\tau \widetilde{r}^{-1}r\right)\\
=&r(\pa_t+\pa_r)u_*^\ga+r(\pa_r-\pa_t)v_*^\ga-\tau r(\pa_t+\pa_r)u_*^\ga-\tau r(\pa_t-\pa_r)v_*^\ga=0.
\end{align*}
Hence, the last line of the computations of the lower order terms vanishes. This implies that the lower order terms are in divergence form on the surface $\mathcal{N}^{-}(q)$.


After integration by parts on $\mathcal{N}^{-}(q)$, the previous calculations lead to
\begin{equation}\label{eq:EST:Nq:comp}
\begin{split}
&-\int_{\mathcal{N}^{-}(q)}\pa^\mu J^X_\mu[\phi, F]d\vol+ \int_{\mathcal{N}^{-}(q)\cap \B_R}r^{-1}\tau \widetilde{r}^2 (u_*^\ga+v_*^\ga)|\phi|^2d\widetilde{\om}\\
\geq& \int_{\mathcal{N}^{-}(q)}v_*^\ga \left(|\widetilde{\rho}|^2+|\widetilde{\si}|^2+2r^{-2}|\widetilde{\D}(r\phi)|^2\right) +W_q\left(|\widetilde{\ab}|^2+2r^{-2}|D_{\widetilde{\Lb}}(r\phi)|^2\right) \quad \widetilde{r}^2d\widetilde{u}d\widetilde{\om}.
\end{split}
\end{equation}
The boundary term on $\mathcal{N}^{-}(q)\cap \B_R$ comes from those terms with the $\widetilde{\Lb}$ derivatives in the previous identity for the lower order terms.

We next deal with the boundary integral over $\B_R\cap \mathcal{J}^{-}(q)$. First of all, we compute that
\begin{equation}
\label{eq:0022}
\begin{split}
i_{J^{X, \chi}[\phi, F]}d\vol=& T[\phi, F]_{0 L}X^L+T[\phi, F]_{\Lb 0}X^{\Lb}- \f12 \pa_t\chi |\phi|^2+ \f12\chi \pa_t|\phi|^2\\
=& \f12 v_*^\ga(|\a|^2+|\si|^2+|\rho|^2+2|D_L\phi|^2+2|\D\phi|^2)-\f12\pa_t\chi \cdot |\phi|^2+\f12\chi \pa_t|\phi|^2 \\
&+\f12 u_*^\ga(|\ab|^2+|\si|^2+|\rho|^2+2|D_{\Lb}\phi|^2+2|\D\phi|^2)\\
=&\f12(u_*^\ga+v_*^\ga)(|\rho|^2+|\si|^2+2|\D\phi|^2)+\f12 v_*^\ga(|\a|^2+2r^{-2}|D_L(r\phi)|^2)\\
&+\f12 u_*^\ga(|\ab|^2+2r^{-2}|D_{\Lb}(r\phi)|^2)-\div(\om r^{-1}|\phi|^2(u_*^\ga+v_*^\ga)).
\end{split}
\end{equation}
We compute the divergence term using the the coordinates $\widetilde{x}=x-x_0$. After integration by parts, we have
\begin{align*}
-\int_{\mathcal{J}^{-}(q)\cap \B_R} \div (\om r^{-1}|\phi|^2(u_*^\ga+v_*^\ga))dx&=-\int_{\mathcal{J}^{-}(q)\cap \B_R} \div (\om r^{-1}|\phi|^2(u_*^\ga+v_*^\ga))d\widetilde{x}\\
&=-\int_{\mathcal{N}^{-}(q)\cap \B_R} \widetilde{r}^2 \widetilde{\om} \cdot\om r^{-1}|\phi|^2(u_*^\ga+v_*^\ga)d\widetilde{\om}.
\end{align*}
Since $\tau=\om\cdot \widetilde{\om}$, it cancels the integral over $\mathcal{N}^{-}(q)$ in \eqref{eq:EST:Nq:comp}.

Finally, by the assumption on the initial data, we have
\begin{align*}
&\int_{\mathcal{J}^{-}(q)\cap \B_R}(u_*^\ga+v_*^\ga)(|\rho|^2+|\si|^2+2|\D\phi|^2)+v_*^\ga(|\a|^2+2r^{-2}|D_L(r\phi)|^2)+u_*^\ga(|\ab|^2+2r^{-2}|D_{\Lb}(r\phi)|^2)dx\\
&\leq C_R\int_{\B_R}(R-|x|)^\ga(|\a|^2+|D_L\phi|^2)+(|\rho|^2+|\si|^2+|\D\phi|^2+|\ab|^2+|D_{\Lb}\phi|^2)+|\phi|^2dx\leq C \mathcal{E}_{0, \ga}^R
\end{align*}
where the constant $C$ depends only on $R$ and is independent of the choice of $q$. Combining estimates \eqref{eq:positive:ga:comp}, \eqref{eq:EST:Nq:comp} and the lower bound of $W_q$ in Lemma \ref{lem:Wq:lowerBD}, we obtain the bound for the first integral of \eqref{eq:comp:v:EF}. For the second integral on $B_{(t, x_0)}(t_0-t)$, it suffices to change the region to $\mathcal{J}^{-}(q)\cap ([0,t]\times\mathbb{R}^3)$.
The flux integral on $B_{(t, x_0)}(t_0-t)$ is of the same form as that on the initial hypersurface $\B_R$ (see the above computation \eqref{eq:0022}). As the divergence terms are canceled and $u_*\geq v_*$, we then derive the bound for the second integral in \eqref{eq:comp:v:EF}. This completes the proof.
\end{proof}

The rest of the section is devoted to a uniform bound for the Higgs field $\phi$ on backward light cones. We need the $L^2$ bounds of $\phi$ on spheres $S_{q}( \widetilde{r})$. The next result studies the case when $q=(t_0, 0)$ and it is a consequence of Proposition \ref{prop:EF:cone:gamma}.

\begin{cor}
\label{cor:bd4rphi2:x0}
Let  $|D\phi|^2=\sum\limits_{\mu=0}^{3}|D_\mu\phi|^2$. Then we have
\begin{align}
\label{eq:bd4rphi2:x0}
r\int_{S_{(t, 0)}( r)}|\phi|^2d\om &\les \mathcal{E}_{0,\ga}^{R} (R-t)^{-\ga},\quad \forall r\leq R-t,\\
\label{eq:bd4Dphi:B}
\int_{B_{(t, 0)}( R-t)}v_*^{\ga+\ep}|D\phi|^2 dx &\les \mathcal{E}_{0,\ga}^{R} (R-t)^\ep,\quad  \forall  t\leq R.
\end{align}
\end{cor}
\begin{proof}
To prove\eqref{eq:bd4Dphi:B}, in view of \eqref{eq:Eflux:Bt} and the fact that $0<\ep<\frac{1-\ga}{2}$, we have
\begin{align*}
\int_{B_{(t, 0)}( R-t)}v_*^{\ga+\ep}|D\phi|^2 dx&=(\ga+\ep)
\int_{0}^{R-t}(R-t-r)^{\ga+\ep-1}\left(\int_{B_{(t,0)}( r)}|D\phi|^2dx \right) dr\\
&\les \mathcal{E}_{0, \ga}^{R}\int_{0}^{R-t}(R-t-r)^{\ep-1}dr\les \mathcal{E}_{0, \ga}^R (R-t)^\ep.
\end{align*}

To prove \eqref{eq:bd4rphi2:x0}, we use  \eqref{eq:comp:v:EF} on $B_{(t, 0)}( r)$ to derive\begin{align*}
\int_{S_{(t, 0)}(r_0)}|r\phi|^2 d\om &\les \int_{B_{(t, 0)}(r_0)}v_*^\ga |D(r\phi)|^2 r^{-2}dx \cdot \int_{0}^{r_0}v_*^{-\ga}dr\\
&\les \mathcal{E}_{0, \ga}^R (R-t)^{1-\ga}-\mathcal{E}_{0, \ga}^R (R-t-r_0)^{1-\ga}\\
&\les \mathcal{E}_{0, \ga}^R (R-t)^{-\ga}r_0.
\end{align*}
This completes the proof of the lemma.
\end{proof}
We will need a uniform bound for the Higgs field $\phi$. Let's first prove several preliminary lemmas:
\begin{Lem}
\label{lem:bd4Rss2}
For all $\widetilde{r}\in (0,R_0]$, we have
\begin{align}
\label{eq:bd4Rss2}
\int_{\widetilde{r}}^{R_0}(R_0-s)^{-\ga}s^{-2}ds \leq \frac{2^{1+\ga}}{1-\ga} R_0^{-\ga}\widetilde{r}^{-1}.
\end{align}
\end{Lem}
\begin{proof}
We consider two cases: $\widetilde{r}\geq \frac{1}{2}R_0$ or $\widetilde{r}< \frac{1}{2}R_0$. If $\widetilde{r}\geq \frac{1}{2}R_0$, we have
\begin{align*}
\int_{\widetilde{r}}^{R_0}(R_0-s)^{-\ga}s^{-2}ds&\leq 4 R_0^{-2}\int_{\widetilde{r}}^{R_0}(R_0-s)^{-\ga} ds=\frac{4}{1-\ga}(R_0-\widetilde{r})^{1-\ga}R_0^{-2}\leq \frac{2^{1+\ga}}{1-\ga}R_0^{-\ga}\widetilde{r}^{-1}.
\end{align*}
If $\widetilde{r}<\f12 R_0$, we have
\begin{align*}
\int_{\widetilde{r}}^{R_0}(R_0-s)^{-\ga}s^{-2}ds&=\left(\int_{\f12 R_0}^{R_0}+\int^{\f12 R_0}_{\widetilde{r}}\right)(R_0-s)^{-\ga}s^{-2}ds\leq 4R_0^{-2}\frac{1}{1-\ga}(R_0-\f12 R_0)^{1-\ga}+(\f12 R_0)^{-\ga}\widetilde{r}^{-1}\\
&\leq \frac{2^{\ga}}{1-\ga}R_0^{-\ga} \widetilde{r}^{-1}+2^{\ga}R_0^{-\ga}\widetilde{r}^{-1}\leq \frac{2^{1+\ga}}{1-\ga}R_0^{-\ga}\widetilde{r}^{-1}.
\end{align*}
This proves the lemma.
\end{proof}
The $L^2$ norm of $\phi$ on $S_{(t, x_0)}(\widetilde{r})$ can be bounded, up to an error term, by those on another sphere.
\begin{Lem}
\label{lem:smalllarge}
For all $\widetilde{r}_1,\widetilde{r}_2\in (0,R-t-r_0]$, we have
\begin{equation}
\label{eq:smalllarge}
\min\{\widetilde{r}_1 ,\widetilde{r}_2\}\int_{S_{(t, x_0)}(\widetilde{r}_1)}|\phi|^2d\widetilde{\om}\les \min\{\widetilde{r}_1,\widetilde{r}_2\} \int_{S_{(t, x_0)}(\widetilde{r}_2)}|\phi|^2d\widetilde{\om}+\mathcal{E}_{0, \ga}^R (R-t-r_0)^{-\ga-\ep}(R-t)^\ep.
\end{equation}
\end{Lem}
\begin{proof}
By symmetry, we may assume that $\widetilde{r}_1\leq \widetilde{r}_2$.  We first  integrate between $S_{(t, x_0)}(\widetilde{r}_1)$ and $S_{(t, x_0)}(\widetilde{r}_2)$. This leads to
\begin{align*}
|\phi|(t, x_0+\widetilde{r}_1\widetilde{\om})&\leq |\phi|(t, x_0+\widetilde{r}_2\widetilde{\om})+\int_{\widetilde{r}_1}^{\widetilde{r}_2}|\widetilde{D}\phi(t, x_0+s\widetilde{\om})|ds\\
&\leq |\phi|(t, x_0+\widetilde{r}_2\widetilde{\om})+\left(\int_{\widetilde{r}_1}^{\widetilde{r}_2}v_*^{\ga+\ep}|\widetilde{D}\phi|^2 s^2ds\right)^\f12\left(\int_{\widetilde{r}_1}^{\widetilde{r}_2}v_*^{-\ga-\ep}s^{-2}ds\right)^\f12.
\end{align*}
Since $v_*=R-t-r\geq R-t-r_0-s$ and $\widetilde{r}_2\leq R-t-r_0$, we use \eqref{eq:bd4Dphi:B} from Corollary \ref{cor:bd4rphi2:x0} and Lemma \ref{lem:bd4Rss2} to bound the above integrals. Hence,
\begin{align*}
\int_{S_{(t, x_0)}(\widetilde{r}_1)}|\phi|^2d\widetilde{\om}&\les \int_{S_{(t, x_0)}(\widetilde{r}_2)}|\phi|^2d\widetilde{\om}+(R-t-r_0)^{-\ga-\ep} \widetilde{r}_1^{-1}\int_{B_{(t, x_0)}(\widetilde{r}_2)}v_*^{\ga+\ep}|\widetilde{D}\phi|^2  d\widetilde{x}\\
&\les \int_{S_{(t, x_0)}(\widetilde{r}_2)}|\phi|^2d\widetilde{\om}+(R-t-r_0)^{-\ga-\ep}(R-t)^\ep \widetilde{r}_1^{-1 } \mathcal{E}_{0, \ga}^R,
\end{align*}
where $\widetilde{x}=\widetilde{r}\widetilde{\om}$. This proves \eqref{eq:smalllarge}.
\end{proof}
The last lemma proves a rough $L^2$ bound for the Higgs field.
\begin{Lem}
\label{lem:bd4phi2:2}
For all $\widetilde{r}\leq R-t-r_0$, we have
\begin{equation}
\label{eq:bd4phi2:2}
\int_{S_{(t, x_0)}(\widetilde{r})}|\phi|^2d\widetilde{\om}\les \mathcal{E}_{0,\ga}^R \widetilde{r}^{-2} (R-t)^{1-\ga}.
\end{equation}

\end{Lem}
\begin{proof}
Denote $R_0^*=R-t-r_0$. Now from the bound \eqref{eq:bd4rphi2:x0}, we integrate in $r$ from $\max\{r_0-\frac{2}{3}R_0^*, 0\}$ to $r_0+\frac{2}{3}R_0^*$. We derive that
\begin{align*}
\int_{B_{(t, x_0)}(\frac{2}{3}R_0^*)}|\phi|^2\widetilde{r}^2 d\widetilde{r}d\widetilde{\om}&=\int_{B_{(t, x_0)}(\frac{2}{3}R_0^*)}|\phi|^2r^2 dr d\om\\
&\les \mathcal{E}_{0, \ga}^R\int_{\max\{r_0-\frac{2}{3}R_0^*, 0\}}^{r_0+\frac{2}{3}R_0^*}r (R-t)^{-\ga}dr\\
&\les \mathcal{E}_{0,\ga}^R (|r_0+\frac{2}{3}R_0^*|^2-|\max\{r_0-\frac{2}{3}R_0^*, 0\}|^2) (R-t)^{-\ga}\\
&\les \mathcal{E}_{0, \ga}^R (R-t)^{-\ga} R_0^* \max\{r_0, \frac{2}{3}R_0^*\}\\
&\les \mathcal{E}_{0, \ga}^R R_0^* (R-t)^{1-\ga}.
\end{align*}
Here we notice that $dx=d\widetilde{x}$ and $\widetilde{x}$ are the coordinates centered at the point $x_0$. The last step follows as $\frac{1}{3}(R-t)\leq \max\{r_0, \frac{2}{3}R_0^*\}\leq (R-t)$. Therefore we conclude that there must be some $\widetilde{r}_0\in[\frac{1}{3}R_0^*, \frac{2}{3}R_0^*]$ such that
\[
\int_{S_{(t, x_0)}(\widetilde{r}_0)}|\widetilde{r}_0\phi|^2 d\widetilde{\om}\les \mathcal{E}_{0,\ga}^R (R-t)^{1-\ga}.
\]
Then from the previous lemma \ref{lem:smalllarge}, for all $\widetilde{r}_0<\widetilde{r}_1=R_0^*\leq 3\widetilde{r}_0$, we derive that
\begin{align*}
\widetilde{r}_0^2 \int_{S_{(t, x_0)}(\widetilde{r}_1)}|\phi|^2d\widetilde{\om}&\les \widetilde{r}_0^2 \int_{S_{(t, x_0)}(\widetilde{r}_0)}|\phi|^2d\widetilde{\om}+ \mathcal{E}_{0, \ga}^R (R-t-r_0)^{1-\ga-\ep}(R-t)^\ep\\
&\les \mathcal{E}_{0,\ga}^R (R-t)^{1-\ga}.
\end{align*}
Here note that $\ep$ is chosen such that $\ep+\ga<1$. This in particular means that the Lemma holds for all $\widetilde{r}_0\leq \widetilde{r}\leq R_0^*=R-t-r_0$. For smaller $\widetilde{r}<\widetilde{r}_0$, using the previous Lemma \ref{lem:smalllarge} again, we have
\begin{align*}
  \widetilde{r}\int_{S_{(t, x_0)}(\widetilde{r})}|\phi|^2d\widetilde{\om} &\les \widetilde{r} \int_{S_{(t, x_0)}(\widetilde{r}_0)}|\phi|^2d\widetilde{\om}+\mathcal{E}_{0, \ga} (R-t-r_0)^{-\ga-\ep}(R-t)^\ep\\
  &\les \mathcal{E}_{0,\ga}^R \widetilde{r}_0^{-1}(R-t)^{1-\ga}+\mathcal{E}_{0, \ga}^R (R-t-r_0)^{-\ga-\ep}(R-t)^\ep\\
  &\les \mathcal{E}_{0,\ga}^R \widetilde{r}^{-1}(R-t)^{1-\ga}.
\end{align*}
Therefore the Lemma holds for all $\widetilde{r}\leq R-t-r_0$.
\end{proof}
We prove the necessary uniform (independent of the center $x_0$) bound for the Higgs field.
\begin{Prop}
\label{prop:bd4phi:need}
For any $q=(t_0, x_0)\in\mathcal{J}^+(\B_R)$ and $0\leq \widetilde{r}\leq t_0$, we have
\begin{equation}
\label{eq:bd4phi:need}
\begin{split}
\int_{S_{(t_0-\widetilde{r}, x_0)}(\widetilde{r}) }|\phi|^2 \widetilde{r} d\widetilde{\om}\les \mathcal{E}_{0,\ga}^R \widetilde{r}^{-\ga-\ep}.
\end{split}
\end{equation}
\end{Prop}
\begin{proof}
Let $r_0=|x_0|$ and $R_0=R-t_0-r_0$. $R_0$ is nonnegative since 
$t_0+|x_0|\leq R$.
In view of Lemma \ref{lem:smalllarge} and the rough bound \eqref{eq:bd4phi2:2}, we have\begin{equation}
\label{eq:smalllargeP2}
\begin{split}
\int_{S_{(t_0-\widetilde{r}, x_0)}(\widetilde{r})}|\phi|^2\widetilde{r}d\widetilde{\om}
&\les \int_{S_{(R_0+t_0-(\widetilde{r}+R_0), x_0)}(R_0+\widetilde{r})}|\phi|^2(R_0+\widetilde{r})d\widetilde{\om}+\mathcal{E}_{0, \ga}^{R} (R_0+\widetilde{r})^{-\ga-\ep} \\
&\les \mathcal{E}_{0, \ga}^{R} \big(\widetilde{r}^{-\ga-\ep}+(R_0+\widetilde{r})^{-1}(R-t_0+\widetilde{r})^{1-\ga}\big).
\end{split}
\end{equation}
If $\widetilde{r}\geq \frac{1}{10}r_0$, as $t_0+r_0\leq R$, we have $R-t_0+\widetilde{r}\leq 11(R-t_0-r_0+\widetilde{r})=11R_0$. Therefore, \eqref{eq:bd4phi:need} holds.

It suffices to assume  $\widetilde{r}\leq \widetilde{r}_0^*=\min\{  t_0, \frac{1}{10} r_0\}$.  For $\widetilde{r}_0\leq \widetilde{r}_0^*$, we integrate along the backward light cone:
\begin{equation}\label{eq:smalllargePP2}
\begin{split}
|r\phi|(t_0-\widetilde{r}_0, x_0+\widetilde{r}_0\widetilde{\om})&\leq |r\phi|(t_0-\widetilde{r}_0^*, x_0+\widetilde{r}_0^*\widetilde{\om})+\int_{\widetilde{r}_0}^{\widetilde{r}_0^*}|\widetilde{D}_{\widetilde{\Lb}}(r\phi)|(t_0-\widetilde{r}, x_0+\widetilde{r} \widetilde{\om})d\widetilde{r}\\
&\les |r\phi|(t_0-\widetilde{r}_0^*, x_0+\widetilde{r}_0^*\widetilde{\om})+\widetilde{r}_0^{-\frac{1+\ga}{2}}\big(\int_{\widetilde{r}_0}^{\widetilde{r}_0^*}(R-t_0+\widetilde{r}-r_0)^{\ga}
|\widetilde{D}_{\widetilde{\Lb}}(r\phi)|^2\widetilde{r}^2 d\widetilde{r}\big)^{\f12}.
\end{split}
\end{equation}
On $\mathcal{N}^{-}(q)\cap\{t_0-\widetilde{r}_0^*\leq t\leq t_0\}$, the radius $r=|x|$ verifies the following bound
\begin{equation*}
\f12 r_0\leq r_0-\widetilde{r}_0^*\leq r_0-\widetilde{r}\leq r\leq r_0+\widetilde{r}\leq r_0+\widetilde{r}_0^*\leq \frac{3}{2}r_0 .
 \end{equation*}
Thus, radius function $r=|x_0+\widetilde{r}\widetilde{\om}|$ is comparable everywhere on the line segment $\big\{(t-\widetilde{r}, x_0+\widetilde{r}\widetilde{\om})\big|\widetilde{r}\in [0, \\ \widetilde{r}_0^* ]\big\}$. Dividing both sides of \eqref{eq:smalllargePP2} by $r$ and using the improved weighted energy estimate \eqref{eq:comp:v:EF}, we have
\begin{align*}
\int_{S_{(t_0-\widetilde{r}_0, x_0)}(\widetilde{r}_0)}|\phi|^2d\widetilde{\om}&\les \int_{S_{(t_0-\widetilde{r}_0^*, x_0)}(\widetilde{r}_0^*)}|\phi|^2d\widetilde{\om}+\widetilde{r}_0^{-1-\ga} \int_{\widetilde{r}_0}^{\widetilde{r}_0^*}\int_{\widetilde{\om}}(R-t-r_0)^{\ga}r^{-2}|\widetilde{D}_{\widetilde{\Lb}}(r\phi)|^2\widetilde{r}^2 d\widetilde{r} d\widetilde{\om}\\
&\les \int_{S_{(t_0-\widetilde{r}_0^*, x_0)}(\widetilde{r}_0^*)}|\phi|^2d\widetilde{\om} +\widetilde{r}_0^{-1-\ga}\mathcal{E}_{0,\ga}^{R}.
\end{align*}

We have two sub-cases: $\widetilde{r}_0^*=\frac{1}{10} r_0$ or $t_0$.

If $\widetilde{r}_0^*=\frac{1}{10} r_0$, then $\frac{1}{10} r_0\leq  t_0$, as \eqref{eq:smalllargeP2} holds for all $\widetilde{r}$, we in particular have
\begin{equation*}
\begin{split}
\int_{S_{(t_0-\widetilde{r}_0^*, x_0)}(\widetilde{r}_0^*)}|\phi|^2\widetilde{r}_0^*d\widetilde{\om}
&\les \mathcal{E}_{0, \ga}^{R} ((\widetilde{r}_0^*)^{-\ga-\ep}+(R-t_0-r_0+\widetilde{r}_0^*)^{-1}(R-t_0+\widetilde{r}_0^*)^{1-\ga})\\
&\les \mathcal{E}_{0, \ga}^{R} ((\widetilde{r}_0^*)^{-\ga-\ep}+(R-t_0+\widetilde{r}_0^*)^{-1}(R-t_0+\widetilde{r}_0^*)^{1-\ga})\les \mathcal{E}_{0, \ga}^{R} (\widetilde{r}_0^*)^{-\ga-\ep}.
\end{split}
\end{equation*}
Using the fact that $\widetilde{r}_0\leq \widetilde{r}_0^*$, the above estimate implies the desired estimate:
\begin{align*}
\int_{S_{(t_0-\widetilde{r}_0, x_0)}(\widetilde{r}_0)}|\phi|^2 d\widetilde{\om}
&\les \mathcal{E}_{0,\ga}^{R} ((\widetilde{r}_0^*)^{-\ga-\ep-1}+\widetilde{r}_0^{-1-\ga})\les \mathcal{E}_{0,\ga}^{R} \widetilde{r}_0^{-\ga-\ep-1}.
\end{align*}

If $\widetilde{r}_0^*= t_0\leq \frac{1}{10}r_0$,
in view of Lemma \ref{lem:smalllarge} and \eqref{eq:bd4phi2:2}, we have
\begin{align*}
t_0\int_{S_{(0, x_0)}(t_0)}|\phi|^2 d\widetilde{\om}&\les t_0 \int_{S_{(0, x_0)}(r_0)}|\phi|^2d\widetilde{\om}+\mathcal{E}_{0, \ga}^R (R-r_0)^{-\ga-\ep}\\
& \les \mathcal{E}_{0, \ga}^R t_0 r_0^{-2}  +\mathcal{E}_{0, \ga}^R (R-r_0)^{-\ga-\ep} .
\end{align*}
As $\widetilde{r}_0\leq \widetilde{r}_0^*=t_0\leq \frac{1}{10}r_0$ and $R-r_0\geq t_0$, the above estimate leads to
\begin{align*}
\int_{S_{(t_0-\widetilde{r}_0, x_0)}(\widetilde{r}_0)}|\phi|^2d\widetilde{\om}
&\les \mathcal{E}_{0,\ga}^{R}(\widetilde{r}_0^{-1-\ga}+r_0^{-2}+t_0^{-\gamma -\ep -1})\les \mathcal{E}_{0,\ga}^{R}(\widetilde{r}_0^{-1-\ga-\ep }+r_0^{-2} ).
\end{align*}
We use the condition $\widetilde{r}_0\leq t_0\leq R$: if $r_0\geq \frac{1}{10} R$, then $r_0^{-2}\les 1\les \widetilde{r}_0^{-1-\ga-\ep }$; if $r_0\leq \frac{1}{10} R$, then $t_0\leq \frac{1}{100}R$. In particular, we have $R_0=R-t_0-r_0\geq \frac{R}{2}$.
In view of  \eqref{eq:smalllargeP2}, we conclude that
\begin{equation*}
\begin{split}
\int_{S_{(t_0-\widetilde{r}, x_0)}(\widetilde{r})}|\phi|^2\widetilde{r}d\widetilde{\om}
 &\les \mathcal{E}_{0, \ga}^{R} \big(\widetilde{r}^{-\ga-\ep}+(R_0+\widetilde{r})^{-1}(R-t_0+\widetilde{r})^{1-\ga}\big)\les \mathcal{E}_{0, \ga}^{R}  \widetilde{r}^{-\ga-\ep}.
\end{split}
\end{equation*}

The above discussions covers all the cases and complete the proof of the proposition.
\end{proof}
This proposition is important in the following sense: The estimate \eqref{eq:bd4phi:need} is uniform with respect to the center $x_0$. Moreover, the right hand side of \eqref{eq:bd4phi:need} is integrable in $\widetilde{r}$ on the interval $[0, t_0]$.

To end the section, we prove another uniform bound for the Higgs field. We will use this bound to compensate the $r$-weight of the Higgs field in the weighted energy estimate \eqref{eq:comp:v:EF} .
\begin{Prop}
\label{prop:bd4phi:phir}
For all $q=(t_0, x_0)\in\mathcal{J}^{+}(\B_R)$ and $0\leq \widetilde{r}\leq t_0$, let $r=|x_0+\widetilde{r}\widetilde{\om}|$, $(\widetilde{t}, \widetilde{r}, \widetilde{\om})$ be the polar coordinates centered at $q$ and $S_{(t_0-\widetilde{r}, x_0)}(\widetilde{r})$ be the sphere with radius $\widetilde{r}$ centered at the point $(t_0-\widetilde{t}, x_0)$. We have the following estimate:
\begin{equation}
\label{eq:bd4phi:phir}
\int_{S_{(t_0-\widetilde{r}, x_0)}(\widetilde{r}) }|\phi|r^{-1}\widetilde{r} d\widetilde{\om}\les \sqrt{\mathcal{E}_{0,\ga}^{R}} \widetilde{r}^{-\frac{1+\ga+\ep}{2}}.
\end{equation}
\end{Prop}
\begin{proof}
 By H\"older's inequality, we first have
 \begin{align*}
\int_{ S_{(t_0-\widetilde{r}, x_0)}(\widetilde{r}) }|\phi|r^{-1}  d\widetilde{\om}\leq \left(\int_{S_{(t_0-\widetilde{r}, x_0)}(\widetilde{r})}|\phi|^2d\widetilde{\om}\right)^\f12\left(\int_{S_{(t_0-\widetilde{r}, x_0)}(\widetilde{r})}r^{-2}d\widetilde{\om}\right)^\f12 .
\end{align*}
Since $r^2=r_0^2+\widetilde{r}^2+2\widetilde{r}x_0\cdot \widetilde{\om}$, we can compute that
\begin{align*}
I:=\int_{S_{(t_0-\widetilde{r}, x_0)}(\widetilde{r})}r^{-2}d\widetilde{\om}=4\pi \int_{-1}^{1}(r_0^2+\widetilde{r}^2+2r_0\widetilde{r}\tau)^{-1}d\tau=4\pi (r_0\widetilde{r})^{-1}\ln\frac{r_0+\widetilde{r}}{|r_0-\widetilde{r}|}.
\end{align*}
We have two cases: $|r_0-\widetilde{r}|\geq \f12 \widetilde{r}$ or $\f12 \widetilde{r}<r_0\leq \frac{3}{2}\widetilde{r}$.

If $|r_0-\widetilde{r}|\geq \f12 \widetilde{r}$, thus, $(r_0\widetilde{r})^{-1}\ln\frac{r_0+\widetilde{r}}{|r_0-\widetilde{r}|}\les \widetilde{r}^{-2}$. By \eqref{eq:bd4phi:need} { in Proposition} \ref{prop:bd4phi:need}, we have
\begin{align}
\label{eq:bdphir1}
  \int_{ S_{(t_0-\widetilde{r}, x_0)}(\widetilde{r}) }|\phi|r^{-1}  d\widetilde{\om}\les \widetilde{r}^{-1}\sqrt{\mathcal{E}_{0, \ga}^{R}} \widetilde{r}^{-\frac{1+\ga+\ep}{2}}.
\end{align}

If $\f12 \widetilde{r}<r_0\leq \frac{3}{2}\widetilde{r}$, the above approach fails since { the integral $I$ blows up} 
 for $r_0$ close to $\widetilde{r}$. Fix the time $t=t_0-\widetilde{r}$ and the center $x_0$, we shall compare the integral of $r^{-1}|\phi|$ on two spheres centered at $(t, x_0)$ with radii $\frac{2}{3}r_0$ and $\widetilde{r}$ respectively. For all $\frac{2}{3}r_0\leq \widetilde{r}\leq 2r_0 $, we first have
\begin{align*}
  (|\phi|r^{-1})(t, x_0+\widetilde{r}\widetilde{\om})\leq (|\phi|r^{-1})\left(t, x_0+\frac{2}{3}r_0\widetilde{\om}\right)+\int_{\frac{2}{3}r_0}^{\widetilde{r}}(|D\phi|r^{-1}+ r^{-2}|\phi|)(t, x_0+s\widetilde{\om}) ds
\end{align*}
We then integrate over $\widetilde{\om}$ and we derive
\begin{align*}
  \int_{ S_{(t_0-\widetilde{r}, x_0)}(\widetilde{r}) }|\phi|r^{-1}  d\widetilde{\om}\les \underbrace{\int_{ S_{(t_0-\widetilde{r}, x_0)}(\frac{2}{3}r_0) }|\phi|r^{-1}  d\widetilde{\om}}_{I_\mathbf{S}}+\underbrace{ \int_{\frac{2}{3}r_0}^{\widetilde{r}}\int_{ S_{(t_0-\widetilde{r}, x_0)}(s )}  |D\phi|r^{-1}+ r^{-2}|\phi| ds d\widetilde{\om}}_{I_\mathbf{B}}.
\end{align*}
To bound the sphere integral $ I_\mathbf{S}$, since $|r_0-\frac{2}{3}r_0|\geq\frac{1}{3}r_0$, the bound \eqref{eq:bdphir1} for the case $|r_0-\widetilde{r}|\geq \f12 \widetilde{r}$ implies that
\begin{align*}
I_\mathbf{S}&\les r_0^{-1}\sqrt{\mathcal{E}_{0, \ga}^R} r_0^{-\frac{1+\ga+\ep}{2}} .
\end{align*}
 To bound the bulk integral $I_\mathbf{B}$, we notice that $\frac{1}{3}\widetilde{r}\leq \frac{2}{3}r_0\leq s\leq \widetilde{r}\leq 2 r_0$. We proceed in the coordinates $(t, x)$ and we can show that
 \begin{align*}
  I_\mathbf{B}
   &\les \widetilde{r}^{-2} \int_{\frac{2}{3}r_0}^{\widetilde{r}}\int_{ S_{(t_0-\widetilde{r}, x_0)}(s )}  (|D\phi|r^{-1}+ r^{-2}|\phi|)(t, x_0+s\widetilde{\om}) s^2 ds d\widetilde{\om}\\
   &\les \widetilde{r}^{-2}\int_{B_{(t_0-\widetilde{r}, 0)({ r_0+\widetilde{r}})}}  |D\phi|r^{-1}+ r^{-2}|\phi| dx.
 \end{align*}
 By  \eqref{eq:bd4rphi2:x0}, we have
 \begin{align*}
   \int_{B_{(t_0-\widetilde{r}, 0)({ r_0+\widetilde{r}})}} r^{-2}|\phi| dx &=\int_0^{{ r_0+\widetilde{r}}}\int_{S_{(t_0-\widetilde{r}, 0)}(r)}|\phi| d\om dr\les\int_0^{{ r_0+\widetilde{r}}} (\int_{S_{(t_0-\widetilde{r}, 0)}(r)}|\phi|^2 d\om)^{\f12} dr\\
 &\les \sqrt{\mathcal{E}_{0, \ga}^R}\int_0^{{ r_0+\widetilde{r}}} r^{-\f12} (R-t_0+\widetilde{r})^{-\frac{\ga}{2}} dr\les   \sqrt{\mathcal{E}_{0, \ga}^R}  \widetilde{r}^{\frac{1-\ga}{2}}.
 \end{align*}
By \eqref{eq:bd4Dphi:B}, we have
 \begin{align*}
   \int_{B_{(t_0-\widetilde{r}, 0)({ r_0+\widetilde{r}})}}  |D\phi|r^{-1} dx &\les (\int_{B_{(t_0-\widetilde{r}, 0)({ r_0+\widetilde{r}})}}  |D\phi|^2 v_*^{\ga+\ep}dx)^{\f12} (\int_{B_{(t_0-\widetilde{r}, 0)({ r_0+\widetilde{r}})}}  r^{-2}v_*^{-\ga-\ep} dx)^{\f12}\\
   &\les \sqrt{\mathcal{E}_{0, \ga}^R}(R-t_0+\widetilde{r})^{\frac{\ep}{2}} (\int_0^{{ r_0+\widetilde{r}}} (R-t_0+\widetilde{r}-r)^{-\ga-\ep}dr)^{\f12}\\
    &\les \sqrt{\mathcal{E}_{0, \ga}^R}(R-t_0+\widetilde{r})^{\frac{\ep}{2}} ((R-t_0+\widetilde{r})^{1-\ga-\ep}-(R-t_0{ -r_0})^{1-\ga-\ep})^{\f12}\\
    &\les \sqrt{\mathcal{E}_{0, \ga}^R}(R-t_0+\widetilde{r})^{\frac{\ep}{2}} (R-t_0+\widetilde{r})^{\frac{-\ga-\ep}{2}}r_0^{\f12}\les \sqrt{\mathcal{E}_{0, \ga}^R}\widetilde{r}^{\frac{1-\ga}{2}},
 \end{align*}
where we use $v_*=R-(t_0-\widetilde{r})-r$, $r_0\leq \frac{3}{2}\widetilde{r}$, $0<\ep+\ga<1$ and $\ep>0$. Combining the above estimates, for the case $\f12 \widetilde{r}\leq r_0\leq \frac{3}{2}\widetilde{r}$, we conclude
 \begin{align*}
   \int_{ S_{(t_0-\widetilde{r}, x_0)}(\widetilde{r}) }|\phi|r^{-1}  d\widetilde{\om}\les \widetilde{r}^{-1}\sqrt{\mathcal{E}_{0, \ga}^R} \widetilde{r}^{-\frac{1+\ga+\ep}{2}}+\widetilde{r}^{-2}\sqrt{\mathcal{E}_{0, \ga}^R}\widetilde{r}^{\frac{1-\ga}{2}}\les \widetilde{r}^{-1}\sqrt{\mathcal{E}_{0, \ga}^R} \widetilde{r}^{-\frac{1+\ga+\ep}{2}} .
 \end{align*}
This completes the proof of the proposition.
\end{proof}

\section{Linear equation with infinite energy}
\label{sec:linear:th}

Let $w$ be the free wave with data prescribed { on the} initial hypersurface $\B_R:=\{|x|\leq R\}$ as follows:
\begin{equation}
\label{eq:lineareq:com}
\Box w=0,\quad w(0, x)=0,\quad \pa_t w(0, x)=w_1,\quad |x|\leq R.
\end{equation}
For $0<\ga<1$, the initial weighted energy of $w$ is given by
\begin{equation}\label{def: E gamma w B R}
\mathcal{E}_{\ga}[w](\B_{R})=\int_{|x|\leq R}|w|^2(R-|x|)^\ga +\sum\limits_{i,j}|\Om_{ij}w|^2(R-|x|)^\ga+|\nabla w|^2(R-|x|)^{2+\ga}dx.
\end{equation}
We have the following pointwise estimate for $w$:
\begin{Prop}
\label{prop:lineardecay:comp}
Let $w$ be the solution to the Cauchy problem \eqref{eq:lineareq:com}. Then, there exists { a} constant $C$ depending only on $\ga$ in such a way that
\begin{equation}\label{eq:lineardecay:comp}
|w(t, x)|^2\leq C  (R-t)^{-1-\ga} \mathcal{E}_{ \ga}[w_1](\B_R),\quad \forall \ (t, x)\in \mathcal{J}^+(\B_{R}).
\end{equation}
In particular, we emphasize that the constant $C$ does not depend on the radius $R$.
\end{Prop}
We remark that the solution may { blow} up at the vertex $(R, 0)$ of the cone $\mathcal{J}^{+}(\B_R)$. The linear estimate will be used to control the linear part of the full solution $(\phi, F)$ of the YMH system. Its proof relies on the following weighted versions of Hardy's inequality and Sobolev inequality.
\begin{Lem}
\label{lem:Hardy:com}
Given $0<\ga<1$ and $\ep>0$, we have the following type of Hardy's inequality: there exists a constant $C$ depending only on $\ep$ and $\ga$ so that for all smooth function $\varphi$ defined on $\B_{R_0}$, we have
\begin{equation*}\label{eq:Hardy:com}
\int_{\B_{R_0}}|\varphi|^2 (R_0-|x|)^{-1+\ep}dx\leq C R_0^{\ep -\ga+1} \int_{\B_{R_0}}(|\varphi|^2R_0^{-2}+|\nabla \varphi|^2)(R_0-|x|)^\ga dx.
\end{equation*}
\end{Lem}
\begin{proof}
We may assume $\varphi(x)=\varphi_1\left(\frac{x}{R_0}\right)$ and this reduces the problem to the case when $R_0=1$. Moreover, since the inequality is linear in $\varphi$, we may assume that
\[
\int_{\B_{1}}(|\varphi|^2+|\nabla \varphi|^2)(1-|x|)^\ga dx=1.
\]
We consider the following inequality
\begin{align*}
&\int_{\om}|\varphi|^2(r\om)(1-r)^{\ep}r^4 d\om +\ep\int_{0}^{r}\int_{\om}|\varphi|^2 (1-s)^{\ep-1}s^4 dsd\om \\
=&\int_{0}^{r}\int_{\om}|\varphi|^2 4(1-s)^{\ep}s^3+2 \varphi \cdot \pa_r\varphi (1-s)^{\ep}s^4  dsd\om\\
\leq& \int_{0}^{r}\int_{\om} 64\ep^{-1}|\varphi|^2 (1-s)^{\ga}s^2+16\ep^{-1}|\nabla\varphi|^2(1-s)^\ga s^2+\f12 \ep |\varphi|^2(1-s)^{2\ep-\ga}s^4   dsd\om.
\end{align*}
In view of the fact that $\ga<1$ and $\ep>0$, we set $r=1$ to derive
\begin{align*}
\ep\int_{0}^{1}\int_{\om}|\varphi|^2 (1-s)^{\ep-1}s^4 dsd\om \leq 160\ep^{-1}.
\end{align*}
On the other hand, we have
\begin{align*}
\int_{|x|\leq \f12}|\varphi|^2(1-|x|)^{\ep-1}dx \leq 2^{1+\ga-\ep}\int_{|x|\leq \f12}|\varphi|^2(1-|x|)^{\ga} dx \leq 2^{1+\ga-\ep}.
\end{align*}
Combining the above two inequalities, we conclude that
\begin{align*}
\int_{|x|\leq 1}|\varphi|^2(1-|x|)^{\ep-1}dx\leq 2^{1+\ga-\ep}+640\ep^{-2}.
\end{align*}
This proves the lemma.
\end{proof}
\begin{Lem}
\label{lem:weightedSob:com}
Given $0<\ga<1$ and $\ep>0$, we have the following type of Sobolev inequality: there exists a constant $C$ depending only on $\ga$ so that for all $R_0>0$, for all smooth function $\varphi$ defined on the ball $\B_{R_0}$ and for all  $x$ such that $|x|\leq R_0$, we have
\begin{equation}
\label{eq:weighteSob:com}
|\varphi|^2(x)\leq C R_0^{-1-\ga}\left(\mathcal{E}_{ \ga}[\nabla\varphi](\B_{R_0})+\int_{\B_{R_0}}|\varphi|^2(R_0-|x|)^{\ga}R_0^{-2}dx\right).
\end{equation}
 
\end{Lem}
\begin{proof}
We may assume $\varphi(x)=\varphi_1\left(\frac{x}{R_0}\right)$ and this reduces the problem to the case when $R_0=1$. It thus suffices to prove the lemma for $R_0=1$. Moreover, since the inequality is linear in $\varphi$, we may renormalize the right hand side to be $1$:
\[
\int_{\B_{1}}(|\varphi|^2+|\nabla\varphi|^2+|\nabla\Om_{ij}\varphi|^2)(1-|x|)^{\ga}+|\nabla^2\varphi|^2(1-|x|)^{2+\ga}dx=1.
\]
First consider the case where $|x|$ is small. We assume that $|x|\leq \frac{1}{2}$. The standard Sobolev inequality gives:
\begin{align*}
|\varphi|^2(x)\leq C \|\varphi\|_{H^2(\B_{\f12})}^2\leq 2^{2+\ga}C,
\end{align*}
where $C$ is a universal constant from the embedding inequality. In the rest of the proof, we will use the constant $C$ to denote a universal constant but it may vary in different places.

If $|x|$ is large, i.e., $|x|\geqslant \frac{1}{2}$, we will take advantage of the angular derivative $\Omega_{ij}$. According to Lemma \ref{lem:Hardy:com}, we have better decay for $\varphi$ and $\pa_\om\varphi$:
\begin{equation}
\label{eq:bd4varphiom}
\int_{|x|\leq 1}(|\varphi|^2+|\pa_\om\varphi|^2)(1-|x|)^{\ep-1}dx\leq C_{\ep, \ga},
\end{equation}
where $C_{\ep, \ga}$ depending only on $\ep$ and $\ga$. The notation $\pa_\om$ denotes the derivatives of the angular momentums $\Om_{ij}$. We use the Poincar\'e inequality on the unit sphere to derive
\begin{align*}
\int_{\om}|\varphi|^6d\om \leq C\int_{\om}|\varphi|^2+|\pa_\om\varphi|^2d\om \cdot \int_{\om}|\varphi|^4d\om.
\end{align*}
For some fixed constant $r_0\in[\frac{1}{3}, \f12]$, for all $r_0\leq r \leq 1$, we integrate between $r_0$ and $r$ to derive
\begin{align*}
\int_{\om}|\varphi|^4(r\om)d\om&=\int_{\om}|\varphi|^4(r_0\om)d\om+4\int_{r_0}^{r}\pa_r\varphi \cdot \varphi^3 d\om dr\\
&\leq \int_{\om}|\varphi|^4(r_0\om)d\om+C\left(\int_{r_0}^{r}\int_{\om}|\pa_r\varphi|^2(1-|x|)^\ga drd\om \right)^\f12\left(\int_{r_0}^{r}\int_{\om}|\varphi|^6(1-|x|)^{-\ga} drd\om \right)^\f12\\
&\leq \int_{\om}|\varphi|^4(r_0\om)d\om+C \sup\limits_{r_0\leq s\leq r}\left(\int_{\om}|\varphi|^4(s\om)d\om\right)^\f12 \left(\int_{r_0}^{r}\int_{\om}(|\varphi|^2+|\pa_\om \varphi|^2)(1-|x|)^{-\ga} drd\om \right)^\f12.
\end{align*}
We set $r_0=\f12$ in the above inequality and set $\ep=1-\ga$ in \eqref{eq:bd4varphiom}. Therefore, for $\f12 \leq r\leq 1$, we have
\[
\int_{\om}|\varphi|^4(r\om)d\om\leq C_\ga,\]
where the constant $C_\ga$ depends only on $\ga$ (coming from the bound of $\varphi$ on the sphere of radius $\f12$).

For $\pa_\om \varphi$, we do not have the same bound due to the lack of $\pa_\om \pa_\om \varphi$. By Sobolev inequality, we first have
\[
\left(\int_{\frac{1}{3}}^{\frac{1}{2}}\int_{\om}|\pa_\om \varphi|^4 r^2 dr d\om\right)^\f12 \leq C\int_{\frac{1}{3}}^{\f12}\int_{\om} |\pa_\om \varphi|^2+|\nabla\pa_\om \varphi|^2 r^2  drd\om\leq C.
\]
In particular, we can choose some $r_0\in[\frac{1}{3}, \frac{1}{2}]$ in such a way that
\[
\int_{\om}|\pa_\om\varphi|^4(r_0\om)d\om \leq C.
\]
We then replace $\varphi$ by $\pa_\om \varphi$ in the previous $L^4$ estimate. This leads to
\begin{align*}
\int_{\om}|\pa_\om\varphi|^4d\om &\leq C+\sup\limits_{r_0\leq s\leq r}\left(\int_{\om}|\pa_\om\varphi|^4(s\om)d\om\right)^\f12 \left(\int_{r_0}^{r}\int_{\om}(|\pa_\om\varphi|^2+|\nabla \pa_\om \varphi|^2)(1-|x|)^{-\ga} drd\om \right)^\f12\\
&\leq C+\sup\limits_{r_0\leq s\leq r}\left(\int_{\om}|\pa_\om\varphi|^4(s\om)d\om\right)^\f12 (1-r)^{-\ga}.
\end{align*}
Hence, for all $r_0\leq r\leq 1$, we have
\begin{align*}
\int_{\om}|\pa_\om \varphi|^4d\om \leq C (1-r)^{-2\ga}.
\end{align*}
For any $2\leq p\leq 4$, similar to the previous step, for $r\geq r_0$, we have
\begin{align*}
\int_{\om}|\pa_\om\varphi|^p(r\om) d\om\leq& \int_{\om}|\pa_\om \varphi|^p(r_0\om) d\om\\
&+C_p\left(\int_{r_0}^r \int_{\om}|\pa_r\pa_\om \varphi|^2(1-|x|)^{\ga} drd\om\right)^\f12\left(\int_{r_0}^{r}\int_{\om }|\pa_\om \varphi|^{2p-2}(1-|x|)^{-\ga}   drd\om \right)^\f12.
\end{align*}
We take $p=2+\ep$, $\ep=\frac{1-\ga}{3}$ and interpolate between $L^2$ and $L^4$. This yields
\begin{align*}
\int_{\om}|\pa_\om\varphi|^p(r\om) d\om &\leq C+C\left(\int_{r_0}^{r}\int_{\om }|\pa_\om \varphi|^{2}(1-|x|)^{\ep-1}   drd\om\right)^{\frac{1-\ep}{2}}\left(\int_{r_0}^{r}\int_{\om }|\pa_\om \varphi|^{4}(1-|x|)^{1+\ep}   drd\om\right)^{\frac{\ep}{2}}\\
&\leq C+C_\ga \left(\int_{r_0}^{r}(1-s)^{1+\ep-2\ga}ds\right)^{\f12\ep}\leq C_\ga,
\end{align*}
where we used $\ga<1$ and the $L^4$ bound on $\pa_\om \varphi$. Since $p=2+\ep>2$, this estimate together with the above uniform $L^4$ bound for $\varphi$ proves \eqref{eq:weighteSob:com}.
\end{proof}
\begin{Remark}
\label{remark:coWS}
The above lemma of weighted type Sobolev inequality holds if  $\varphi$ is a $V$-valued or $\lg$-valued function and $\nabla$ is replaced with
the covariant derivative $D$ associated to a connection $A$. This is because
\[|\pa |\varphi|^2 |=|\langle D\varphi, \varphi\rangle+\langle\varphi, D\varphi\rangle |\leq { 2}|\varphi| |D\varphi|.\]
Hence, $|\pa |\varphi||\leq |D\varphi|$. { Similarly $|X |\varphi||\leq |D_X\varphi|.$}
\end{Remark}
\begin{proof}[Proof of Proposition \ref{prop:lineardecay:comp}]
By linearity, we may assume
\[
\int_{|x|\leq R}(|w_1|^2  +\sum\limits_{i,j}|\Om_{ij}w_1|^2)(R-|x|)^\ga dx=1.
\]
We use $B_{(t_0,x_0)}(\widetilde{r})$ to denote the ball centered at $x_0$ with radius $\widetilde{r}$ in $\{t=t_0\}$. According to the energy estimates for the linear wave equation, we have
\begin{align*}
\int_{B_{(t, 0)}(r)}|\pa w|^2 dx+\int_{r}^{t+r}\int_{\om}|(\pa_u, \nabb) w|^2(t+r-s, r\om) s^2 ds d \om\leq \int_{B_{(0, 0)}(t+r)}|\pa w|^2 dx,
\end{align*}
where $\pa$ denotes for the full derivatives $(\pa_t, \nabla)=(\pa_v,\pa_u, \nabb)$. We multiply both sides of the above inequality by $(R-t-r)^{\ga-1}$ and integrate $r$ from $0$ to $R-t$. This gives
\begin{align*}
\int_{B_{(t, 0)}(R-t)}|\pa w|^2(R-t-|x|)^\ga dx&=\ga \int_{0}^{R-t}(R-t-r)^{\ga-1}\int_{B_{(t, 0)}(r)}|\pa w|^2 dx dr\\
&\leq \ga\int_{0}^{R-t}(R-t-r)^{\ga-1} \int_{s\leq r+t}\int_{\om }|w_1|^2 s^2 dsd\om \\
&=\int_{\B_{R}}|w_1|^2(R-\max\{|x|, t\})^{\ga}dx\leq 1.
\end{align*}
We commute the $\Om_{ij}$ with \eqref{eq:lineareq:com} and we repeat the above process to derive
\[
\int_{B_{(t, 0)}(R-t)}|\pa\Om_{ij} w|^2(R-t-|x|)^\ga dx\leq 1.
\]
Similarly, by commuting $\nabla$ with \eqref{eq:lineareq:com} , we have
\begin{align*}
&\int_{B_{(t, 0)}(R-t)}|\pa \nabla w |^2(R-t-|x|)^{2+\ga}dx \leq \int_{\B_{R}}|\nabla w_1|^2(R-\max\{|x|, t\})^{\ga+2}dx\leq 1.
\end{align*}
To apply Lemma \ref{lem:weightedSob:com} to bound $w(t, x)$ pointwise, it remains to derive the weighted $L^2$ estimate for the solution $w$. It relies on the energy flux through the cone $\mathcal{N}^{-}(t+r, 0)$. Indeed, since $w$ vanishes on $\B_R$, we have
\begin{align*}
\int_{\om}|w|^2(t, r\om) d\om &\leq \int_{\om}\left(\int_{r}^{r+t}|\pa_u w|(t+r-s, s\om)ds\right)^2d\om\\
&\leq \int_{\om} \int_{r}^{r+t}|\pa_u w|^2(t+r-s, s\om)s^2 \cdot r^{-1} ds  d\om\leq r^{-1} \int_{B_{(0, 0)}(t+r)}|w_1|^2dx.
\end{align*}
 We multiply both sides of the above inequality by $(R-t-r)^{\ga-1}$ to derive
\begin{align*}
\int_{0}^{R-t}\int_{\om}|w|^2(t, r\om)(R-t-r)^{\ga-1} rd\om dr \leq \int_{0}^{R-t}(R-t-r)^{\ga-1}\int_{|x|\leq r+t}|w_1|^2 dx dr\leq \ga^{-1}.
\end{align*}
In particular we have
\begin{align*}
\int_{B_{(t, 0)}(R-t)}|w|^2(R-t-r)^{\ga}dx\leq (R-t)^2 \int_{0}^{R-t}\int_{\om}|w|^2(R-t-r)^{\ga-1} rd\om dr\leq \ga^{-1}(R-t)^2.
\end{align*}
Therefore, \eqref{eq:lineardecay:comp} follows immediately from Lemma \ref{lem:weightedSob:com}. This completes the proof of Proposition \ref{prop:lineardecay:comp}.
\end{proof}

\section{The pointwise estimates for the solutions}\label{sec:5}
We will apply the weighted energy estimates from Section \ref{sec:weightedEE} to prove Theorem \ref{thm:EM} in this section. We need to derive wave type equations for $D\phi$ and $F$. For the Higgs field $D\phi$, we simply commute the covariant wave equation with $D$. For the Yang-Mills field $F$, we use the Bianchi identity $D_[\ga F_{\mu\nu]}=0$ together with \eqref{eq:YMH} to derive
\begin{align*}
  D^\ga D_{\ga}F_{\mu\nu}&=-D^{\ga}D_{\mu}F_{\nu\ga}-D^\ga D_{\nu}F_{\ga\mu}=-[D^\ga, D_\mu]F_{\nu\ga}-D_{\mu}D^\ga F_{\nu\ga}-
  [D^\ga, D_\nu]F_{\ga\mu}-D_{\nu}D^\ga F_{\ga\mu}\\
   &=-[F^\ga_{\ \ \mu}, F_{\nu\ga}]-[F^\ga_{\ \ \nu}, F_{\ga\mu}]-D_\nu J[\phi]_{\mu} +D_\mu J[\phi]_{\nu}=2[F_{\mu\ga}, F_{\nu}^{\ \ga}] +D_\mu J[\phi]_{\nu}-D_\nu J[\phi]_{\mu},
\end{align*}
where $J[\phi]_{\mu}=\langle \cT\phi, D_\mu \phi \rangle +\langle  D_\mu\phi, \cT\phi \rangle$. We conclude that
\begin{equation} \label{eq:EQDphiF}
\begin{cases}
\Box_A D\phi=[\Box_A, D]\phi,\\
\Box_A F_{\mu\nu}=2[F_{\mu\ga}, F_{\nu}^{\ \ga}] +D_\mu J[\phi]_{\nu}-D_\nu J[\phi]_{\mu},
\end{cases}
\end{equation}
where $\Box_A=D^\mu D_\mu$. This system is covariant and is independent of the choice of the coordinates system.

The initial data may not be uniformly bounded in $H^2$. This prevents one using the celebrated bilinear estimates of Klainerman-Machedon (see \cite{MKGkl}) to control the nonlinearities. Instead, we rely on a refinement of the approach developed in \cite{Moncrief1} and \cite{Moncrief2} by Eardley and Moncrief. However, the choice of the Cr\"{o}nstrom gauge  in \cite{Moncrief1} and \cite{Moncrief2} may cause a loss of regularity. We instead use the Kirchoff-Sobolev paramatrix for the covariant wave equation introduced by Klainerman and Rodnianski in \cite{Kl:paramatrix}.

We first control the linear evolution with data bounded in the norm $\mathcal{E}_{2, \ga}^{R}$ (see \eqref{def:E_2 gamma R} for the definition).
\begin{Lem}
\label{lem:decay:lin:com}
Let $(\phi, F)$ be a solution to the YMH system on $\mathcal{J}^+(\B_R)$ so that $\mathcal{E}_{2, \ga}^{R}$ is finite. There exists a constant $C$ depending only on $R$, $\ga$ and $\ep$, so that for all $(t_0, x_0)\in\mathcal{J}^+(\B_R)$, we have the following bounds on the linear evolution:
\begin{align}
\label{eq:bd4phi:lin:com}
\int_{\widetilde{\om}}t_0|D\phi|(0, x_0+t_0\widetilde{\om})+|\phi|(0, x_0+t_0\widetilde{\om}) d\widetilde{\om}&\leq C(R-t_0)^{-\frac{1+\ga}{2}}\sqrt{\mathcal{E}_{1, \ga}^{R}(1+\mathcal{E}_{0, \ga}^R)},\\
\label{eq:bd4DphiF:lin:com}
\sum\limits_{k\leq 1}t_0^{k}\int_{\widetilde{\om}}(|D^{k+1}\phi|+ |D^k F|)(0, x_0+t_0\widetilde{\om})d\widetilde{\om} &\leq C { \sqrt{\mathcal{E}_{2, \ga}^R }(1+\mathcal{E}_{0, \ga}^R) }(R-t_0-r_0)^{-\frac{1+\ga}{2}-\ep}(R-t_0)^{-1+\ep}.
\end{align}
\end{Lem}
\begin{proof}
For $|\phi(0, x_0+t_0\widetilde{\om})|$ and $r_0=|x_0|$, by the Remark \ref{remark:coWS} and the definition of $\mathcal{E}_{1, \ga}^R$, we have
\[
|\phi|(0, x)\les \sqrt{\mathcal{E}_{1, \ga}^R},
\]
where $|x|\leq R$ and the implicit constant depends only $\ga$, $0<\ep<10^{-2}(1-\ga)$ and $R$. This controls the second integrand in \eqref{eq:bd4phi:lin:com}. For the first { integral} involving $D\phi$, we use the linear wave equation on $\mathcal{J}^+(\B_R)$:
\[
\Box w=0,\quad w(0, x)=0, \quad \pa_t w(0, x)=|D\phi|(0, x).
\]
By the representation formula for free waves, we have
\begin{align*}
4\pi w(t_0, x_0)=t_0\int_{\widetilde{\om}}\pa_t w(0, x_0+t_0\widetilde{\om})d\om=t_0\int_{\widetilde{\om}}|D\phi|(0, x_0+t_0\widetilde{\om})d\widetilde{\om}.
\end{align*}
Therefore to prove \eqref{eq:bd4phi:lin:com}, it suffices to bound $w(t_0, x_0)$. In view of Proposition \ref{prop:lineardecay:comp}, it suffices to bound
$\mathcal{E}_{\ga}[\pa_t w(0, x)](\B_R)$ (see \eqref{def: E gamma w B R}). Indeed, since 
$$ |\pa |D\phi|| \leq  |DD\phi|,\ |\Om_{ij} |D\phi|| \leq  |D_{\Om_{ij}}D\phi|, $$
 we then have
\begin{align*}
\mathcal{E}_{\ga}[\pa_t w(0, x)](\B_R)&=\int_{\B_{R}}(|D \phi|^2+|\Om_{ij}|D\phi||^2)(R-|x|)^\ga+|\nabla |D\phi||^2 (R-|x|)^{2+\ga}dx\\
&\leq  \int_{\B_{R}}(|D \phi|^2+|D_{\Om_{ij}}D\phi|^2)(R-|x|)^\ga+|D D\phi|^2 (R-|x|)^{2+\ga}dx.
\end{align*}
By commuting $D_{\Om_{ij}}$ and $D$, we have $$ |D_{\Om_{ij}}D\phi|^2 \les |DD_{\Om_{ij}}\phi|^2+|F|^2|\phi|^2.$$
Then by the definition of $\mathcal{E}_{1, \ga}^R$ and the above bound for $|\phi(0, x)|$, we derive that
\begin{align*}
\mathcal{E}_{\ga}[\pa_t w(0, x)](\B_R)
&\les  \mathcal{E}_{1, \ga}^R+ \int_{\B_{R}}|F|^2|\phi|^2(R-|x|)^\ga dx\les \mathcal{E}_{1, \ga}^R+ \mathcal{E}_{1, \ga}^R\int_{\B_{R}}|F|^2 (R-|x|)^\ga dx \les \mathcal{E}_{1, \ga}^R(1+ \mathcal{E}_{0, \ga}^R).
\end{align*}
Therefore, Proposition \ref{prop:lineardecay:comp} and Lemma \ref{lem:weightedSob:com} imply that
\[
w(t_0, x_0)=\frac{t_0}{4\pi}\int_{\widetilde{\om}}|D\phi|(0, x_0+t_0\widetilde{\om})d\widetilde{\om}\les \sqrt{\mathcal{E}_{1, \ga}^R(1+\mathcal{E}_{0, \ga}^R)} (R-t_0)^{-\frac{1+\ga}{2}}.
\]
This finishes the estimate in \eqref{eq:bd4phi:lin:com}.

To prove \eqref{eq:bd4DphiF:lin:com}, we start with a pointwise bound for the initial Yang-Mills field $F(0, x)$.  For $|x|<R$, let
\[\varphi(x)=(R-|x|)^{\frac{1+\ga}{2}+\ep}F(0, x).\]
We have
\begin{align*}
&\int_{\B_R}(R-|x|)^{1-2\ep}(|\varphi|^2+|D\varphi|^2+|D D_{\Om_{ij}}\varphi|^2)+|DD\varphi|^2(R-|x|)^{3-2\ep} dx\\
\les& \sum\limits_{ l_1+l_2\leq 2}\sum\limits_{i, j}\int_{\B_{R}}|D^{l_1} D_{\Om_{ij}}^{l_2}F|^2(R-|x|)^{\ga+2l_1} dx+\mathcal{E}_{0, \ga}^R \les \mathcal{E}_{2, \ga}^R .
\end{align*}
Therefore, by Remark \ref{remark:coWS}, we derive that
\[
|F(0, x_0+t_0\widetilde{\om})|\les \sqrt{\mathcal{E}_{2, \ga}^R}(R-|x_0+t_0\widetilde{\om}|)^{-\frac{1+\ga}{2}-\ep}\les \sqrt{\mathcal{E}_{2,\ga}^R}(R-t_0-r_0)^{-\frac{1+\ga}{2}-\ep}.
\]
To bound the integral of $|D F|$, we apply Proposition \ref{prop:lineardecay:comp} to the following free wave equation:
\[
\Box \widetilde{w}=0,\ \widetilde{w}(0, x)=0, \ \pa_t \widetilde{w}(0, x)=(R-|x|)^{\frac{1+\ga}{2}+\ep} |D F(0, x)|,
\]
where  $|x|\leq R$. We consider weighted Sobolev norms for $F$ and we derive that
\begin{align*}
\mathcal{E}_{1-2\ep}[\pa_t \widetilde{w}(0, x)](\B_{R})&\les \sum\limits_{l_1+l_2\leq 1}\sum\limits_{i, j}\int_{|x|\leq R}|D^{l_1} D_{\Om_{ij}}^{l_2}D F|^2 (R-|x|)^{\ga+2l_1+2}dx\\
&\les \mathcal{E}_{2, \ga}^R+ \sum\limits_{i, j}\int_{|x|\leq R}|  [D_{\Om_{ij}},D] F|^2 (R-|x|)^{\ga+2}dx\\
&\les \mathcal{E}_{2, \ga}^R+ \mathcal{E}_{2, \ga}^R\sum\limits_{i, j}\int_{|x|\leq R}|  F|^2 (R-|x|)^{\ga}dx\les \mathcal{E}_{2, \ga}^R(1+\mathcal{E}_{0, \ga}^R).
\end{align*}
By using Proposition \ref{prop:lineardecay:comp} and the representation formula for $\widetilde{w}$, we derive that
\begin{align*}
t_0\int_{\widetilde{\om}}{ (R-|x_0+t_0\widetilde{\om}|)^{\frac{1+\ga}{2}+\ep}}|D F|(0, x_0+t_0\widetilde{\om})d\widetilde{\om} \les \sqrt{\mathcal{E}_{2, \ga}^R(1+\mathcal{E}_{0, \ga}^R)}(R-t_0)^{-1+\ep}.
\end{align*}
The Yang-Mills field part in \eqref{eq:bd4DphiF:lin:com} follows immediately as $|x|=|x_0+t_0\widetilde{\om}|\leq t_0+r_0$.

It remains to prove the Higgs field part  in \eqref{eq:bd4DphiF:lin:com}. It relies on the previous bound on the Yang-Mills field since commuting $D_{\Om_{ij}}$ with $D$ generates terms with the Yang-Mills field. We need the following auxiliary bound:
\begin{equation}
  \label{eq:bd:ID:DFphi}
  \int_{\B_{R}}|DF|^2|\phi|^2(R-|x|)^{\ga+2}dx\les \mathcal{E}_{2, \ga}^R(1+\mathcal{E}_{0, \ga}^R)^{2}.
\end{equation}
We split the integral in \eqref{eq:bd:ID:DFphi} into the two regions $\mathcal{B}_{\rm in}=\{|x|\leq \f12 R\}$ and $\mathcal{B}_{\rm out}=\{|x|\geq \f12 R\}$. For the inner region $\mathcal{B}_{\rm in}$, the weight $(R-|x|)$ has a lower bound and the estimate follows by the standard Sobolev inequality. More precisely, we have
\begin{align*}
  \int_{\mathcal{B}_{\rm in}}|DF|^2|\phi|^2(R-|x|)^{\ga+2}dx &\les \int_{\mathcal{B}_{\rm in}}|DF|^2|\phi|^2 dx
  \les \left(\int_{\mathcal{B}_{\rm in}}|D^2F|^2+|DF|^2 dx\right)\left( \int_{\mathcal{B}_{\rm in}}|D\phi|^2+|\phi|^2 dx\right)\\
  &\les \mathcal{E}_{2, \ga}^R \mathcal{E}_{0, \ga}^R.
\end{align*}
For the outer region $\mathcal{B}_{\rm out}$, repeating the argument to bound $\int_{\om}|\pa_{\om}\varphi|^p d\om$ in the proof of Lemma \ref{lem:weightedSob:com}, for all $r\geq \f12 R$, $p=2+\epsilon'$, $\epsilon'=\frac{1-\gamma}{3}$ we have
\begin{align*}
  \int_{\om}|\phi(0, r\om)|^pd\om \les (\mathcal{E}_{0, \ga}^R)^{p/2}.
\end{align*}
By  Sobolev inequality, we can bound that (here $q=\frac{p}{2}=1+\frac{1-\gamma}{6}$, $q'=\frac{q}{q-1}=1+\frac{6}{1-\gamma}$)
\begin{align*}
  &\int_{\mathcal{B}_{\rm out}} |DF|^2|\phi|^2(R-|x|)^{2+\ga}dx
  \les \int_{\f12R}^{R}s^2(R-s)^{2+\ga }(\int_{\om} |DF|^{2q'}d\om)^{\frac{1}{q'}} (\int_{\om} |\phi|^{2q} d\om )^{\frac{1}{q}}ds\\
 &\les \mathcal{E}_{0, \ga}^R \int_{\f12R}^{R}s^2(R-s)^{2+\ga}\int_{\om} |DF|^{2} +|D_{\Om}DF|^2d\om   ds\\
 &\les \mathcal{E}_{0, \ga}^R \int_{\B_{R}}(R-|x|)^{2+\ga }(|DF|^{2} +|D D_{\Om}F|^2+ |[D, D_{\Om}]F|^2)dx\les \mathcal{E}_{0, \ga}^R \mathcal{E}_{2, \ga}^R{ (1+\mathcal{E}_{0, \ga}^R)}.
\end{align*}
We have used the pointwise bound for $F$ in the last step to bound the commutator.

We turn to the Higgs field part in \eqref{eq:bd4DphiF:lin:com}. To bound $|D\phi(0, x)|$ pointwise,  for $|x|<R$, we define
\[\widetilde{\varphi}(x)=(R-|x|)^{\frac{1+\ga}{2}+\ep}D\phi(0, x).\]
In view of the bounds for $F(0, x)$ and \eqref{eq:bd:ID:DFphi}, we show that
\begin{align*}
&\int_{\B_R}(R-|x|)^{1-2\ep}(|\widetilde{\varphi}|^2+|D\widetilde{\varphi}|^2+|D D_{\Om_{ij}}\widetilde{\varphi}|^2)+|DD\widetilde{\varphi}|^2(R-|x|)^{3-2\ep} dx\\
\les& \sum\limits_{l\leq 1, l_1+l_2\leq 2}\sum\limits_{i, j}\int_{\B_{R}}|D^{l_1+1} D_{\Om_{ij}}^{l_2}\phi|^2(R-|x|)^{\ga+2l_1}+|D^{l}[D_{\Om_{ij}}, D]\phi|^2 (R-|x|)^{\ga+2l}dx+\mathcal{E}_{0, \ga}^R\\
\les& \mathcal{E}_{2, \ga}^R+\int_{\B_{R}}(|DF|^2|\phi|^2+|F|^2|D\phi|^2) (R-|x|)^{\ga+2}+|F|^2|\phi|^2 (R-|x|)^{\ga}dx \les \mathcal{E}_{2, \ga}^R(1+\mathcal{E}_{0, \ga}^R)^{2}.
\end{align*}
Similar to the case for $|F(0, x)|$, the pointwise estimate for $|D\phi(0, x)|$ follows by Lemma \ref{lem:weightedSob:com}.

Finally,  to bound $|DD\phi|$ in \eqref{eq:bd4DphiF:lin:com}, we use the following linear wave equation:
\[
\Box \widetilde{W}=0,\ \widetilde{W}(0, x)=0, \ \pa_t \widetilde{W}(0, x)=(R-|x|)^{\frac{1+\ga}{2}+\ep}|DD\phi(0, x)|, \quad |x|\leq R.
\]
We can proceed as before:
\begin{align*}
\mathcal{E}_{1-2\ep}[\pa_t \widetilde{W}(0, x)](\B_{R})&\les \sum\limits_{l_1+l_2\leq 1}\sum\limits_{i, j}\int_{|x|\leq R}|D^{l_1} D_{\Om_{ij}}^{l_2}DD\phi|^2(R-|x|)^{\ga+2l_1+2}dx\\
&\les \mathcal{E}_{2, \ga}^R+ \sum\limits_{i, j}\int_{|x|\leq R}( |[D_{\Om_{ij}}, D]D\phi|^2+| D [D_{\Om_{ij}}, D]\phi|^2)(R-|x|)^{\ga+ 2}dx\\
&\les \mathcal{E}_{2, \ga}^R+ \sum\limits_{i, j}\int_{|x|\leq R}( |F|^2|D\phi|^2+| D F|^2|\phi|^2+|F|^2|D\phi|^2)(R-|x|)^{\ga+ 2}dx\\
&\les \mathcal{E}_{2, \ga}^R(1+\mathcal{E}_{0, \ga}^R)^{2}.
\end{align*}
Once again, by  Proposition \ref{prop:lineardecay:comp} and the representation formula for $\widetilde{W}$, we derive that
\begin{align*}
t_0\int_{\widetilde{\om}}{ (R-|x_0+t_0\widetilde{\om}|)^{\frac{1+\ga}{2}+\ep}}|DD\phi(0, x_0+t_0\widetilde{\om})| d\widetilde{\om} \les { \sqrt{\mathcal{E}_{2, \ga}^R}(1+\mathcal{E}_{0, \ga}^R)}(R-t_0)^{-1+\ep}.
\end{align*}
This finishes the proof of \eqref{eq:bd4DphiF:lin:com} since $|x|=|x_0+t_0\widetilde{\om}|\leq t_0+r_0$.
\end{proof}

We also need the following Lemma (alternatively see Lemma 5.1 in \cite{yang:NLW:ptdecay:3D}).
\begin{Lem}
\label{lem:bd:vga}
Fix a point $(t_0, x_0)\in \mathcal{J}^+(\B_R)$. For all $0\leq \widetilde{r}\leq t_0$, $t=t_0-\widetilde{r}$, $r=|x_0+\widetilde{r}\widetilde{\om}|$ and $0<\ga <1$,
we have the following estimate:
 
\begin{equation*}
\int_{S_{(t, x_0)}(\widetilde{r})}(R-t-r)^{-\ga}d\widetilde{\om}\les (R-t)^{\ga}(R-t_0+r_0)^{-\ga} \widetilde{r}^{-\ga}.
\end{equation*}
\end{Lem}
\begin{proof}
Let $r_0=|x_0|$, $u_*=R-t_0+r_0$ and $v_*=R-t_0-r_0$. From $t_0+r_0\leq R$, we have $r\leq R-t$. Hence,
\[
(R-t-r)^{-\ga}\les \left((R-t)^2-r^2\right)^{-\ga}(R-t)^{\ga}.
\]
Using $r^2=r_0^2+\widetilde{r}^2+2 \widetilde{r} x_0\cdot \widetilde{\om}$, { we compute} that
\begin{align*}
\int_{|\widetilde{\om}|=1}(R-t-r)^{-\ga}d\widetilde{\om}&\les (R-t)^{\ga} \int_{-1}^{1}((R-t)^2-{ r_0^2}-\widetilde{r}^2-2\widetilde{r}r_0\tau)^{-\ga}d\tau\\
&\les (R-t)^{\ga} (r_0 \widetilde{r})^{-1} (u_*^{1-\ga}(v_*+2\widetilde{r})^{1-\ga}-v_*^{1-\ga}(u_*+2\widetilde{r})^{1-\ga}).
\end{align*}
By definition, we have $u_*(v_*+2\widetilde{r})-v_*(u_*+2\widetilde{r})=4\widetilde{r}r_0$. As $0<\gamma<1$, we derive that
\[u_*^{1-\ga} (v_*+2\widetilde{r})^{1-\ga}-v_*^{1-\ga}(u_*+2\widetilde{r})^{1-\ga}\les  \widetilde{r}r_0 (u_* v_*+(u_*+v_*)\widetilde{r})^{-\ga}.\]
The lemma then follows because $
  (u_* v_*+(u_*+v_*)\widetilde{r})^{-\ga}\les u_*^{-\gamma}\widetilde{r}^{-\gamma}$.
\end{proof}
 
 \subsection{The proof of Theorem \ref{thm:EM}}\label{sec:proof of 5.1}

To use the Kirchoff-Sobolev paramatrix introduced by Klainerman and Rodnianski, we work in the coordinates $(\widetilde{t}, \widetilde{x})$ and with the null frame $(\widetilde{L}, \widetilde{\Lb}, \widetilde{e}_1, \widetilde{e}_2)$ which are centered at the fixed point $q=(t_0, x_0)\in \mathcal{J}^{+}(\B_{R})$. Let $\widetilde{\lap}_A=D^{\widetilde{e}_1}D_{\widetilde{e}_1}+D^{\widetilde{e}_2}D_{\widetilde{e}_2}$ be the projection of the covariant Laplacian on the sphere with constant radius $\widetilde{r}=|\widetilde{x}|$. Given $\mathbf{J}_{q} \in V$, let $h$ be $V$-valued function so that, on $\mathcal{N}^{-}(q)$, we have
\begin{align*}
D_{\widetilde{\Lb}}h=0, \quad h(q)=\mathbf{J}_{q}.
\end{align*}
The paramatrix in Theorem 4.1 in \cite{Kl:paramatrix} gives the following representation formula for $\phi$ as solution of the covariant linear wave equation $\Box_A \phi=0$:
\begin{equation}\label{eq:rep4phi}
\begin{split}
4\pi \langle \mathbf{J}_{q}, \phi(t_0, x_0) \rangle&=\int_{\B_{R}} \langle \widetilde{r}^{-1}h\delta(\widetilde{v}), D_0 \phi\rangle-\langle D_0(\widetilde{r}^{-1}h\delta(\widetilde{v})), \phi \rangle dx +\iint_{\mathcal{N}^{-}(q)}\langle\widetilde{\lap}_A h{-}\widetilde{\rho}\cdot h, \phi\rangle\widetilde{r} d\widetilde{r}d\widetilde{\om}.
\end{split}
\end{equation}
In the above formula, $\delta(\widetilde{v})$ is the Dirac distribution for $\widetilde{v}=\frac{\widetilde{t}+\widetilde{r}}{2}$, $\widetilde{\rho}=\f12 F(\widetilde{\Lb}, \widetilde{L})$ and $\widetilde{\rho}\cdot h$ is through the representation $d\beta:\lg\rightarrow \mathfrak{u}(V)$. Later on, we show that the first integral in \eqref{eq:rep4phi} has an upper bound depending only on the data although the function $h$ depends on $A$ and $q=(t_0, x_0)$. The second integral in \eqref{eq:rep4phi} comes from the connection field $A$. In order to control this integral, we need bounds for $h$ which will also be useful to control $D\phi$. We remark that although $h$ is gauge dependent, the norm $|h|$ is gauge invariant and it is clear that, on the cone $\mathcal{N}^{-}(q)$, we have
\[
|h|\leq |\mathbf{J}_{q}|.
\]
Without loss of generality, we assume $|\mathbf{J}_{q}|=1$. The formula \eqref{eq:rep4phi} depends only the value of $h$ and $\widetilde{\lap}_A h$ on the cone $\mathcal{N}^{-}(q)$. In what follows, it suffices to analyze $h$ on the cone.

Since $D_{\widetilde{\Lb}}h=0$ and $[\widetilde{L}, \widetilde{r}\widetilde{e}_j]=0$, we compute that
\begin{align*}
D_{\widetilde{\Lb}}(\widetilde{r} D_{\widetilde{e}_j}h)=[D_{\widetilde{\Lb}}, D_{\widetilde{r}\widetilde{e}_j}]h=\widetilde{r}\widetilde{\ab}_j \cdot h,
\end{align*}
where $\widetilde{r}\widetilde{e}_j$ is a linear combination of $\widetilde{\Om}_{ij}$ on each point. We can commute another $D_{\widetilde{r} \widetilde{e}_i}$ to derive
\begin{align*}
D_{\widetilde{\Lb}}(\widetilde{r}^2\widetilde{\lap}_A h)=2 \widetilde{r}^2\widetilde{\ab}_i \cdot D^{\widetilde{e}_i}h+\widetilde{r}^2  D^{\widetilde{e}_j} \widetilde{\ab}_j\cdot h.
\end{align*}
We can rewrite \eqref{eq:YMH} in the null frame (see Lemma 2 of \cite{yangMKG} for the similar computation). Hence,
\begin{align*}
D^{e_j} \widetilde{\ab}_j=J[\phi]_{\widetilde{\Lb}}+\widetilde{r}^{-2}D_{\widetilde{\Lb}}(\widetilde{r}^2\widetilde{\rho}).
\end{align*}
This leads to the transport equation
\begin{align*}
D_{\widetilde{\Lb}}(\widetilde{r}^2\widetilde{\lap}_A h- \widetilde{r}^2\widetilde{\rho} \cdot h)=2 \widetilde{r}^2\widetilde{\ab}_i \cdot D^{\widetilde{e}_i}h+\widetilde{r}^2  J[\phi]_{\widetilde{\Lb}}\cdot h.
\end{align*}
By integrating from the vertex $q$ to the sphere $S_{(t_0-\widetilde{r}, x_0)}(\widetilde{r})$ which will be shorted as $S(\widetilde{r})$ in the sequel, we derive that
\begin{align*}
\int_{S(\widetilde{r})}|\widetilde{r}^2\widetilde{\lap}_A h- \widetilde{r}^2\widetilde{\rho}\cdot  h|d\widetilde{\om}\les \int_{0}^{\widetilde{r}}\int_{S(s)}(|\widetilde{\D} h||\widetilde{\ab}|+|J[\phi]_{\widetilde{\Lb}}|) s^2dsd\widetilde{\om},
\end{align*}
where $\widetilde{\D}=(D_{\widetilde{e}_1}, D_{\widetilde{e}_2})$. From the transport equation for $D_{\widetilde{e}_i}h$, we see that the structure of $|D_{\widetilde{e}_i}h||\widetilde{\ab}|$ is similar to {that} of $|J[\phi]_{\widetilde{\Lb}}|$.

We first derive a bound on $D_{\widetilde{e}_i}h$. The transport equation for $D_{\widetilde{e}_i}h$ together with the improved weighted energy flux bound \eqref{eq:comp:v:EF} for $\widetilde{\ab}_i$ implies that
\begin{align*}
\int_{S(\widetilde{r})}|\widetilde{r} D_{\widetilde{e}_j}h|^2 d\widetilde{\om}&\leq \int_{\widetilde{\om}}(\int_{0}^{\widetilde{r}}| s\widetilde{\ab}_j \cdot h|ds )^2 d\widetilde{\om}\\
 &\les \int_{\widetilde{\om}}\left(\int_{0}^{\widetilde{r}} |s\widetilde{\ab}|^2(R-(t_0-s)-r_0)^{\ga} ds\right)\left( \int_{0}^{\widetilde{r}}(R-(t_0-s)-r_0)^{-\ga}ds\right)d\widetilde{\om}\\
 &\les \int_{\mathcal{N}^{-}(q)} |\widetilde{\ab}|^2(R-t-r_0)^{\ga} s^2ds d\widetilde{\om} \cdot \int_{0}^{\widetilde{r}}s^{-\ga}ds d\widetilde{\om} \les  \mathcal{E}_{0, \ga}^R\widetilde{r}^{1-\ga},
\end{align*}
where we used $t=t_0-s$, $t_0+r_0\leq R$ and the uniform bound for $h$. Therefore,
\begin{align*}
\int_{\mathcal{N}^{-}(q)}|D_{\widetilde{e}_j}h|^2 \widetilde{r}d\widetilde{r}d\widetilde{\om}\les  \mathcal{E}_{0, \ga}^R\int_{0}^{t_0} \widetilde{r}^{-\ga}d\widetilde{r}\les \mathcal{E}_{0, \ga}^R.
\end{align*}
We remark that this bound is the same as the bound \eqref{eq:bd4phi:need} for $\phi$.

We turn to the quadratic nonlinear term $J[\phi]$. Since the estimate in Proposition \ref{prop:EF:cone:gamma} is for the field $r\phi$ rather than $\phi$, we will write the $1$-form $J[\phi]$ in terms of $r\phi$. This procedure reveals  a hidden null structure. Since $d\beta: \lg \rightarrow \mathfrak{u}(V)$, for all $a\in \lg$ and $\varphi\in V$, we have
$\langle a\cdot \varphi,\varphi \rangle=-\langle\varphi, a\cdot \varphi \rangle$. Hence, for the linear map $\cT$ and for all $a\in \lg$, we conclude that
 \begin{align*}
    \langle  \langle\cT\phi, \phi \rangle+ \langle\phi, \cT\phi \rangle, a \rangle= \langle a\cdot \phi, \phi \rangle+ \langle\phi, a\cdot \phi \rangle=0.
 \end{align*}
 Hence, $\langle \cT\phi, \phi\rangle +\langle \phi, \cT\phi\rangle=0$. As a consequence, we derive that
 \begin{align*}
   r^2 J[\phi]=J[r\phi]-r \pa r(\langle \cT\phi, \phi\rangle +\langle \phi, \cT\phi\rangle)=J[r\phi].
 \end{align*}
We then use \eqref{eq:comp:v:EF} to bound $D_{\widetilde{\Lb}}(r\phi)$ and use \eqref{eq:bd4phi:need} to bound $\phi$. This leads to
\begin{equation*}
\begin{split}
\int_{S(\widetilde{r})}|\widetilde{r}^2\widetilde{\lap}_A h-\widetilde{r}^2\widetilde{\rho}\cdot h|d\widetilde{\om}&\les \int_{0}^{\widetilde{r}}
\int_{S(s)}(|\widetilde{\D} h||\widetilde{\ab}|+|J[\phi]_{\widetilde{\Lb}}|) s^2dsd\widetilde{\om}\\
&\les\left(\int_{0}^{\widetilde{r}}\int_{S(s)}(|\widetilde{\ab}|^2+r^{-2}|D_{\widetilde{\Lb}}(r\phi)|^2)(R-(t_0-s)-r_0)^{\ga} s^2 dsd\widetilde{\om}\right)^{\f12}\\
&\quad \cdot
\left(\int_{0}^{\widetilde{r}}\int_{S(s)}(R-(t_0-s)-r_0)^{-\ga}(|\phi|^2+|\widetilde{\D} h|^2) s^2 ds d\widetilde{\om}\right)^{\f12}.
\end{split}
\end{equation*}
Hence,
\begin{equation}
\label{eq:bd4h}
\begin{split}
\int_{S(\widetilde{r})}|\widetilde{r}^2\widetilde{\lap}_A h-\widetilde{r}^2\widetilde{\rho}\cdot h|d\widetilde{\om}&\les \sqrt{\mathcal{E}_{0, \ga}^R} \left(\int_{0}^{\widetilde{r}}\int_{S(s)}(|\phi|^2+|\widetilde{\D} h|^2) s^{2-\ga} ds d\widetilde{\om}\right)^{\f12}\\
&\les  \sqrt{\mathcal{E}_{0, \ga}^R} \widetilde{r}^{\frac{1-\ga}{2}} \left(\int_{\mathcal{N}^{-}(q)} s(|\phi|^2+{\sum_{i=1}^2}|D_{\widetilde{e}_i}h|^2) ds d\widetilde{\om}\right)^{\f12}\les  \mathcal{E}_{0, \ga}^R \widetilde{r}^{\frac{1-\ga}{2}}.
\end{split}
\end{equation}
Since $\ga<1$, the function $\widetilde{r}^{\frac{1-\ga}{2}-1}$ is integrable. The above estimate is sufficient to bound the integral on the cone $\mathcal{N}^{-}(q)$ in \eqref{eq:rep4phi} for $\phi$. It remains to control the integral on $\B_{R}$ in \eqref{eq:rep4phi}. We can integrate by parts:
\begin{align*}
&\left|\int_{\B_{R}}\langle\widetilde{r}^{-1}h\delta(\widetilde{v}), D_0 \phi\rangle-\langle D_0(\widetilde{r}^{-1}h\delta(\widetilde{v})), \phi \rangle dx\right|= \left|\int_{\B_{R}} \langle h\delta(\widetilde{v}), \widetilde{r} D_0 \phi \rangle -\langle D_0(h)\delta(\widetilde{v})+h \delta(\widetilde{v})', \widetilde{r}\phi\rangle d\widetilde{r}d\widetilde{\om}\right|\\
&=\left|\int_{\B_{R}} \langle h\delta(\widetilde{v}), \widetilde{r} D_0 \phi\rangle - \langle D_0(h)\delta(\widetilde{v}), \widetilde{r}\phi\rangle+\delta(\widetilde{v})(\langle D_{\widetilde{r}}(h), \widetilde{r}\phi\rangle +\langle h, D_{\widetilde{r}}(\widetilde{r}\phi)\rangle) d\widetilde{r}d\widetilde{\om}\right|\\
&\leq \int_{\widetilde{\om}}t_0|D\phi(0, x_0+t_0\widetilde{\om})|+|\phi|(0, x_0+t_0\widetilde{\om}) d\widetilde{\om}.
\end{align*}
We have used the uniform bound for $h$ and the fact that $D_{\widetilde{r}}h-D_0 h=-D_{\widetilde{\Lb}}h=0$. We remark that the above bound holds for all $V$-valued functions $\phi$ in particular for $D\phi$ in the place of $\phi$. By  Lemma \ref{lem:decay:lin:com} and the representation formula \eqref{eq:rep4phi} for $\phi$ with arbitrary $\mathbf{J}_q\in V$, $|\mathbf{J}_q|=1$, we conclude that
\begin{align*}
|\phi|(t_0, x_0)&\les \sqrt{\mathcal{E}_{1, \ga}^R(1+\mathcal{E}_{0, \ga}^R)}(R-t_0)^{-\frac{1+\ga}{2}}+\int_{\mathcal{N}^{-}(q)}|\widetilde{r}\widetilde{\lap}_A h{-} \widetilde{r}\widetilde{\rho} \cdot h||\phi| d\widetilde{r}d\widetilde{\om}\\
&\les \sqrt{\mathcal{E}_{1, \ga}^R(1+\mathcal{E}_{0, \ga}^R)}(R-t_0)^{-\frac{1+\ga}{2}}+\int_{0}^{t_0}\sup\limits_{S(\widetilde{r})}|\phi| \cdot \int_{S(\widetilde{r})}|\widetilde{r}\widetilde{\lap}_Ah-\widetilde{r}\widetilde{\rho}\cdot  h|
 d\widetilde{\om} d\widetilde{r} \\
 &\les \sqrt{\mathcal{E}_{1, \ga}^R(1+\mathcal{E}_{0, \ga}^R)}(R-t_0)^{-\frac{1+\ga}{2}}+\mathcal{E}_{0, \ga}^R \int_{0}^{t_0}s^{\frac{1-\ga}{2}-1}
 \sup\limits_{x}|\phi|(t_0-s, x)  ds,
\end{align*}
where $v_*=R-t-r$. 
This inequality leads to
\begin{align*}
  (R-t_0)^{\frac{1+\ga}{2}}|\phi|(t_0, x_0)
 &\les \sqrt{\mathcal{E}_{1, \ga}^R(1+\mathcal{E}_{0, \ga}^R)} +(\mathcal{E}_{0, \ga}^R) (R-t_0)^{\frac{1+\ga}{2}}\left(\int_{0}^{t_0}s^{\frac{1-\ga}{2}-1}
 \sup\limits_{x}|\phi|(t_0-s, x)  ds\right).
\end{align*}
We define $M_0(t)=(R-t)^{1+\ga}\sup\limits_{x} |\phi|^2(t, x)$. Therefore,
\begin{align*}
  M_0(t_0)&\les \sqrt{\mathcal{E}_{1, \ga}^R(1+\mathcal{E}_{0, \ga}^R)} +(\mathcal{E}_{0, \ga}^R)  \int_{0}^{t_0}s^{\frac{1-\ga}{2}-1}M_0(t_0-s) ds.
\end{align*}
By Lemma \ref{lem:Gronwall:V},  there exists a constant $C$ depending on $\mathcal{E}_{0, \ga}^R$, $\ga$, $\epsilon$ and $R$,  so that
\[
 M_0(t_0) \leq  C \sqrt{\mathcal{E}_{1, \ga}^R}.
\]
This proves the bound for the Higgs field $\phi$ in Theorem \ref{thm:EM}.

\medskip

We move to the estimate of $D\phi$. First of all, we have the following representation formula for $D\phi$:
\begin{equation}\label{eq: rep for Dphi}
\begin{split}
4\pi\langle \mathbf{J}_{q},  D\phi(t_0, x_0)\rangle =&\int_{\B_{R}} \langle \widetilde{r}^{-1}h\delta(\widetilde{v}), D_0 D\phi\rangle -\langle D_0(\widetilde{r}^{-1}h\delta(\widetilde{v})), D\phi\rangle  dx \\
&+\iint_{\mathcal{N}^{-}(q)}\langle \widetilde{\lap}_A h{-}\widetilde{\rho}\cdot h, D\phi\rangle \widetilde{r}^2d\widetilde{r}d\widetilde{\om}-\iint_{\mathcal{N}^{-}(q)}\langle h, [\Box_A, D]\phi\rangle \widetilde{r}d\widetilde{r}d\widetilde{\om},
\end{split}
\end{equation}
where $h$ is given at the beginning of this subsection. We can control the first two integrals as above: we simply repeat the previous estimate but replacing $\phi$ with $D\phi$, see  \eqref{eq:rep4phi}. More precisely, by  \eqref{eq:bd4DphiF:lin:com} of Lemma \ref{lem:decay:lin:com}, we have
\begin{align*}
|\int_{\B_{R}}\langle \widetilde{r}^{-1}h\delta(\widetilde{v}), D_0 D\phi\rangle -\langle D_0(\widetilde{r}^{-1}h\delta(\widetilde{v})), D\phi\rangle  dx|
&\leq \int_{\widetilde{\om}}t_0|DD\phi(0, x_0+t_0\widetilde{\om})|+|D\phi|(0, x_0+t_0\widetilde{\om}) d\widetilde{\om}\\
&\les \sqrt{\mathcal{E}_{2, \ga}^R}(1+\mathcal{E}_{0, \ga}^R)(R-t_0-r_0)^{-\frac{1+\ga}{2}-\ep}(R-t_0)^{-1+\ep}.
\end{align*}
We define the function $$M_1(t)=\sup\limits_{|x|\leq R-t}\left(v_*^{\frac{1+\ga}{2}+\ep}u_*^{1-\ep}(|D\phi|(t, x)+|F|(t, x))\right)$$ on  $\mathcal{J}^+(\B_R)$.
Recall that  $v_*(t, x)=R-t-r$, $u_*(t, x)=R-t+r$. According to $r=|x|\leq R-t$, we have $R-t\leq u_*\leq 2(R-t)$. Thus, $R-t$ is equivalent to $u_*$ up to a universal constant. For a fixed $q=(t_0, x_0)$, we define
\[
\bar v_*=v_*(q)=R-t_0-r_0,\quad \bar u_*=u_*(q)=R-t_0+r_0.
\]
On the backward light cone $\mathcal{N}^{-}(q)$, we always have
\[
v_*(t_0-\widetilde r, x_0+\widetilde{r}\widetilde{\om})\geq R-t_0-r_0=\bar v_*,\quad u_* \geq R-t\geq R-t_0\geq \frac{1}{2}\bar u_*.
\]
To estimate the second integral in \eqref{eq: rep for Dphi}, we can use \eqref{eq:bd4h} for $\widetilde{r}\widetilde{\lap}_A h- \widetilde{r}\widetilde{\rho} \cdot h$ to derive
\begin{align*}
|\int_{\mathcal{N}^{-}(q)}\langle \widetilde{\lap}_A h-\widetilde{\rho}\cdot h, D\phi\rangle \widetilde{r}^2d\widetilde{r}d\widetilde{\om}|
&\les \int_{0}^{t_0}M_1(t_0-\widetilde{r}) \int_{S(\widetilde{r})}v_*^{-\frac{1+\ga}{2}-\ep}u_*^{\ep-1}|\widetilde{r}\widetilde{\lap}_A h-\widetilde{r}\widetilde{\rho}\cdot h| d\widetilde{\om}d\widetilde{r}\\
&\les \mathcal{E}_{0, \ga}^R \bar v_*^{-\frac{1+\ga}{2}-\ep}\bar u_*^{\ep-1}\int_{0}^{t_0}\widetilde{r}^{-\frac{1+\ga}{2}}M_1(t_0-\widetilde{r}) d\widetilde{r}.
\end{align*}
Similar to  the $L^\infty$ estimate for $|\phi|$, this term will be absorbed by the Gronwall type inequality in Lemma \ref{lem:Gronwall:V}.

We move to the nonlinear term $[\Box_A, D]\phi$ coming from the commutation of the equation with $D$. Indeed, we have
\begin{align*}
[\Box_A, D_{\mu}]\phi=2 F_{\nu \mu}\cdot D^\nu\phi+D^\nu F_{\nu\mu}\cdot \phi=2 F_{\nu \mu}\cdot D^\nu\phi{+}J[\phi]_{\mu}\cdot \phi.
\end{align*}
Given the index $\mu$, the quadratic term $F_{\nu \mu}\cdot D^\nu\phi$ satisfies the null structure. We can bound it as follows:
\begin{align*}
|F_{\nu \mu}D^\nu\phi|\les (|F|+|D\phi|)(r^{-1}|D_{\widetilde{\Lb}}(r\phi)|+r^{-1}|\widetilde{\D}(r\phi)|+|\widetilde{\ab}|+|\widetilde{\rho}|+|\widetilde{\si}|+r^{-1}|\phi|).
\end{align*}
The norms $|F|$ and $|D\phi|$ are the same for both coordinates $(t, x)$ and $(\widetilde{t}, \widetilde{x})$. To bound the lower order term $r^{-1}\phi$, we use Proposition \ref{prop:bd4phi:phir}. The other terms can be bounded by the weighted energy estimate \eqref{eq:comp:v:EF}. By Proposition \ref{prop:EF:cone:gamma} and Lemma \ref{lem:bd:vga}, we have
\begin{equation}
\label{eq:NL:pf1}
\begin{split}
&\int_{\mathcal{N}^{-}(q)}(|F|+|D\phi|)(r^{-1}|D_{\widetilde{\Lb}}(r\phi)|+r^{-1}|\widetilde{\D}(r\phi)|+|\widetilde{\ab}|+|\widetilde{\rho}|+
|\widetilde{\si}|)\widetilde{r}d\widetilde{r}
d\widetilde{\om}\\
\les& \sqrt{\mathcal{E}_{0, \ga}^R}\left(\int_{0}^{t_0}\int_{S(\widetilde{r})} (|D\phi|^2+|F|^2) v_*^{-\ga}d\widetilde{r}d\widetilde{\om}\right)^\f12\\
\les& \sqrt{\mathcal{E}_{0, \ga}^R} \left(\int_{0}^{t_0}M_1^2(t_0-\widetilde{r}) \int_{S(\widetilde{r})} \bar v_*^{-\ga-1-2\ep}(R-t_0+\widetilde{r})^{-2+2\ep}v_*^{-\ga}d\widetilde{\om} d\widetilde{r} \right)^\f12\\
\les& \sqrt{\mathcal{E}_{0, \ga}^R}\bar v_*^{-\frac{1+\ga}{2}-\ep}\bar u_*^{\ep-1} (\int_{0}^{t_0}M_1^2(t_0-\widetilde{r})  \widetilde{r}^{-\ga}d\widetilde{r} )^\f12.
\end{split}
\end{equation}
For $r^{-1}|\phi|$, by Proposition \ref{prop:bd4phi:phir}, we have
\begin{equation}
\label{eq:NL:pf2}
\begin{split}
\int_{\mathcal{N}^{-}(q)}(|D\phi|+|F|)r^{-1}|\phi|\widetilde{r}d\widetilde{r}d \widetilde{\om}&\les \bar v_*^{-\frac{1+\ga}{2}-\ep}\bar u_*^{\ep-1}\int_{0}^{t_0}M_1(t_0-\widetilde{r}) \cdot \int_{S_{(t_0-\widetilde{r}, x_0)}(\widetilde{r})}r^{-1}|\phi| \widetilde{r}d \widetilde{\om} d\widetilde{r}\\
&\les \sqrt{\mathcal{E}_{0, \ga}^R} \bar v_*^{-\frac{1+\ga}{2}-\ep}\bar u_*^{\ep-1}\int_{0}^{t_0}M_1(t_0-\widetilde{r})  \widetilde{r}^{-\frac{1+\ga+\ep}{2}} d\widetilde{r} .
\end{split}
\end{equation}
For $J[\phi]\cdot \phi$, since $
|J[\phi]\cdot \phi|\les |J[\phi]||\phi|\les |D\phi||\phi|^2$, in view of Proposition \ref{prop:bd4phi:need}, we have
\begin{equation}
\label{eq:NL:pf3}
\begin{split}
\int_{\mathcal{N}^{-}(q)}(|F|+|D\phi|)|\phi|^2\widetilde{r}d\widetilde{r}d \widetilde{\om} &\les \bar v_*^{-\frac{1+\ga}{2}-\ep}\bar u_*^{\ep-1}\int_{0}^{t_0}M_1(t_0-\widetilde{r})  \int_{S_{(t_0-\widetilde{r}, x_0)}(\widetilde{r})}|\phi|^2 \widetilde{r}d \widetilde{\om} d\widetilde{r}\\
& \les \mathcal{E}_{0, \ga}^R\bar v_*^{-\frac{1+\ga}{2}-\ep}\bar u_*^{\ep-1}\int_{0}^{t_0}M_1(t_0-\widetilde{r})  \widetilde{r}^{-\ga-\ep}  d\widetilde{r}.
\end{split}
\end{equation}
Combining the above estimates and the fact that $|h|\leq 1$, we conclude that
\begin{align*}
&\bar v_*^{\frac{1+\ga}{2}+\ep}\bar u_*^{1-\ep}|\int_{\mathcal{N}^{-}(q)}{ \langle}h, [\Box_A, D]\phi{ \rangle}\widetilde{r}d\widetilde{r}d\widetilde{\om}|\\
&\les \sqrt{\mathcal{E}_{0, \ga}^R} \left(\int_{0}^{t_0}M_1^2(t_0-\widetilde{r}) \widetilde{r}^{-\ga} d\widetilde{r}\right)^\f12
+\int_{0}^{t_0}M_1(t_0-\widetilde{r}) ( \sqrt{\mathcal{E}_{0, \ga}^R} +\mathcal{E}_{0, \ga}^R)  \widetilde{r}^{-\frac{1+\ga+\ep}{2}}   d\widetilde{r}\\
&\les \sqrt{\mathcal{E}_{0, \ga}^R} \left(\int_{0}^{t_0}M_1^2(t_0-\widetilde{r}) \widetilde{r}^{-\ga} d\widetilde{r}\right)^\f12
+( 1 +\mathcal{E}_{0, \ga}^R) \left(\int_{0}^{t_0}M_1^2(t_0-\widetilde{r})  \widetilde{r}^{-\ga -2\ep }  d\widetilde{r}\right)^{\f12 } \\
&\les ( 1 +\mathcal{E}_{0, \ga}^R) \left(\int_{0}^{t_0}M_1^2(t_0-\widetilde{r})  \widetilde{r}^{-\ga -2\ep }  d\widetilde{r}\right)^{\f12 }.
\end{align*}
 Now combining all the above estimates for the right hand side of the representation formula for $D\phi$, we have shown that
\begin{equation}
\label{eq:bd4Dphi:0}
\begin{split}
|\bar v_*^{\frac{1+\ga}{2}+\ep}\bar u_*^{1-\ep} D\phi|^2(t_0, x_0)&\les  \mathcal{E}_{2, \ga}^R(1+\mathcal{E}_{0, \ga}^R)^2 + ( 1 +\mathcal{E}_{0, \ga}^R)^2  \int_{0}^{t_0}M_1^2(t_0-\widetilde{r})  \widetilde{r}^{-\ga -2\ep }  d\widetilde{r} .
\end{split}
\end{equation}
The implicit constant depends only on $\ep$, $\ga$ and $R$. In particular, \eqref{eq:bd4Dphi:0} is uniform in $x_0$.

In view of the presence of $F$ in $M_1(t)$, we need a similar bound for the Yang-Mills field $F$. Since $F_{\mu\nu}$ verifies a covariant wave equation \eqref{eq:EQDphiF}, we will apply the similar representation formula for $F_{\mu\nu}$. Since $F$ takes values in $\lg$, we need to construct a $\lg$-valued function $h$ instead of  $V$-valued function. Let $\mathbf{J}_{q} \in \lg$ so that $|\mathbf{J}_{q}|=1$. Given $\mathbf{J}_{q} \in \lg$, let $h$ be $\lg$-valued function so that, on $\mathcal{N}^{-}(q)$, we have
\[D_{\widetilde{\Lb}} h=0,\quad h(q)=\mathbf{J}_{q}.\]
Compared to the previous $V$-valued function,  the action of $a\in \lg$ on $b \in \lg$ becomes $[a, b]$. The representation formula for $F_{\mu\nu}$ is
\begin{equation}\label{eq: repre for F 4}
\begin{split}
4\pi\langle \mathbf{J}_{q},  (F_{\mu\nu})|_{q}\rangle=&\int_{\B_{R}}\langle\widetilde{r}^{-1}h\delta(\widetilde{v}), D_0 F_{\mu\nu}\rangle-\langle D_0(\widetilde{r}^{-1}h\delta(\widetilde{v})), F_{\mu\nu}\rangle dx \\ &+\iint_{\mathcal{N}^{-}(q)}\langle\widetilde{\lap}_A h{-}[\widetilde{\rho}, h], F_{\mu\nu}\rangle\widetilde{r} d\widetilde{r}d\widetilde{\om}-\iint_{\mathcal{N}^{-}(q)}\langle h,  \Box_A F_{\mu\nu}\rangle\widetilde{r}d\widetilde{r}d\widetilde{\om}.
\end{split}
\end{equation}
The first integral in \eqref{eq: repre for F 4} depends only on the initial data. Similar to the case for $\phi$,  we use  \eqref{eq:bd4DphiF:lin:com} and integration by parts to bound it. Indeed,
\begin{align*}
&|\int_{\B_{R}}\langle \widetilde{r}^{-1}h\delta(\widetilde{v}), D_0 F_{\mu\nu}\rangle-\langle D_0(\widetilde{r}^{-1}h\delta(\widetilde{v})), F_{\mu\nu}\rangle dx|\\
=& |\int_{\B_{R}} \langle h\delta(\widetilde{v}), \widetilde{r} D_0 F_{\mu\nu}\rangle-\langle D_0(h)\delta(\widetilde{v})+h \delta(\widetilde{v})', \widetilde{r}F_{\mu\nu}\rangle d\widetilde{r}d\widetilde{\om}|\\
=&|\int_{\B_{R}}\langle h\delta(\widetilde{v}), \widetilde{r} D_0F_{\mu\nu}\rangle-\langle D_0(h)\delta(\widetilde{v}), \widetilde{r}F_{\mu\nu}\rangle +\delta(\widetilde{v})(\langle D_{\widetilde{r}}(h), \widetilde{r}F_{\mu\nu}\rangle +\langle h, D_{\widetilde{r}}(\widetilde{r}F_{\mu\nu})\rangle) d\widetilde{r}d\widetilde{\om}|.
\end{align*}
Hence,
\begin{align*}
|\int_{\B_{R}}\langle \widetilde{r}^{-1}h\delta(\widetilde{v}), D_0 F_{\mu\nu}\rangle-\langle D_0(\widetilde{r}^{-1}h\delta(\widetilde{v})), F_{\mu\nu}\rangle dx|
\leq& \int_{\widetilde{\om}}t_0|DF_{\mu\nu}(0, x_0+t_0\widetilde{\om})|+|F_{\mu\nu}|(0, x_0+t_0\widetilde{\om}) d\widetilde{\om}\\
\les& \sqrt{\mathcal{E}_{2, \ga}^R }(1+\mathcal{E}_{0, \ga}^R)(R-t_0-r_0)^{-\frac{1+\ga}{2}-\ep}(R-t_0)^{-1+\ep},
\end{align*}
where we used that $\delta (\widetilde{v}) (D_{\widetilde{r}}h-D_{0}h)=-\delta(\widetilde{v})D_{\widetilde{\Lb}}(h)=0$ and $|h|\leq 1$.

For the second integral in \eqref{eq: repre for F 4}, it has the same bound as $D\phi$:
\begin{align*}
&|\int_{\mathcal{N}^{-}(q)}\langle \widetilde{\lap}_A h{-}[\widetilde{\rho}, h], F_{\mu\nu}\rangle\widetilde{r} d\widetilde{r}d\widetilde{\om}|\les
\mathcal{E}_{0, \ga}^R \bar v_*^{-\frac{1+\ga}{2}-\ep}\bar u_*^{\ep-1}\int_{0}^{t_0}\widetilde{r}^{-\frac{1+\ga}{2}}M_1(t_0-\widetilde{r}) d\widetilde{r},
\end{align*}
where we used $M_1(t)$ to bound $F$.

Finally, we consider the third integral in \eqref{eq: repre for F 4} involving the nonlinear term $\Box_A F_{\mu\nu}$. It can be computed by \eqref{eq:EQDphiF} and it consists of $[F_{\mu\ga}, F_{\nu}^{\ \ga}]$ and $D_{\mu }J[\phi]_{\nu}-D_{\nu}J[\phi]_{\mu}$. The term $[F_{\mu\ga}, F_{\nu}^{\ \ga}]$ satisfies the null structure and we can bound it by \begin{align*}
  |[F_{\mu\ga}, F_{\nu}^{\ \ga}]|\les |F|(|\widetilde{\rho}|+|\widetilde{\si}|+|\widetilde{\ab}|).
\end{align*}
Since $J[\phi]_{\mu}=\langle \cT\phi, D_{\mu}\phi\rangle +\langle D_{\mu}\phi, \cT \phi\rangle$, we compute that
\begin{align*}
  D_{\mu }J[\phi]_{\nu}=\langle D_{\mu}\cT\phi, D_{\nu}\phi\rangle+\langle\cT\phi, D_{\mu}D_{\nu}\phi\rangle+\langle D_{\nu}\phi , D_{\mu}\cT\phi \rangle+\langle D_{\mu}D_{\nu}\phi, \cT\phi\rangle.
\end{align*}
It is straightforward to check that  $\cT$ commutes with the covariant derivative $D$, i.e., $\langle D_{\mu}\cT\phi, D_{\nu}\phi\rangle=\langle\cT D_{\mu}\phi, D_{\nu}\phi\rangle$. By writing $\pa_\mu$ in the null frame centered at the point $q$, we obtain that
\begin{align*}
|\langle \cT D_{\mu}\phi, D_{\nu}\phi\rangle|\les |D\phi|(|D_{\widetilde{\Lb}}\phi|+|\widetilde{\D}\phi|)\les |D\phi|r^{-1}(|D_{\widetilde{\Lb}}(r\phi)|+|\widetilde{\D}(r\phi)|+|\phi|).
\end{align*}
Therefore, we have
\begin{align*}
  |D_{\mu }J[\phi]_{\nu}-D_{\nu}J[\phi]_{\mu}|&\les |D\phi|r^{-1}(|D_{\widetilde{\Lb}}(r\phi)|+|\widetilde{\D}(r\phi)|+|\phi|)+|\phi|| [D_{\mu}, D_{\nu}]\phi|\\
  &\les |D\phi|r^{-1}(|D_{\widetilde{\Lb}}(r\phi)|+|\widetilde{\D}(r\phi)|+|\phi|)+|\phi|^2 | F|.
\end{align*}
In particular, we derive that
\begin{align*}
  |\Box_A F_{\mu\nu}|\les (|F|+|D\phi|)(r^{-1}|D_{\widetilde{\Lb}}(r\phi)|+r^{-1}|\widetilde{\D}(r\phi)|+|\widetilde{\ab}|+|\widetilde{\rho}|+|\widetilde{\si}|+r^{-1}|\phi|+|\phi|^2).
\end{align*}
In view of estimates \eqref{eq:NL:pf1} to \eqref{eq:NL:pf3} and the fact that $|h|\leq 1$,  we then have
\begin{align*}
  \bar v_*^{\frac{1+\ga}{2}+\ep}\bar u_*^{1-\ep}\left|\iint_{\mathcal{N}^{-}(q)}\langle h,  \Box_A F_{\mu\nu}\rangle\widetilde{r}d\widetilde{r}d\widetilde{\om}\right|&\les  { \bar v_*^{\frac{1+\ga}{2}+\ep}\bar u_*^{1-\ep}}\iint_{\mathcal{N}^{-}(q)}| \Box_A F_{\mu\nu}|\widetilde{r}d\widetilde{r}d\widetilde{\om}\\
  &\les (1+ \mathcal{E}_{0, \ga}^R)  \big(\int_{0}^{t_0}M_1^2(t_0-\widetilde{r}) \widetilde{r}^{-\ga-2\epsilon } d\widetilde{r}\big)^\f12 .
\end{align*}
Combining the above estimates for the Yang-Mills field, we have shown that
\begin{equation*}
\begin{split}
|\bar v_*^{\frac{1+\ga}{2}+\ep}\bar u_*^{1-\ep} F|^2(t_0, x_0)&\les  \mathcal{E}_{2, \ga}^R(1+\mathcal{E}_{0, \ga}^R)^{ 2}+(1+\mathcal{E}_{0, \ga}^R)^2 \int_{0}^{t_0}M_1^2(t_0-\widetilde{r}) \cdot \widetilde{r}^{-\ga-2\ep }d\widetilde{r} .
\end{split}
\end{equation*}
In view of the estimate \eqref{eq:bd4Dphi:0} for $D\phi$ and the definition for $M_1(t)$, by taking the supreme over all $x_0$, we derive that
\begin{equation*}
M_1^2 (t_0)\les  \mathcal{E}_{2, \ga}^R(1+\mathcal{E}_{0, \ga}^R)^{ 2}+ (1+\mathcal{E}_{0, \ga}^R)^{ 2}  \int_{0}^{t_0}M_1^2(t_0-\widetilde{r}) \cdot \widetilde{r}^{-\ga-2\ep }d\widetilde{r},
\end{equation*}
which, in view of Lemma \ref{lem:Gronwall:V}, leads to
   \begin{align*}
    M_1^2(t_0)\leq C \mathcal{E}_{2, \ga}^R
 \end{align*}
 for some constant $C$ depending only on $\ep$, $\ga$, $R$ and $\mathcal{E}_{0, \ga}^R$. Theorem \ref{thm:EM} then follows by the definition of the function $M_1(t_0)$.

\section{Proof for Theorem \ref{thm:YMH:hyperboloid}}
\label{sec:conformal}
The proof for Theorem \ref{thm:YMH:hyperboloid} relies on the main Theorem \ref{thm:EM}. It suffices to check that under the assumption that the solution is bounded in the weighted energy space $\mathcal{E}_{2, \ga_1}^{\mathcal{H}}$, the associated conformal transformation verifies the assumption in Theorem \ref{thm:EM}. Then the decay estimates of Theorem \ref{thm:YMH:hyperboloid} follow from the reversion of the transformation.

\bigskip

Define the map
\begin{align*}
\mathbf{\Phi}:(t, x)\longmapsto \left(\widetilde{t}=-\frac{t^*}{(t^*)^2-|x|^2}+R,\quad \widetilde{x}= \frac{x}{(t^*)^2-|x|^2}\right)
\end{align*}
from the future $\mathbf{D}$ of the hyperboloid $\mathcal{H}$ to the Minkowski space. Here recall that $t^*=t+R+1$.
It is easy to check that the image of $\mathbf{D}$ under the above map $\mathbf{\Phi}$ is a truncated backward light cone in Minkowski space
\begin{align*}
\mathbf{\Phi}(\mathbf{D})=\left\{(\widetilde{t}, \widetilde{x})|\quad \widetilde{t}+|\widetilde{x}|<R ,\quad \widetilde{t}\geq 0\right\}.
\end{align*}
Here we use $(\widetilde t, \widetilde{x})$ as coordinates on the image of the map $\mathbf{\Phi}$.
Define the conformal factor
\begin{equation*}
\Lambda(t, x)=(t^*)^2-|x|^2.
\end{equation*}
By a theorem of Choquet-Bruhat-Christodoulou, a solution $(A, \phi)$ of the same YMH equations on $\mathbf{D}$ is equivalent to that $((\mathbf{\Phi}^{-1})^* A, (\Lambda \phi)\circ \mathbf{\Phi}^{-1})$ solves the YMH equations on $\mathbf{\Phi}(\mathbf{D})$, see the paper \cite{Christ81:YM} or the book \cite{Christodoulou:Book:GR1} for more details. In another word, to obtain decay estimates for solutions of YMH equations on $\mathbf{D}$, it suffices to control the solution on $\mathbf{\Phi}(\mathbf{D})$, which is a backward light cone in Minkowski space as depicted below.
 
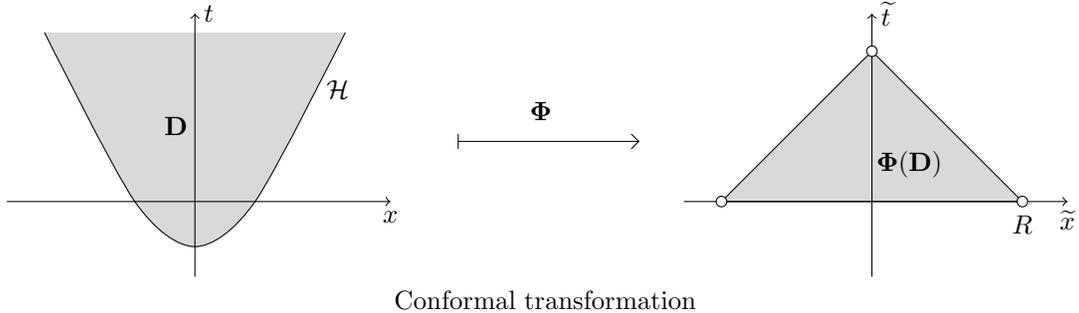
\begin{center}

 \begin{tikzpicture}

  \filldraw [fill=gray!30] plot[smooth] coordinates {(-8,2.25) (-6.8,0) (-6, -0.6) (-5.2, 0) (-4, 2.25)};
  \node at (-4.1, 1.5) {$\mathcal{H}$};
  \node at (-6.25, 1) {$\mathbf{D}$};
    \draw[->] (-8.5,0) -- (-3.4,0) coordinate[label = {below:$x$}] (xmax);
  \draw[->] (-6,-1) -- (-6,2.5) coordinate[label = {right:$t$}] (ymax);
\filldraw [fill=gray!30]  (3, 2)-- (1,0)-- (5, 0)--cycle;
\draw[->] (0.5,0) -- (5.6,0) coordinate[label = {below:$\widetilde{x}$}] (xtmax);
\filldraw [fill=white] ( 5, 0) circle (2pt);
\filldraw [fill=white] ( 1, 0) circle (2pt);
\draw[->] (3,-1) -- (3,2.5) coordinate[label = {right:$\widetilde{t}$}] (ytmax);
\filldraw [fill=white] ( 3, 2) circle (2pt);
  \node at (5, -0.3) {$R$};
  \node at (3.5, 0.5) {$\mathbf{\Phi}(\mathbf{D})$};
  \draw (-2.5, 0.8)--(-0.1, 0.8);
  \draw (-2.5, 0.7)--(-2.5, 0.9);
   \draw (-0.2, 0.9 )--(-0.1, 0.8)--(-0.2, 0.7);
      \node at (-1.4, 1.2) {$\mathbf{\Phi}$};
 \end{tikzpicture}

Conformal transformation
  \end{center}
The initial hypersurface for this backward light cone is
$$\B_{R}=\{(0, \widetilde{x})||\widetilde{x}|\leq R\}=\mathbf{\Phi}(\mathcal{H}),$$
which is the image of the hyperboloid $\mathcal{H}$ under the map $\mathbf{\Phi}$.
For simplicity, we identify $(\widetilde{A}, \Lambda \phi)$ with $((\mathbf{\Phi}^{-1})^* A, (\Lambda \phi)\circ \mathbf{\Phi}^{-1})$ on $\mathbf{\Phi}(\mathbf{D})$ and denote $\widetilde{F}=d\widetilde{A}+\f12[\widetilde{A}\wedge \widetilde{A}]$ as the curvature 2-form associated to $\widetilde{A}$. To apply Theorem \ref{thm:EM}, it remains to understand $(\widetilde{A}, \Lambda \phi)$ on the initial hypersurface $\B_{R}$. Since the information of $(\widetilde{A}, \Lambda \phi)$ is determined by $(F, \phi)$ on the hyperboloid $\mathcal{H}$ via the map $\mathbf{\Phi}$, it then suffices to control the weighted norm $\mathcal{E}_{2, \ga_1}^R$ on $\B_{R}$ in terms of the norm $\mathcal{E}_{2, \ga_1}^{\mathcal{H}}$ on the hyperboloid $\mathcal{H}$.

Define the null coordinates $$\widetilde{u}=\frac{\widetilde{t}-\widetilde{r}}{2},\quad  \widetilde{v}=\frac{\widetilde{t}+\widetilde{r}}{2}, \quad \widetilde{\om}=\frac{\widetilde{x}}{|\widetilde{x}|}$$ on $\mathbf{\Phi}(\mathbf{D})$. Since $\frac{x}{|x|}=\frac{\widetilde{x}}{|\widetilde{x}|}$, we may choose the associated null frame
$$\{\widetilde{L}=\pa_{\widetilde{t}}+\pa_{\widetilde{r}}, \quad \widetilde{\Lb}=\pa_{\widetilde{t}}-\pa_{\widetilde{r}}, \quad \La e_1,\quad  \La e_2\},$$
 where we recall that $\{L, \Lb, e_1, e_2 \}$ is the null frame on $\mathbf{D}$. Recall the weighted energy norm $\mathcal{E}_{2, \ga}^R$ associated to $(\Lambda \phi, \widetilde{F})$ on $\mathbf{\Phi}(\mathcal{H})$
\begin{align*}
  \mathcal{E}_{0, \ga}^R& =\int_{B_{R}}(R-|\widetilde{x}|)^{\ga}(|\widetilde{\a}|^2+|\widetilde{D}_{\widetilde{L}}(\La\phi)|^2)+|\widetilde{\rho}|^2+|\widetilde{\si}|^2+|\widetilde{\a}|^2
  +|\widetilde{D}_{\widetilde{\Lb}}(\La\phi)|^2+|\La \phi|^2d\widetilde{x},\\
\mathcal{E}_{k, \ga}^R&=\sum\limits_{l_1+l_2\leq k}\sum\limits_{i, j}\int_{B_{R}}(|\widetilde{D}^{l_1}\widetilde{D}_{\widetilde{\Om}_{ij}}^{l_2}\widetilde{F}|^2+|\widetilde{D}^{l_1+1} \widetilde{D}_{\widetilde{\Om}_{ij}}^{l_2} (\La\phi)|^2)(R-|\widetilde{x}|)^{\ga+2l_1}d\widetilde{x}+\mathcal{E}_{0, \ga}^R, \quad 1\leq k\leq 2,
\end{align*}
where
\[
\widetilde{\Om}_{ij}=\widetilde{x}_i\widetilde{\pa}_{j}-\widetilde{x}_j\widetilde{\pa}_{i}=x_i\pa_j-x_j\pa_i=\Om_{ij}.
\]
Here the following calculations may be of help
\begin{align*}
&d\widetilde{t}=\La^{-2}((t^*)^2+r^2)dt-2 \La^{-2}t^* x dx,\quad d\widetilde{x}=\La^{-1}dx-2\La^{-2}t^* x dt+2\La^{-2}rx dr,\\
&\pa_t=\La^{-2}((t^*)^2+r^2)\pa_{\widetilde{t}}-2\La^{-2}t^* x\pa_{\widetilde{x}},\quad \pa_{x}=-2\La^{-2}t^* x \pa_{\widetilde{t}}+\La^{-1}\pa_{\widetilde{x}}+2\La^{-2}rx \pa_{\widetilde{r}},\\&\pa_t=((R-\tilde{t})^2+|\tilde{x}|^2)\pa_{\tilde{t}}-2(R-\tilde{t}) \tilde{x}\pa_{\tilde{x}},\quad S+(R+1)\pa_{t}=t_*\pa_{t}+x\pa_{x}=(R-\tilde{t})\pa_{\tilde{t}}- \tilde{x}\pa_{\tilde{x}}.
\end{align*}More precisely we can write  $\Phi_*\Om_{ij}=\Om_{ij} $ and\begin{align*}
&\Phi_*\pa_t=((R-{t})^2+|{x}|^2)\pa_{{t}}-2(R-{t}) {x}\pa_{{x}}:=\tilde{K},\quad \Phi_*(S+(R+1)\pa_{t})=(R-{t})\pa_{{t}}- {x}\pa_{{x}}:=\tilde{S}.
\end{align*}
In particular the null frame obeys the following transformation rules
\begin{align*}
&L=\pa_t+\pa_r=\La^{-2}((t^*)^2+r^2-2t^* r)(\pa_{\widetilde{t}}+\pa_{\widetilde{r}})=(t^*+r)^{-2}\widetilde{L},\\
&\Lb=\pa_t-\pa_r=\La^{-2}((t^*)^2+r^2+2t^* r)(\pa_{\widetilde{t}}-\pa_{\widetilde{r}})=(t^*-r)^{-2}\widetilde{\Lb},\\
&\pa_i-\om_i \pa_r=\La^{-1}(\pa_{\widetilde{x}_i}-\widetilde{\om}_i\pa_{\widetilde{r}})=((t^*)^2-r^2)^{-1}(\pa_{\widetilde{x}_i}-\widetilde{\om}_i\pa_{\widetilde{r}}).
\end{align*}
As $\widetilde{A}=(\mathbf{\Phi}^{-1})^*A$, we have $\widetilde{F}=(\mathbf{\Phi}^{-1})^* F$. Then the curvature components can be written as
\begin{align*}
\widetilde{\a}_i &=\widetilde{F}(\widetilde{L}, \widetilde{e}_i)=\La (t^*+r)^2 \a_i,\quad \widetilde{\ab}_i=\widetilde{F}(\widetilde{\Lb}, \widetilde{e}_i)=\La (t^*-r)^2 \ab_i,\quad i=1, 2,\\
\widetilde{\si}&=\widetilde{F}(\widetilde{e}_1,\widetilde{e}_2)=\La^2 \si,\quad \widetilde{\rho}=\widetilde{F}(\widetilde{L},\widetilde{\Lb})=(t^*+r)^2(t^*-r)^2 \rho=\La^2 \rho.
\end{align*}
The transformation formula for $\widetilde{G}=(\mathbf{\Phi}^{-1})^* G$ is similar (here $G$ is a $\lg$-valued 2-form). Let $$
\tilde{\Gamma}=\{\tilde{K},\Om_{ij},\tilde{S}-(R+1)\tilde{K}\} =\{\Phi_*X:X\in \Gamma\}.$$
 We can estimate $D\varphi$ in terms of $D_X\varphi$, $X\in \tilde{\Gamma}$.
\begin{Lem}
\label{lem1}For a $V$-valued or $\lg$-valued function $ \varphi$,  $0\leq k\leq 2$, $t=0,$ $R/2\leq |x|\leq R$ we have
\begin{align*}
&(R-|x|)|D^{k+1}\varphi|\leq C\sum_{Z\in \tilde{\Gamma}}|D^kD_Z\varphi|+C\sum_{j=1}^k|D^{j}\varphi|.\end{align*}\end{Lem}\begin{proof}
Notice that (here $i\in\{1,2,3\}$ and we omitted the summation {sign} for $j\in\{1,2,3\}$)\begin{align*}
&\tilde{\Lambda}\pa_t=2(R-{t})\tilde{S}-\tilde{K},\quad -|{x}|^2\pa_i=x_i\tilde{S}+x_j\Om_{ij}-x_i(R-t)\partial_t,\quad \tilde{\Lambda}:=(R-{t})^2-|{x}|^2.
\end{align*}Then there exist functions $a_{\mu Z},\ b_{\mu Z}\in C^{\infty}(\{R/2\leq |x|\leq R,\ |t|\leq R/4\})$ for $\mu\in\{0,1,2,3\}$, $Z\in \tilde{\Gamma}\cup\{\partial_t\}$ such that (for $i\in\{1,2,3\}$)\begin{align*}
&\tilde{\Lambda}\pa_t=\sum_{Z\in \tilde{\Gamma}}a_{0Z}Z,\quad \pa_i=\sum_{Z\in \tilde{\Gamma}\cup\{\partial_t\}}b_{iZ}Z,\quad \tilde{\Lambda}\pa_i=\sum_{Z\in \tilde{\Gamma}}a_{iZ}Z.
\end{align*}As $\tilde{\Lambda}=R^2-|{x}|^2 $, $ R-|x|\leq C\tilde{\Lambda}$ on $t=0,$ $R/2\leq |x|\leq R$; the functions $\tilde{\Lambda},$ $a_{\mu Z}$ are $C^{\infty}$; we have\begin{align*}
(R-|x|)|D^{k+1}\varphi|\les& \tilde{\Lambda}|D^{k+1}\varphi|\les |D^k(\tilde{\Lambda}D\varphi)|+\sum_{j=1}^k|D^{j}\varphi|\les\sum_{Z\in \tilde{\Gamma}}\sum_{\mu=0}^3|D^k(a_{\mu Z}D_Z\varphi)|+\sum_{j=1}^k|D^{j}\varphi|\\ \les&\sum_{Z\in \tilde{\Gamma}}\sum_{j=0}^k|D^jD_Z\varphi|+\sum_{j=1}^k|D^{j}\varphi| \les\sum_{Z\in \tilde{\Gamma}}|D^kD_Z\varphi|+\sum_{j=1}^k|D^{j}\varphi|.
\end{align*}
Here we still assume $t=0,$ $R/2\leq |x|\leq R$, and we used that the vector field $Z\in \tilde{\Gamma} $ is $C^{\infty}$ to deduce that $|D^jD_Z\varphi|\les\sum_{l=1}^{j+1}|D^{l}\varphi|$ for $0\leq j<k$. This completes the proof.
\end{proof}Note that
\[
\B_{R}=\mathbf{\Phi}(\mathcal{H}_+)\cup \mathbf{\Phi}(\mathcal{H}_{-})=\{(0, \tilde{x})| R_*\leq |\tilde{x}|<R\}\cup\{(0, \tilde{x})| |\tilde{x}|< R_*\},\quad R_*=\sqrt{R^2-R/(R+1)}.
\]
We  define an auxiliary  energy norms associated to $(\phi,F)$ as (here $k=1,2$)
\begin{align*}
&\mathcal{E}_{k, \ga}^{R*}=\int_{R_*\leq |x|<R}(R-|x|)^{\ga}\sum_{l\leq k}\sum_{Z\in \tilde{\Gamma}}(|\mathcal{L}_Z^lF|^2+|DD_Z^l\phi|^2)(0,x)dx\\&+\int_{ |x|< R_*}\sum_{l\leq k}(|D^lF|^2+|D^{l+1}\phi|^2)(0,x)dx+\mathcal{E}_{0, \ga}^{R},\quad \mathcal{E}_{0, \ga}^{R*}=\mathcal{E}_{0, \ga}^{R}.
\end{align*}
\begin{Lem}
\label{lem2}The solution $({F}, \phi)$ verifies the following bound on $\B_{R}$:
$$\mathcal{E}_{k, \ga}^{R}\leq C\mathcal{E}_{k, \ga}^{R*} , \quad k=1,2.$$
\end{Lem}
\begin{proof}
As $R>1$ we have $R/2<R_*<R$. Then for $t=0,$ $R_*\leq |x|< R,$ by Lemma \ref{lem1} (for $k=0,$ $ \varphi=D_{\Omega_{ij}}^lF$; $k=1,$ $ \varphi=F$; $k=0,$ $ \varphi=D_{Z}^lF$, $l=0,1$) we have
\begin{align*}
&(R-|x|)|DD_{\Omega_{ij}}^lF|\leq C\sum_{Z\in \tilde{\Gamma}}|D_ZD_{\Omega_{ij}}^lF|,\quad l=0,1,\\& (R-|x|)^2|D^2F|\leq C(R-|x|)\sum_{Z\in \tilde{\Gamma}}(|DD_ZF|+|DF|)\leq C\sum_{X,Z\in \tilde{\Gamma}}(|D_XD_ZF|+|D_ZF|),\end{align*} which along with $ \Omega_{ij}\in \tilde{\Gamma}$ gives (for $k=1,2$)\begin{align}\label{eq:F1}
&\sum_{l_1+l_2\leq k}(R-|x|)^{l_1}|D^{l_1}D_{\Omega_{ij}}^{l_2}F|\leq C\sum_{l\leq k}\sum_{Z\in \tilde{\Gamma}}|D_Z^lF|.\end{align}
Now we show that $D_Z$ can be bounded by $\mathcal{L}_Z$. Recall that
\begin{align*} &(\mathcal{L}_ZG)_{\mu\nu}=D_{Z}G_{\mu\nu}+\partial_{\mu}Z^{\delta}G_{\delta\nu}+
\partial_{\nu}Z^{\delta}G_{\mu\delta}. \end{align*}
Then for $ X,Z\in \tilde{\Gamma}$ and a $\lg$-valued 2-form $G$ we have
\begin{align*}
&|\mathcal{L}_ZG-D_{Z}G|\leq C|DZ||G|\leq C|G|,
\\&|D_X(\mathcal{L}_ZG-D_{Z}G)|\leq C|X||D^2Z||G|+|DZ||D_XG|\leq C(|G|+|D_XG|). \end{align*}
Thus for $ X,Z\in \tilde{\Gamma}$ we have
\begin{align*} &|D_ZF|\leq |\mathcal{L}_ZF |+C|F|,\\&|D_XD_ZF|\leq |D_X\mathcal{L}_ZF |+C(|F|+|D_XF|)\leq |\mathcal{L}_X\mathcal{L}_ZF |+C|\mathcal{L}_ZF |+C(|F|+|\mathcal{L}_XF|), \end{align*} and
\begin{align}\label{eq:F2} &\sum_{l\leq k}\sum_{Z\in \tilde{\Gamma}}|D_Z^lF|\leq C\sum_{l\leq k}\sum_{Z\in \tilde{\Gamma}}|\mathcal{L}_Z^lF|,\quad k=1,2. \end{align}
Similarly by Lemma \ref{lem1} (for $k=1,$ $ \varphi=D_{\Omega_{ij}}^l\phi$; $k=2,$ $ \varphi=\phi$; $k=1,$ $ \varphi=D_{Z}^l\phi$, $l=0,1$) we have\begin{align*}
&(R-|x|)|D^2D_{\Omega_{ij}}^l\phi|\leq C\sum_{Z\in \tilde{\Gamma}}(|DD_ZD_{\Omega_{ij}}^l\phi|+|DD_{\Omega_{ij}}^l\phi|),\quad l=0,1,\\& (R-|x|)^2|D^3\phi|\leq C(R-|x|)\sum_{Z\in \tilde{\Gamma}}(|D^2D_Z\phi|+|D^2\phi|+|D\phi|),\\&(R-|x|)|D^2D_Z\phi|\leq C\sum_{X\in \tilde{\Gamma}}(|DD_XD_Z\phi|+|DD_Z\phi|),\\&(R-|x|)|D^2\phi|\leq C\sum_{Z\in \tilde{\Gamma}}(|DD_Z\phi|+|D\phi|),\quad (R-|x|)|D\phi|\leq C|D\phi|,\\&(R-|x|)^2|D^3\phi|\leq C\sum_{X,Z\in \tilde{\Gamma}}(|DD_XD_Z\phi|+|DD_Z\phi|+|D\phi|),\end{align*}
(here we still assume $t=0,$ $R_*\leq |x|< R$) which along with $ \Omega_{ij}\in \tilde{\Gamma}$ gives (for $k=1,2$)\begin{align}\label{eq:F3}
&\sum_{l_1+l_2\leq k}(R-|x|)^{l_1}|D^{l_1+1}D_{\Omega_{ij}}^{l_2}\phi|\leq C\sum_{l\leq k}\sum_{Z\in \tilde{\Gamma}}|DD_Z^l\phi|.\end{align}
Summing up the estimate of $F$ and $ \phi$ (i.e. \eqref{eq:F1}, \eqref{eq:F2}, \eqref{eq:F3}) we have
\begin{align*}
&\sum\limits_{l_1+l_2\leq k}\sum\limits_{i, j}\int_{R_*\leq |x|<R}(|D^{l_1} D_{\Om_{ij}}^{l_2}F|^2+|D^{l_1+1} D_{\Om_{ij}}^{l_2}\phi|^2)(0,x)(R-|x|)^{\ga+2l_1}dx\\ \leq& C\int_{R_*\leq |x|<R}(R-|x|)^{\ga}\sum_{l\leq k}\sum_{Z\in \tilde{\Gamma}}(|\mathcal{L}_Z^lF|^2+|DD_Z^l\phi|^2)(0,x)dx\leq C\mathcal{E}_{k, \ga}^{R*}.
\end{align*}
Then by the definition of $\mathcal{E}_{k, \ga}^{R}$ we have (also using that $ \Omega_{ij}$ is $C^{\infty}$)
\begin{align*}
\mathcal{E}_{k, \ga}^{R}&=\sum\limits_{l_1+l_2\leq k}\sum\limits_{i, j}\int_{\B_{R}}(|D^{l_1} D_{\Om_{ij}}^{l_2}F|^2+|D^{l_1+1} D_{\Om_{ij}}^{l_2}\phi|^2)
(R-|x|)^{\ga+2l_1}dx+\mathcal{E}_{0, \ga}^{R}\\
&\leq\sum\limits_{l_1+l_2\leq k}\sum\limits_{i, j}\int_{R_*\leq |x|<R}(|D^{l_1} D_{\Om_{ij}}^{l_2}F|^2
+|D^{l_1+1} D_{\Om_{ij}}^{l_2}\phi|^2)(0,x)
(R-|x|)^{\ga+2l_1}dx\\
&\quad +\sum\limits_{l_1+l_2\leq k}\sum\limits_{i, j}\int_{ |x|<R_*}(|D^{l_1}
D_{\Om_{ij}}^{l_2}F|^2+|D^{l_1+1} D_{\Om_{ij}}^{l_2}\phi|^2)(0,x)
(R-|x|)^{\ga+2l_1}dx+\mathcal{E}_{0, \ga}^{R}\\
&\leq C\mathcal{E}_{k, \ga}^{R*}+C\int_{ |x|<R_*}\sum_{l\leq k}(|D^lF|^2+|D^{l+1}\phi|^2)(0,x)dx+\mathcal{E}_{0, \ga}^{R}\leq C\mathcal{E}_{k, \ga}^{R*}.
\end{align*}This completes the proof.\end{proof}
Now we use $ \mathcal{E}_{k, \ga}^{R*}$ to denote the energy norms associated to $(\Lambda \phi, \tilde{F})$ on $\mathbf{\Phi}(\mathcal{H})$. Then
(recall that $\mathbf{\Phi}(\mathcal{H}_+)=\{(0, \tilde{x})| R_*\leq |\tilde{x}|<R\}=\B_{R}\setminus\B_{R_*}$, $\mathbf{\Phi}(\mathcal{H}_{-})=\{(0, \tilde{x})| |\tilde{x}|< R_*\}= \B_{R_*}$)
\begin{align*}
\mathcal{E}_{k, \ga}^{R*}&=\int_{\B_{R}\setminus\B_{R_*}}(R-|\tilde{x}|)^{\ga}\sum_{l\leq k}\sum_{Z\in \tilde{\Gamma}}(|\mathcal{L}_Z^l\tilde{F}|^2+|\tilde{D}\tilde{D}_Z^l(\La\phi)|^2)d\tilde{x}\\&+\int_{ \B_{R_*}}\sum_{l\leq k}(|\tilde{D}^l\tilde{F}|^2+|\tilde{D}^{l+1}(\La\phi)|^2)d\tilde{x}+\mathcal{E}_{0, \ga}^{R},
\end{align*}
We prove the following estimate.
\begin{Prop}
  \label{prop:bd:ID:comp}
  Let $\ga=2-\ga_1$. The solution $(\widetilde{F}, \La\phi)$ verifies the following bound on $\mathbf{\Phi}(\mathcal{H})$
  \begin{equation}
    \label{eq:bd:ID:comp}
    \mathcal{E}_{k,\ga}^{R*}\leq C \mathcal{E}_{k, \ga_1}^{\mathcal{H}},\quad k=0, 1,  2
  \end{equation}
  for some constant C depending only on $R$ and $\ga_1$.
\end{Prop}
\begin{proof}
 For simplicity, in the following during the proof, the implicit constant in $A\les B$ relies only on $R$ and $\ga_1$.

We need to estimate the size of the weighted energies $\mathcal{E}_{k, \ga}^{R*}$ defined on $\bf{\Phi}(\mathcal{H})$ in terms of the associated weighted energies on the pre-image $\mathcal{H}$ of the conformal mapping $\bf{\Phi}$. First we recall the properties of the coordinates function on the hyperboloid $\mathcal{H}$.
We have the relation
\begin{align*}
  (t^*-(2R)^{-1})dt=rdr,\quad (t^*)^2-r^2=R^{-1} t^*,
\end{align*}
which in particular implies that
\begin{align*}
d\widetilde{r}=\La^{-2}((t^*)^2+r^2)dr-2 \La^{-2}t^* r dt=\f12 (t^*)^{-1}(t^*-(2R)^{-1})^{-1}dr.
\end{align*}
Since $\widetilde{\om}=\om$, $\widetilde{r}=\Lambda^{-1}r$, the surface measure obeys
\begin{align*}
  d\widetilde{x}=\widetilde{r}^2d\widetilde{r}d\widetilde{\om}=\f12 \widetilde{r}^2 (t^*)^{-1}(t^*-(2R)^{-1})^{-1}drd\om=\f12 \Lambda^{-2}(t^*)^{-1}(t^*-(2R)^{-1})^{-1} dx.
\end{align*}
We also note that
on the hyperboloid $\mathcal{H}$
\[
0\leq t^*-r=\frac{R^{-1}t^*}{t^*+r}\leq R^{-1},\quad  \La=R^{-1}t^*,\quad t^*\geq R^{-1}.
\]
This  means that $t^*$ is close to $r$ and we in particular can conclude that
\begin{align*}
  d\widetilde{x}\les \La^{-4}dx,\quad t^*+r\les \La\les t^*+r.
\end{align*}
Here and in the sequel we always restrict the analysis to the hyperboloid $\mathcal{H}$ or the image $\bf{\Phi}(\mathcal{H})$.

For the zeroth order weighted energy $\mathcal{E}_{0, \ga}^{R*}=\mathcal{E}_{0, \ga}^R$, by definition we can estimate that
\begin{align*}
  &(R-|\widetilde{x}|)^{\ga}(|\widetilde{\a}|^2+|\widetilde{D}_{\widetilde{L}}(\La\phi)|^2)+|\widetilde{\rho}|^2+|\widetilde{\si}|^2+
  |\widetilde{\D}(\Lambda\phi)|^2+|\widetilde{\ab}|^2
  +|\widetilde{D}_{\widetilde{\Lb}}(\La\phi)|^2+|\La \phi|^2\\
  &=(t^*+r)^{4-\ga}(\Lambda^2|\a|^2+|D_{L}(\La\phi)|^2)+\Lambda^4(|\rho|^2+|\si|^2+|\D\phi|^2)\\
  &\quad+(t^*-r)^{4}(\La^2|\ab|^2
  +|D_{\Lb}(\La\phi)|^2)+|\La \phi|^2\\
   &\les \La^{2+\ga_1}(\Lambda^2|\a|^2+|D_{L}(r\phi)|^2+|D_{L}\phi|^2)+\Lambda^4(|\rho|^2+|\si|^2+|\D \phi|^2)+ \La^2(|\ab|^2
  +|D_{\Lb}\phi|^2+ |\phi|^2).
\end{align*}
Here ${ R}-|\widetilde{x}|=\La^{-1}({ R} \La-r)=(t^*+r)^{-1}$.
On the other hand, for the expression of $E[\phi, F](\mathcal{H})$, we have
\begin{align*}
  E[\phi, F](\mathcal{H}) \geq  \int_{\mathcal{H}} & { R^{-4}} \La^{-2}(|\ab |^2+|D_{\Lb} \phi|^2)+|D_L \phi|^2+|\a|^2+|\rho|^2+|\si|^2+|\D \phi|^2
  dx.
\end{align*}
Therefore we can bound that
\begin{align*}
  \mathcal{E}_{0, \ga}^{R*}=\mathcal{E}_{0, \ga}^R &\les \int_{\mathcal{H}}\big(\La^{2+\ga_1}(\Lambda^2|\a|^2+|D_{L}(r\phi)|^2+ |D_{L}\phi|^2)+\Lambda^4(|\rho|^2+|\si|^2+|\D \phi|^2)\\
   &\qquad+ \La^2(|\ab|^2
  +|D_{\Lb}\phi|^2+|\phi|^2)\big) \Lambda^{-4}dx\\
  &\les E[\phi, F](\mathcal{H})+ \int_{\mathcal{H}} \La^{\ga_1-2}(\Lambda^2|\a|^2+|D_{L}(r\phi)|^2)+\La^{-2}| \phi|^2 dx\\
  &\les E[\phi, F](\mathcal{H})+\int_{\mathcal{H}_+} r^{\ga_1}(|\a|^2+r^{-2}|D_{L}(r\phi)|^2)+r^{-2}| \phi|^2 dx\\
  &\les \mathcal{E}_{0, \ga_1}^{\mathcal{H}}.
\end{align*}
Next for the higher order weighted energy $\mathcal{E}_{k, \ga}^{R*}$, $1\leq k\leq 2$. The integral on $\mathbf{\Phi}(\mathcal{H}_{-})= \B_{R_*}$
can be bounded by the corresponding Sobolev norm as $|\tilde{x}|$ is strictly less than $R^*<R$, i.e.\begin{align*}
\int_{ \B_{R_*}}\sum_{l\leq k}(|\tilde{D}^l\tilde{F}|^2+|\tilde{D}^{l+1}(\La\phi)|^2)d\tilde{x}\les\int_{\mathcal{H}_{-}}\sum_{l\leq k}(|{D}^l{F}|^2+|{D}^{l+1}\phi|^2+|\phi|^2)d{x} \les\mathcal{E}_{k, \ga_1}^{\mathcal{H}}.
\end{align*}
Therefore it suffices to estimate the integral on $\mathbf{\Phi}(\mathcal{H}_+)$, by the definition, we need to estimate
$|\mathcal{L}_{Z}^{l}\tilde{F}|$, $|\tilde{D} \tilde{D}_{Z}^{l} (\La\phi)|$
in $(\tilde{t}, \tilde{x})$ coordinates in terms of the associated quantities in $(t, x)$ coordinates for all $l\leq k\leq 2$.

We first estimate the Yang-Mills field. As $ \tilde{\Ga}=\{\Phi_*X:X\in \Gamma\}$, for $Z\in \tilde{\Ga},$ there exists $Z'\in \Ga$ such that $Z=\Phi_*Z'$. Then $\mathcal{L}_{Z}^{l}\tilde{F}=(\Phi^{-1})^*(\mathcal{L}_{Z'}^{l}{F})$ and
\begin{align*}
  &|\mathcal{L}_{Z}^{l}\tilde{F}|^2=|\tilde{\a}(\mathcal{L}_{Z}^{l}\tilde{F})|^2+|\tilde{\rho}(\mathcal{L}_{Z}^{l}\tilde{F})|^2
  +|\tilde{\si}(\mathcal{L}_{Z}^{l}\tilde{F})|^2
  +|\tilde{\ab}(\mathcal{L}_{Z}^{l}\tilde{F})|^2\\
  =&(t^*+r)^{4}\Lambda^2|\a(\mathcal{L}_{Z'}^{l}{F})|^2+\Lambda^4(|\rho(\mathcal{L}_{Z'}^{l}{F})|^2+|\si(\mathcal{L}_{Z'}^{l}{F})|^2)
  +(t^*-r)^{4}\La^2|\ab(\mathcal{L}_{Z'}^{l}{F})|^2\\
  \les&\Lambda^6|\a(\mathcal{L}_{Z'}^{l}{F})|^2+\Lambda^4(|\rho(\mathcal{L}_{Z'}^{l}{F})|^2+|\si(\mathcal{L}_{Z'}^{l}{F})|^2)
  +\La^2|\ab(\mathcal{L}_{Z'}^{l}{F})|^2.
\end{align*}Therefore for the part of Yang-Mills field, for all $l\leq k\leq 2$, we can estimate that (as $\mathbf{\Phi}(\mathcal{H}_+)=\B_{R}\setminus \B_{R_*}$)
\begin{align*}
   &\int_{\B_{R}\setminus \B_{R_*}}\sum\limits_{ Z\in \widetilde{{\Ga}}}|\mathcal{L}_{Z}^{l}\tilde{F}|^2(R-|\tilde{x}|)^{\ga}d\tilde{x}\\
   \les& \sum\limits_{Z\in \Ga}\int_{\mathcal{H}_+} \La^2 \left(\La^4 | \a(\mathcal{L}_{Z}^{l}{F})  |^2 +| \ab(\mathcal{L}_{Z}^{l}{F})  |^2+ \La^2 (| \rho(\mathcal{L}_{Z}^{l}{F})  |^2+| \si(\mathcal{L}_{Z}^{l}{F})  |^2)\right) (t^*+r)^{-\ga}\La^{-4}dx\\
   \les& \sum\limits_{ Z\in \Ga}\int_{\mathcal{H}_+}  (r^{\ga_1}| \a(\mathcal{L}_{Z}^{l}{F})  |^2 +\La^{-2}| \ab(\mathcal{L}_{Z}^{l}{F})  |^2+  | \rho(\mathcal{L}_{Z}^{l}{F})  |^2+| \si(\mathcal{L}_{Z}^{l}{F})  |^2) dx\\
   \les& \mathcal{E}_{l, \ga_1}^{\mathcal{H}}\les \mathcal{E}_{k, \ga_1}^{\mathcal{H}}.
\end{align*}
Next we estimate the part contributed by the scalar field $\tilde{D}\tilde{D}_Z^{l}(\La \phi)$.
As $ \tilde{\Ga}=\{\Phi_*X:X\in \Gamma\}$, we have
\begin{align*}
  \sum\limits_{Z\in \tilde{\Ga}}|\tilde{D}\tilde{D}_Z^{l}(\La\phi)|^2&\les \sum\limits_{Z\in \Ga}(|(t^*+r)^2 D_L D_Z^{l}(\La\phi)|^2+|(t^*-r)^2 D_{\Lb} D_Z^{l}(\La\phi)|^2+|\La \D D_Z^{l}(\La\phi)|^2)\\
  &\les \sum\limits_{j\leq l, Z\in \Ga}(|\La(t^*+r) D_L D_Z^{j}(r\phi)|^2+|\La^2 D_L D_Z^{j}\phi|^2+|\La (t^*-r)^2 D_{\Lb} D_Z^{j} \phi|^2\\
  &\qquad \qquad \qquad +|\La (t^*-r) D_Z^{j} \phi|^2+|\La^2 \D D_Z^{j}\phi|^2).
\end{align*}Therefore for all $l\leq k\leq 2$, we can estimate that (as $\mathbf{\Phi}(\mathcal{H}_+)=\B_{R}\setminus \B_{R_*}$)\begin{align*}
   &\int_{\B_{R}\setminus \B_{R_*}}\sum\limits_{ Z\in \widetilde{{\Ga}}}|\tilde{D}\tilde{D}_Z^{l}(\La\phi)|^2(R-|\tilde{x}|)^{\ga}d\tilde{x}\\
   \les& \sum\limits_{j\leq l,Z\in \Ga}\int_{\mathcal{H}_+}  \La^2\big(\La^2 | D_L D_{Z}^j (r\phi)  |^2 +\La^2 |D_L D_Z^{j}\phi|^2+| D_{\Lb} D_Z^{j} \phi|^2 +|   D_Z^{j} \phi|^2+\La^2| \D D_Z^{j}\phi|^2\big)\\
   &\qquad\qquad\qquad \cdot  (t^*+r)^{-\ga}\La^{-4}dx\\
   \les& \sum\limits_{j\leq l, Z\in \Ga}\int_{\mathcal{H}_+}  [r^{\ga_1-2}| D_L D_{Z}^j(r\phi)  |^2 +\La^{-2}(| D_{\Lb }D_{Z}^j \phi  |^2+| D_{Z}^j \phi  |^2)+  |D_L D_{Z}^j \phi |^2+  |\D D_{Z}^j \phi |^2] dx\\
   \les& \mathcal{E}_{l, \ga_1}^{\mathcal{H}}\les \mathcal{E}_{k, \ga_1}^{\mathcal{H}}.
\end{align*}
We thus complete the proof and hence the Proposition holds.
\end{proof}

Theorem \ref{thm:YMH:hyperboloid} then follows from the above Proposition together with Theorem \ref{thm:EM}.

\textsl{Proof for Theorem \ref{thm:YMH:hyperboloid}:}
For the associated solution $(\widetilde{F}, \La\phi)$ on the compact region $\mathbf{\Phi}(\mathbf{D})$, from Theorem \ref{thm:EM} and the previous Lemma \ref{lem2}, Proposition \ref{prop:bd:ID:comp}, we conclude that
\begin{align*}
  |\La\phi|(\widetilde{t}, \widetilde{x})&\leq C_0 \sqrt{\mathcal{E}_{1, \ga_1}^{\mathcal{H}}} ({ R}-\widetilde{t})^{-\frac{1+2-\ga_1}{2}},\\
  |\widetilde{D}(\La\phi)|(\widetilde{t}, \widetilde{x})+|\widetilde{F}|(\widetilde{t}, \widetilde{x})&\leq C_0 \sqrt{\mathcal{E}_{2, \ga_1}^{\mathcal{H}}} ({ R}-\widetilde{t}-|\widetilde{x}|)^{-\frac{1+2-\ga_1}{2}-\ep}({ R}-\widetilde{t})^{\ep-1},\quad \forall (\widetilde{t}, \widetilde{x})\in\mathbf{\Phi}(\mathbf{D})
\end{align*}
for some constant $C_0$ depending only $\mathcal{E}_{0, \ga_1}^{\mathcal{H}}$, $R$, $\ga_1$ and $\ep$ (for simplicity we may use the same notation $C_0$ to denote such a constant).
Recall that
\[
\widetilde{t}={ R}-\La^{-1}(t+R+1),\quad |\widetilde{x}|=\La^{-1}r, \quad \La =(t+R+1-r)(t+R+1+r).
\]
And in the future of  the hyperboloid $\mathcal{H}$, we have $v_+\les t+R+1$. In particular there exists a constant $ K_0$ depending only on $R$ such that
\begin{align*}
 K_0^{-1}u_+^{-1}&\leq  R-\widetilde{t}=\La^{-1}(t+R+1)\leq K_0 u_+^{-1},\\
 K_0^{-1}v_+^{-1}&\leq  R-\widetilde{t}-|\widetilde{x}|=\La^{-1}(t+R+1-r)\leq K_0 v_+^{-1}.
\end{align*}
Then the above decay estimate for $\La\phi$ implies that
\begin{align*}
  |\phi(t, x)|\leq C_0 \sqrt{\mathcal{E}_{1, \ga_1}^{\mathcal{H}}} \La^{-1} (\La^{-1}(t+R+1))^{-\frac{1+2-\ga_1}{2}}\leq C_0 \sqrt{\mathcal{E}_{1, \ga_1}^{\mathcal{H}}} u_+^{\frac{1-\ga_1}{2}}v_+^{-1}.
\end{align*}
Here as we have pointed out, the two constants $C_0$ may be different.

For the derivative of the scalar field and the Yang-Mills field, first we can write that
\begin{align*}
  |\widetilde{D}(\La\phi)|(\widetilde{t}, \widetilde{x})+|\widetilde{F}|(\widetilde{t}, \widetilde{x})&=(t^*+r)^2|D_L(\La\phi)|+(t^*-r)^2 |D_{\Lb}(\La\phi)|+\La |\D(\La\phi)|+\La (t^*+r)^2 |\a|\\
  &+ \La (t^*-r)^2 |\ab|+\La^2 (|\rho|+|\si|).
\end{align*}
In particular, we conclude that
\begin{align*}
|\D\phi|+|\rho|+|\si|& \leq C_0 \sqrt{\mathcal{E}_{2, \ga_1}^{\mathcal{H}}} \La^{-2} ({ R}-\widetilde{t}-|\widetilde{x}|)^{-\frac{1+2-\ga_1}{2}-\ep}({ R}-\widetilde{t})^{\ep-1}\\
&\leq C_0 \sqrt{\mathcal{E}_{2, \ga_1}^{\mathcal{H}}} u_+^{-1-\ep}v_+^{-\frac{1+\ga_1}{2}+\ep}.
\end{align*}
Similarly we have
\begin{align*}
&|\a|\leq C_0 \sqrt{\mathcal{E}_{2, \ga_1}^{\mathcal{H}}} u_+^{-\ep}v_+^{-\frac{3+\ga_1}{2}+\ep},\\
&|\ab|\leq C_0 \sqrt{\mathcal{E}_{2, \ga_1}^{\mathcal{H}}} u_+^{-\ep-2}v_+^{-\frac{\ga_1-1}{2}+\ep}.
\end{align*}
Same argument also applies to obtain bounds for $D_{L}(\La\phi)$, $D_{\Lb}(\La\phi)$, which together with the estimate for $\phi$ lead to the desired estimates for $D_{L}\phi$ and $D_{\Lb}\phi$. Indeed by using the above bound for $\phi$, we first control that
\[
\La|D_{\Lb}\phi|\leq |D_{\Lb}(\La\phi)|+ v_+|\phi|\leq  |D_{\Lb}(\La\phi)|+ { C_0 \sqrt{\mathcal{E}_{2, \ga_1}^{\mathcal{H}}}}u_+^{\frac{1-\ga_1}{2}}.
\]
We therefore can show that
\begin{align*}
|D_{\Lb}\phi|\leq \La^{-1}|D_{\Lb}(\La\phi)|+{ C_0 \sqrt{\mathcal{E}_{2, \ga_1}^{\mathcal{H}}}}\La^{-1}u_+^{\frac{1-\ga_1}{2}}\leq C_0 \sqrt{\mathcal{E}_{2, \ga_1}^{\mathcal{H}}} u_+^{-\ep-2}v_+^{-\frac{\ga_1-1}{2}+\ep}.
\end{align*}
Similarly we have
\begin{align*}
|D_{L}\phi|\leq \La^{-1}|D_{ L}(\La\phi)|+{ C_0 \sqrt{\mathcal{E}_{2, \ga_1}^{\mathcal{H}}}}\La^{-1}u_+^{\frac{1-\ga_1}{2}-1}v_+^{-1}\leq C_0 \sqrt{\mathcal{E}_{2, \ga_1}^{\mathcal{H}}} u_+^{\frac{1-\ga_1}{2}}v_+^{-2}.
\end{align*}
As $\ep<\frac{1-\ga_1}{2}$, the above estimate for $D_L\phi$ is weaker than that of $\a$. However we can use this weaker bound to obtain a stronger estimate for $D_{L}(r\phi)$. In fact, we can show that
\begin{align*}
|D_{L}(2r\phi)|&\leq |D_{L}((t^*+r)\phi)|+|D_{L}((t^*-r)\phi)|\\
&\leq u_+^{-1}|D_L(\La\phi)|+u_+ |D_L\phi|\\
&\leq C_0 \sqrt{\mathcal{E}_{2, \ga_1}^{\mathcal{H}}} u_+^{-\ep}v_+^{-\frac{1+\ga_1}{2}+\ep}.
\end{align*}
We thus complete the proof for Theorem \ref{thm:YMH:hyperboloid}.

\end{document}